\documentclass{amsart}
\usepackage{ben}

\newcommand{\phcl}[1]{\tilp_{#1}(\beta)}

\title[Cauchy problem of Lions type with rough data]{On well-posedness for parabolic Cauchy problems of Lions type with rough initial data}

\author{Pascal Auscher}
\address{Universit{\'e} Paris-Saclay, CNRS, Laboratoire de Math\'{e}matiques d'Orsay, 91405 Orsay, France}
\email{pascal.auscher@universite-paris-saclay.fr}

\author{Hedong Hou}
\address{Universit{\'e} Paris-Saclay, CNRS, Laboratoire de Math\'{e}matiques d'Orsay, 91405 Orsay, France}
\email{hedong.hou@universite-paris-saclay.fr}

\date{June 22, 2024}

\begin{document}

\begin{abstract}
    We establish a complete picture for well-posedness of parabolic Cauchy problems with time-independent, uniformly elliptic, bounded measurable complex coefficients. We exhibit a range of $p$ for which tempered distributions in homogeneous Hardy--Sobolev spaces $\DotH^{s,p}$ with regularity index $s \in (-1,1)$ are initial data. Source terms of Lions' type lie in weighted tent spaces, and weak solutions are built with their gradients in weighted tent spaces as well. A similar result can be achieved for initial data in homogeneous Besov spaces $\DotB^{s}_{p,p}$.
\end{abstract}
\maketitle
\tableofcontents
\section{Introduction}
\label{sec:intro}

The goal of this work is to present a complete picture for existence and uniqueness of weak solutions (see Definition \ref{def:weaksol}) to the Cauchy problem
\begin{equation}
    \label{e:ivp_ic}
    \begin{cases}
    \partial_t u - \Div_x A(x) \nabla_x u = f + \Div_x F, \quad (t,x) \in (0,\infty) \times \bR^n \\ 
    u(0)=u_0
    \end{cases},
\end{equation}
using tent spaces introduced by \cite{Coifman-Meyer-Stein1985_TentSpaces} and their weighted generalizations (see Section \ref{ssec:tent-slice} for definition). We assume the coefficient matrix $A \in L^\infty(\bR^n;\mat_n(\bC))$ is \textit{uniformly elliptic}, that is, there are $\Lambda_0,\Lambda_1>0$ so that for a.e. $x \in \bR^n$ and any $\xi,\eta \in \bC^n$,
\begin{equation}
    \label{e:uniformly_elliptic}
    \Re(\langle A(x)\xi,\xi \rangle) \geq \Lambda_0 |\xi|^2, \quad |\langle A(x)\xi,\eta \rangle| \leq \Lambda_1 |\xi| |\eta|.
\end{equation}
By the initial condition $u(0)=u_0$, we require that $u(t)$ tends to $u_0$ as $t \to 0$ in distributional sense, \textit{i.e.}, in $\scrD'(\bR^n)$.

We refer the reader to the introduction of \cite{Auscher-Monniaux-Portal2019Lp} for motivation to include tools from harmonic analysis such as tent spaces when the initial data are in $L^p(\bR^n)$ and the coefficients are not regular. There, well-posedness is established for weak solutions $u$ to the equation without source terms (even in the non-autonomous case) in a class defined by the non-tangential maximal Kenig--Pipher norm, as well as estimates of $\nabla u$ in tent spaces $T^p_0$ (with limited but explicit range of $p$). Uniqueness in the latter class is obtained in \cite{Zaton2020wp}. It is worth pointing out that $L^p$ is not a trace space, with trace defined in relation to (real) interpolation theory, in contrast to (homogeneous) Besov spaces. So already, the above results do not belong to the realm of maximal regularity using the notion of mild solutions.

Source terms in tent spaces have been considered in nonlinear PDE's, see \cite{Koch-Tataru2001_NSBMO-1,Auscher-Frey2017_NS-KT,Danchin-Vasilyev2023_InhomoNS-tent}, and in deterministic estimates used for stochastic PDE's, see \cite{Auscher-vanNeerven-Portal2014_SPDE-tent,Portal-Veraar2019_SPDE-Lp}. More recently, \cite{Auscher-Portal2023Lions} shows existence of weak solutions with $\nabla u \in T^p_0$ when $F \in T^p_0$, $f=0$ and $u_{0}=0$. For source terms $f$ (and $F=0$), our prior work \cite{Auscher-Hou2023_SIOTent} obtains  well-posedness and maximal regularity of $u$ itself in a weighted tent space, the weight being an expression of regularity (see below). In this case, however, the initial data must be zero.

A subsequent forthcoming work of the second author will treat the non-autono\-mous case.

\subsection{Main results}
\label{ssec:intro_results}

In this paper, we treat initial data in spaces which are not necessarily trace spaces and carry regularity information by control of the gradient of weak solutions in a \emph{weighted} tent space, where the weight is a power of the time variable. Indeed, we shall show that our approach gives access to initial data in homogeneous Hardy--Sobolev spaces defined by Littlewood--Paley decomposition (see Section \ref{ssec:HS}). The regularity range is between $-1$ and 1.  Our method also applies for initial data in homogeneous Besov spaces.

The prototypical example illustrating this is the heat equation.

\subsubsection{Heat equation}
\label{sssec:intro_heat}
Our first result gives a flavor for the homogeneous Cauchy problem (By this, we mean having an initial data but  null source terms), and we also provide representation of heat solutions by studying their traces.

\begin{theorem}[Heat equation and weighted tent spaces]
    \label{thm:heat_Hsp}
    Let $s \in \bR$, $0<p \le \infty$, and $g \in \scrS'(\bR^n)$.

    \begin{enumerate}[label=\normalfont(\roman*)]
        \item \label{item:heat_tent_est} {\normalfont (Weighted tent-space estimates)}
        Suppose $s<1$. If $g \in \DotH^{s,p}$, then the function $(t,x) \mapsto \nabla e^{t\Delta} g(x)$ belongs to $T^p_{s/2}$ with
        \[ \|\nabla e^{t\Delta} g\|_{T^p_{s/2}} \eqsim \|g \|_{\DotH^{s,p}}. \] 

        \item \label{item:heat_rep} {\normalfont (Representation of heat solutions)} 
        Let $u$ be a distributional solution to the heat equation on $\bR^{1+n}_+$ with $\nabla u \in T^p_{{s}/{2}}$. Suppose $s>-1$ and $\frac{n}{n+s+1} \le p \le \infty$. Then there exists a unique $u_0 \in \scrS'$ so that $u(t)=e^{t\Delta} u_0$ for all $t>0$. Moreover,
        \begin{enumerate}[label=\normalfont(\arabic*),ref=(\arabic*)]
            \item \label{item:heat_rep_s>1}
            If $s \ge 1$ and $\frac{n}{n+s-1} \le p \le \infty$, then $u$ is a constant.
            
            \item \label{item:heat_rep_s<1}
            If $-1<s<1$ and $\frac{n}{n+s+1} \le p \le \infty$, then there exist $g \in \DotH^{s,p}$ and $c \in \bC$ such that $u_{0}=g+c$, so $u(t)=e^{t\Delta} g+c$ for all $t>0$.
    \end{enumerate}
    \end{enumerate}
\end{theorem}

The first point \ref{item:heat_tent_est} shows a precise correspondence between Sobolev regularity of the initial data and the choice of the power weight. The second point \ref{item:heat_rep} shows the regularity range $s<1$ is essentially sharp.  

When $s=0$, the first point is reminiscent to the work of Fefferman--Stein \cite{Fefferman-Stein1972Hp} on defining Hardy spaces by so-called \textit{conical square functions} when translated to parabolic setting. But their point was to take extensions  not related to any kind of equations to obtain an intrinsic definition of Hardy spaces. Here, we have to restrict ourselves to extensions related to the heat equation. It is also reminiscent to Littlewood--Paley theory using rather vertical square functions. The reader can refer to \cite[\S 4.1]{Triebel2020_IV} for more on this topic.

Combining these two points yields in particular well-posedness of the homogeneous Cauchy problem in the following strong form.

\begin{cor}[Isomorphism]
    \label{cor:bijection} 
    Let $-1<s<1$ and $\frac{n}{n+s+1} \le p \le \infty$. The map $g \mapsto (e^{t\Delta} g)_{t>0}$ is a bijection from $\DotH^{s,p} + \bC$ onto the space of distributional solutions $u$ to the heat equation with $\nabla u \in  T^p_{{s}/{2}}$, and 
    \[ \|g\|_{\DotH^{s,p}} \eqsim \|\nabla u \|_{T^p_{{s}/{2}}}. \]
\end{cor}

\subsubsection{Critical exponents for the parabolic problem}
\label{sssec:criticalexponents}
To treat the more general parabolic problems  \eqref{e:ivp_ic}, we introduce the operator $L$ 
\[ L:=-\Div (A\nabla) \]
on $L^2(\bR^n)$ with domain
\[ D(L):=\{f \in W^{1,2}(\bR^n):\Div(A\nabla f) \in L^2(\bR^n)\}. \]
By the theory of maximal-accretive operators, $-L$ generates a bounded analytic semigroup $(e^{-tL})_{t \ge 0}$ on $L^2(\bR^n)$. Our construction of  weak solutions is based on an appropriate extension of the solution map $\cE_L$, which is initially defined from $L^2(\bR^n)$ to $L^2_{\loc}((0,\infty);W^{1,2}_{\loc}(\bR^n))$ by the semigroup as
\begin{equation}
    \label{e:sol_map}
    \cE_L(u_0)(t,x) := (e^{-tL} u_0)(x).
\end{equation}

The $L^p$-theory of the semigroup is ruled by four critical numbers, which are introduced in \cite[Proposition 3.15]{Auscher2007Memoire} for $1<p<\infty$, and later extended to  $p>\frac{n}{n+1}$ in \cite[\S 6]{Auscher-Egert2023OpAdp} to include Hardy spaces $H^p(\bR^n)$. These numbers are 
\begin{itemize}
    \item $p_\pm(L) \in [\frac{n}{n+1},\infty]$ such that $(p_-(L),p_+(L))$ is the largest open set (an interval) of exponents $p$ for which the semigroup $(e^{-tL})_{t \ge 0}$ is uniformly bounded on $L^p$ when $p>1$ and on  $H^p$ when $p \le 1$;
    \item $q_\pm(L) \in [\frac{n}{n+1},\infty]$ such that $(q_-(L),q_+(L))$ is the largest open set (an interval) of exponents $p$ for which  the family $(t^{1/2} \nabla e^{-tL})_{t>0}$ is uniformly bounded on $L^p$ when $p>1$ and on  $H^p$ when $p \le 1$.
    \end{itemize}
It is known that $p_-(L)=q_-(L)<\frac{2n}{n+2}$, $q_+(L)>2$, and $p_+(L) \ge \frac{nq_+(L)}{n-q_{+}(L)}$ (by convention $p_+(L)=\infty$ if $q_{+}(L)\ge n$). The strict inequalities are best possible. We also have the duality relation  $p_{+}(L^*)=\max\{p_{-}(L), 1\}'$, where $p'$ denotes the H\"older conjugate of $p \in [1,\infty]$. For the negative Laplacian, $p_-(-\Delta) = q_-(-\Delta)=\frac{n}{n+1}$ and $p_+(-\Delta)=q_+(-\Delta)=\infty$.   

We extend the critical numbers to our $\DotH^{s,p}$-setting. For $-1 \le s \le 1$, define $p_\pm(s,L)$ as
\begin{equation}
    \label{e:p-(s,L)}
    \frac{1}{p_-(s,L)} := 
    \begin{cases}
    \frac{1}{p_-(L)}+\frac{s}{n} & \text{ if } 0 \le s \le 1 \\ 
    \frac{1+s}{p_-(L)} - \frac{s}{q_+(L^\ast)'} & \text{ if } -1 \le s \le 0
    \end{cases},
\end{equation}
and 
\begin{equation}
    \label{e:p+(s,L)} 
    p_{+}(s,L) := \max\{p_{-}(-s, L^*),1\}'.
    \end{equation}
Notice that $p_\pm(0,L)=p_\pm(L)$, $p_{-}(-1,L)=q_+(L^\ast)'\in [1,2)$, and  that $p_+(1,L)= q_+(L) \in (2,\infty]$. In particular,  $p_-(s,-\Delta)= \frac{n}{n+s+1}$ and $p_{+}(s,-\Delta)=\infty$.

We also introduce several other numbers which will parametrize our results. For convenience, we use a parameter $\beta$ whose relation to the regularity exponent $s$ is given by
\[ \boxed{s=2\beta+1} \]
with $\beta>-1$ and no upper restriction. Define
\begin{equation}
    \label{e:beta_0}
    \beta(L) := -\frac{1}{2}-\frac{n}{2} \left ( \frac{1}{p_-(L)}-1 \right ) \ge -1.
\end{equation}
Note that $\beta(L) \ge -1/2$ if and only if $p_-(L) \ge 1$. For $\beta>-1$, we introduce the numbers $p_L(\beta) \in (0,2)$ given by
\begin{equation}
    \label{e:pLbeta}
    p_{L}(\beta) := \frac{np_-(L)}{n+(2\beta+1)p_-(L)},
\end{equation}
and also $\phcl{L} \in (0,2)$ so that:
\begin{enumerate}[label=(\roman*)]
    \item \label{item:def_phc_p-(L)>1}
    When $p_-(L) \ge 1$, it is given by
    \begin{equation}
        \label{e:phc-p-(L)>1}
        \phcl{L} := 
        \begin{cases}
        p_L(\beta)  & \text{ if } \beta \ge -1/2 \\ 
        p_-(2\beta+1,L) & \text{ if } -1<\beta<-1/2
        \end{cases}.
    \end{equation}

    \item \label{item:def_phc_p-(L)<1}
    When $p_-(L)<1$, it is given by
    \begin{equation}
        \label{e:phc-p-(L)<1}
        \phcl{L} := 
        \begin{cases}
        p_L(\beta)  & \text{ if } \beta \ge \beta (L) \\ 
        \frac{(\beta(L)+1)q_+(L^\ast)}{(\beta(L)+1)q_+(L^\ast)+\beta-\beta(L)} & \text{ if } -1<\beta<\beta (L)
        \end{cases}.
    \end{equation}
\end{enumerate}
Remark that $\tilp_L(\beta(L))=1$ and $\tilp_L(-1)=q_+(L^\ast)'=p_-(-1,L)$. For the negative Laplacian, as $p_-(-\Delta)=\frac{n}{n+1}$, we have
\[ \phcl{-\Delta} = p_{-\Delta}(\beta) = \frac{n}{n+2\beta+2}, \quad \forall \beta>-1. \]
Observe that for any $L$,
\[ \phcl{L} \ge p_L(\beta) \ge \frac{n}{n+2\beta+2}. \]

To illustrate these exponents, we give graphic representations in Figure \ref{fig:exponents}, distinguishing the two cases $p_-(L) \ge 1$ and $p_-(L)<1$. In these figures, we write $p$ for $1/p$ to ease the presentation. When $p<2$, we use red color for the graph of $p_-(2\beta+1,L)$, blue for that of  $\phcl{L}$, and orange for that of $p_{L}(\beta)$. The orange shaded trapezoids are the regions of well-posedness for $\DotH^{2\beta+1,p}$-initial data, while the gray shaded is for constant initial data. Parallel lines to the orange one are lines of embedding for Hardy-Sobolev spaces and weighted tent spaces going downward. 
 
Interestingly, the (smaller) set delimited by the red lines (excluded) and the black lines (included) has a special signification and is called the \textit{identification range} for the operator $L$, and $\cE_L$ defined in \eqref{e:sol_map} can be extended to a semigroup on $\DotH^{s,p}$. The triangular region between the red and blue segments in the case $p_-(L)<1$ is a region of embedding of an adapted space into $\DotH^{s,p}$ without knowing identification. See Section~\ref{sec:hc} and in particular Proposition~\ref{prop:Ladp-id}.
\footnote{In the case where $p>2$, the red colored broken lines for the values of $p_+(2\beta+1,L)$ defined by \eqref{e:p+(s,L)} are obtained by symmetry about the point $(1/p,\beta)=(1/2,-1/2)$ from the ones for $\max\{p_{-}(-2\beta-1, L^\ast),1\}$. Thus it depends on the value of $p_{-}(L^\ast)$. In the first (resp.~second)  picture, this corresponds to having $p_{-}(L^\ast)\ge 1$ (resp.~$p_{-}(L^\ast)<1$). As the values of $p_{-}(L)$ and $p_{-}(L^\ast)$ are independent, we should have in fact 4 such figures.}

\begin{figure}
    \centering
    \begin{minipage}{0.54\textwidth}
        \centering
        \includegraphics[width=\textwidth]{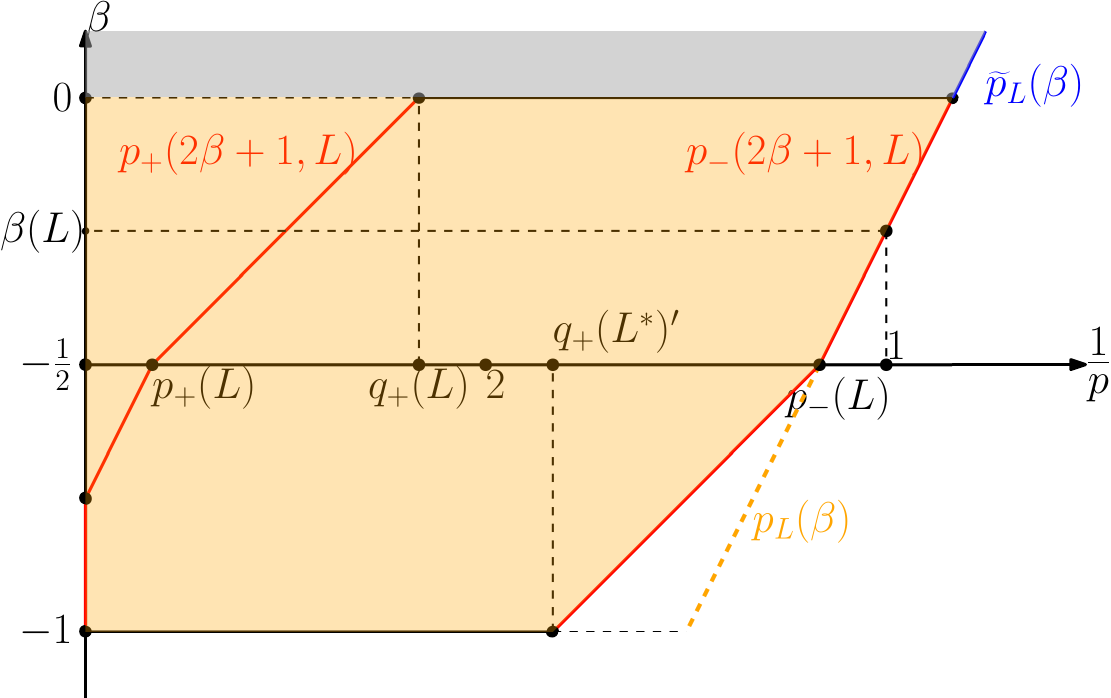}
        Case $p_-(L) \ge 1$
    \end{minipage}
    \hfill
    \begin{minipage}{0.45\textwidth}
        \centering
        \includegraphics[width=\textwidth]{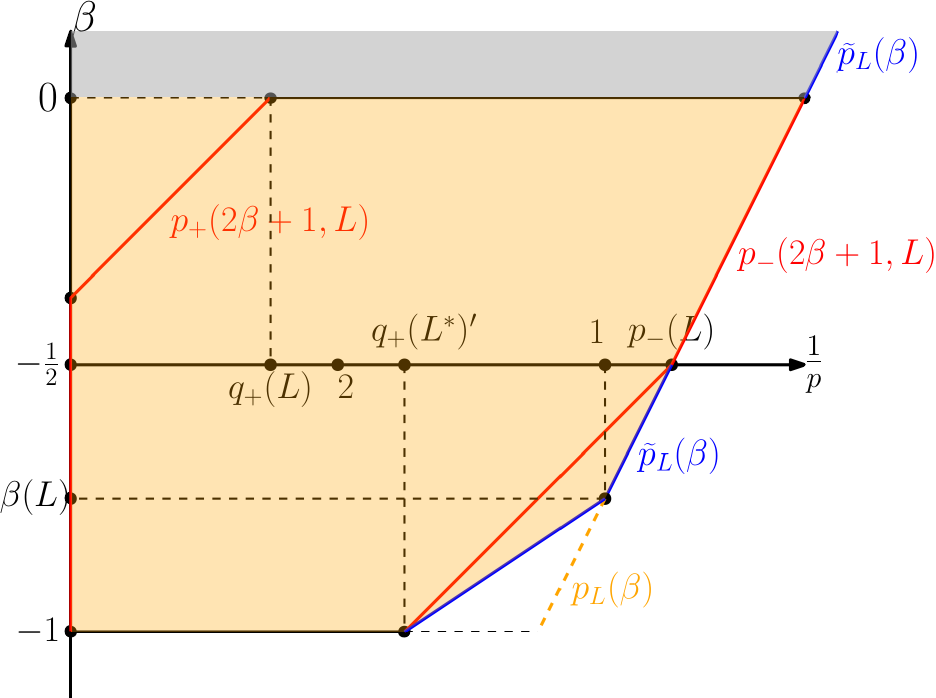}
        Case $p_-(L) < 1$
    \end{minipage}
    \caption{Regions of well-posedness}
    \label{fig:exponents}
\end{figure}

\subsubsection{Parabolic equations of type \eqref{e:ivp_ic}} 
\label{sssec:intromainresults}
We begin with the homogeneous Cauchy problem ($F=0$, $f=0$).

\begin{theorem}[Well-posedness of homogeneous Cauchy problem]
    \label{thm:wp_hc}
    Let $-1<\beta<0$ and $\phcl{L} <p \le \infty$. For any $v_0 \in \DotH^{2\beta+1,p}$, there is a unique global weak solution $v$ to the homogeneous Cauchy problem
    \begin{equation}
        \begin{cases}
        \partial_t v-\Div(A\nabla v) = 0, \quad (t,x) \in (0,\infty) \times \bR^n \\ 
        v(0)=v_0
        \end{cases},
        \tag{HC}
        \label{e:hc}
    \end{equation}
    so that $\nabla v \in T^p_{\beta+{1}/{2}}$,  and one has
    \[ \|\nabla v\|_{T^p_{\beta+{1}/{2}}}\eqsim \|v_{0}\|_{\DotH^{2\beta+1,p}}.
    \]
    Moreover, $v$ belongs to $C([0,\infty);\scrS')$. For $p_-(2\beta+1,L)<p<p_+(2\beta+1,L)$, it further holds that  $v \in C_0([0,\infty);\DotH^{2\beta+1,p}) \cap C^\infty((0,\infty);\DotH^{2\beta+1,p})$.
    \footnote{Here, $C_0([0,\infty);E)$ is the space of continuous functions with limit  0 as $t \to \infty$ in the prescribed topology of $E$.}
\end{theorem}

For $2\beta+1=0$, \cite{Auscher-Monniaux-Portal2019Lp} treats this problem in the non-autonomous case. In the autonomous case here, we obtain a larger range. 

Also, the last point gives the range in which $\DotH^{2\beta+1,p}$-regularity is preserved. In this range with $p \ge 1$ and $2\beta<n/p$,
\footnote{When $2\beta-n/p \ge 0$, our space $\DotH^{2\beta+1,p}$ is not a Banach space. It is semi-normed.}
it also basically says that $v$ is a strong solution of the abstract Cauchy problem $\partial_{t}v+L_{\beta,p}v=0, v(0)=v_{0}$ in the Banach space $\DotH^{2\beta+1,p}$ (where $-L_{\beta,p}$ is the generator of the extended semigroup on $\DotH^{2\beta+1,p}$), see \cite[Definition 4.1.1]{Lunardi1995_semigp}. Note that $p<1$ is possible here, even though there is no abstract theory for strong solutions in quasi-Banach spaces as far as we understand.

The second result is a converse statement. 

\begin{theorem}[Representation]
    \label{thm:rep}
    Let $\beta>-1$ and $\phcl{L}< p \le \infty$. Let $u$ be a weak solution to the equation
    \[ \partial_t u - \Div (A\nabla u) = 0, \quad (t,x) \in (0,\infty) \times \bR^n, \]
    with $\nabla u \in T^p_{\beta+{1}/{2}}$. Then $u$ has a trace $u_0 \in \scrS'$ in the sense that $u(t)$ converges to $u_0$ in $\scrS'$ as $t\to 0$. Moreover,
    \begin{enumerate}[label=\normalfont(\roman*)]
        \item 
        \label{item:rep_beta>0}
        If $\beta \ge 0$ and $\frac{n}{n+2\beta} \le p \le \infty$, then $u$ is a constant.
        
        \item 
        \label{item:rep_-1<beta<0}
        If $-1<\beta<0$, then there exist $g \in \DotH^{2\beta+1,p}$ and $c \in \bC$ so that $u_0=g+c$ and  $u=\cE_L(g)+c$, where $\cE_L$ is an appropriate extension of the semigroup solution map given by \eqref{e:sol_map}.
    \end{enumerate}
\end{theorem}

As we shall see in Proposition \ref{prop:trace}, existence of the trace $u_{0}$ only requires $\phcl{-\Delta}=\frac{n}{n+2\beta+2}< p \le \infty$. The condition $p>\phcl{L}$ is used to prove the rest. Also note that in \ref{item:rep_-1<beta<0}, the extension $\cE_L$ may not be a semigroup if $p$ does not lie in the interval $(p_{-}(2\beta+1,L), p_{+}(2\beta+1,L))$.
 
Theorems \ref{thm:wp_hc} and \ref{thm:rep} also go beyond well-posedness of the homogeneous Cauchy problem. For convenience, define $\cE_L(1)=1$ (constant functions). 

\begin{cor}[Isomorphism for parabolic equations]
    \label{cor:bijectionL} 
    Let $-1<\beta<0$ and $\phcl{L}< p \le \infty$. The map $g \mapsto \cE_L(g)=u$ is a bijection from $\DotH^{2\beta+1,p} + \bC$ onto the space of global weak solutions to $\partial_t u - \Div (A\nabla u) = 0$ with $\nabla u \in  T^p_{\beta+1/{2}}$, and 
    \[ \|g\|_{\DotH^{2\beta+1,p}} \eqsim \|\nabla u \|_{T^p_{\beta+1/{2}}}. \]
\end{cor}
 
Next, we consider the inhomogeneous Cauchy problem (which means null initial data) in the spirit of Lions' theory with source terms of divergence form (that is, $f=0$), for which the weak solutions will be constructed by Duhamel's formula. 

\begin{theorem}[Well-posedness of Lions' equation]
    \label{thm:wp_lions}
    Let $\beta>-1$ and $\phcl{L}<p \le \infty$. Let $F$ be in $T^p_{\beta+{1}/{2}}$. There exists a unique global weak solution $u$ to the Cauchy problem (which we call {\normalfont Lions' equation})
    \[ \begin{cases}
        \partial_t u-\Div(A\nabla u) = \Div F, \quad (t,x) \in (0,\infty) \times \bR^n \\ 
        u(0)=0
    \end{cases}, \]
    with $\nabla u \in T^p_{\beta+{1}/{2}}$. Moreover, $u$ belongs to $T^p_{\beta+1} \cap C([0,\infty);\scrS')$.
\end{theorem}

To finish this introduction, let us state the announced complete picture for \eqref{e:ivp_ic}. For $\gamma>-1/2$, define
\begin{equation}
    \label{e:pL^flat(beta)}
     p_L^{\,\flat}(\gamma):=\frac{n\max\{p_-(L),1\}}{n+(2\gamma+1)\max\{p_-(L),1\}}.
\end{equation}

\begin{theorem}[Well-posedness of Cauchy problems of type \eqref{e:ivp_ic}]
    \label{thm:ic-hc-lions}
    Let $\beta>-1$ and $\phcl{L}<p \le \infty$. Let $\gamma>-1/2$ and $ p_L^{\,\flat}(\gamma)<q \le \infty$. Suppose 
    \begin{equation}
        \label{e:ic-hcl-embed-cond}
        \gamma \ge \beta \quad \text{and}\quad 2\beta - \frac{n}{p} = 2\gamma - \frac{n}{q}.
    \end{equation}
    \begin{enumerate}[label=\normalfont(\roman*)]
        \item \label{item:ichcl-beta<0}
        If $\beta<0$, then for any $u_0 \in \DotH^{2\beta+1,p}$, $F \in T^p_{\beta+{1}/{2}}$, and $f \in T^q_\gamma$, there exists a unique global weak solution $u$ to the Cauchy problem
        \[ \begin{cases}
        \partial_t u-\Div(A\nabla u) = f+\Div F, \quad (t,x) \in (0,\infty) \times \bR^n \\ 
        u(0)=u_0
        \end{cases}, \]
    with $\nabla u \in T^p_{\beta+{1}/{2}}$, and one has the estimate    
    \begin{equation}
        \label{e:globalestimate}
        \|\nabla u\|_{T^p_{\beta+{1}/{2}}} \lesssim \|u_{0}\|_{\DotH^{2\beta+1,p}}+\|F\|_{T^p_{\beta+{1}/{2}}} +  \|f\|_{T^q_{\gamma}}. 
    \end{equation}
    Moreover, $u$ belongs to $C([0,\infty);\scrS')$ and $u-\cE_{L}(u_{0})\in T^p_{\beta+{1}}$.
    
    \item \label{item:ichcl-beta>0}
    If $\beta \ge 0$, then the same statement holds when $u_{0}$ is constant (for which $\|u_{0}\|_{\DotH^{2\beta+1,p}}=0$).
    \end{enumerate}
\end{theorem}

The proof is a simple combination of Theorems \ref{thm:wp_hc}, \ref{thm:wp_lions}, and \cite[Theorem 1.2]{Auscher-Hou2023_SIOTent}, where we showed that for source term $f \in T^q_\gamma$ and $\gamma>-1/2$, the corresponding solution $u_1$ satisfies $\nabla u_1 \in T^q_{\gamma+1/2}$.  The constraint $\gamma>-1/2$ is essentially sharp. So the conditions on $(\beta,p)$ and $(\gamma,q)$ in \eqref{e:ic-hcl-embed-cond} allow one to use embedding between weighted tent spaces (see Section \ref{ssec:tent-slice}) to conclude that $\nabla u_1$ belongs to $T^p_{\beta+1/2}$. The continuity in $\scrS'$ is not proved there for $u_{1}$ but it can be shown analogously to the proof of Theorem \ref{thm:lions_ext} \ref{item:lions_cont}.

Moreover, when $\beta \ge 0$, Theorem \ref{thm:rep} \ref{item:rep_-1<beta<0} implies that constant initial data are the only ones compatible with the solution class $\nabla u \in T^p_{\beta+1/2}$.

\subsection{Organization}
\label{ssec:organization}
The paper is organized as follows. 

Section \ref{sec:spaces} summarizes basic properties of the function spaces to be used, and in particular provides the construction of our homogeneous Hardy--Sobolev spaces $\DotH^{s,p}$ in detail.

Section \ref{sec:heat} is concerned with the heat equation. We provide the tent-space estimates of heat extensions and their gradients, and also study the representation of heat (distributional) solutions. The combination of these two parts implies Theorem \ref{thm:heat_Hsp}.

Sections \ref{sec:lions}, \ref{sec:cont} (resp.~\ref{sec:hc}) are devoted to proving existence of the weak solutions asserted in Theorem~\ref{thm:wp_lions} (resp. Theorem~\ref{thm:wp_hc}), see Theorem~\ref{thm:lions_ext}  (resp.~Theorem~\ref{thm:hc_ext_EL}). Uniqueness is established in Theorem~\ref{thm:unique}.
Section~\ref{sec:ur} also presents the proof of Theorem~\ref{thm:rep}.

Section~\ref{sec:besov} contains extension to initial data in homogeneous Besov spaces, and Section \ref{sec:beta=-1} discusses the homogeneous Cauchy problem at the endpoint for regularity index $\beta=-1$.

\subsection{Notation}
\label{ssec:notation}
Throughout the paper, for any $q,r \in (0,\infty]$, we write 
\[ [q,r]:=\frac{1}{q}-\frac{1}{r}, \]
if there is no confusion with closed intervals. We say $X \lesssim Y$ (or $X \lesssim_A Y$, resp.) if $X \le CY$ with an irrelevant constant $C$ (or depending on $A$, resp.), and say $X \eqsim Y$ if $X \lesssim Y$ and $Y \lesssim X$. 

Write $\bR^{1+n}_+:=\bR_+ \times \bR^n = (0,\infty) \times \bR^n$. For any (Euclidean) ball $B \subset \bR^n$, write $r(B)$ for the radius of $B$. For any function $f$ defined on $\bR^{1+n}_+$, denote by $f(t)$ the function $x \mapsto f(t,x)$ for any $t>0$.

Let $(X,\mu)$ be a measure space. For any measurable subset $E \subset X$ with finite positive measure and $f \in L^1(E,\mu)$, we write
\[ \fint_E f d\mu := \frac{1}{\mu(E)} \int_E f d\mu. \]
Write $\|\cdot\|_p$ as an abbreviation for the norm $\|\cdot\|_{L^p(X,\mu)}$. For any $\beta \in \bR$ and $E \subset \bR^{1+n}$, denote by $L^p_\beta(E)$ the space $L^p(E,t^{-p\beta} dtdy)$.

We use the sans-serif font $\ssfc$ in the scripts of function spaces in short of ``with compact support" in the prescribed set, and ${\loc}$ if the prescribed property holds on all compact subsets of the prescribed set. Often, we omit the domain of the function space if it is clear from the context.
\section{Function spaces}
\label{sec:spaces}

We begin with introducing the spaces for the initial data and next, the spaces for the solutions. Denote by $\scrS$ the space of Schwartz functions on $\bR^n$ and by $\scrS'$ the space of tempered distributions. For any $f \in \scrS'$, write $\hatf$ or $\cF(f)$ for the Fourier transform of $f$.

\subsection{Homogeneous Hardy--Sobolev spaces}
\label{ssec:HS}
Our homogeneous Hardy--Sobolev spaces can be roughly regarded as  ``realizations" of Triebel--Lizorkin spaces in $\scrS'$. We follow the formulation in \cite[\S 2.4.3]{Sawano2018_Besov}, which was originated from an observation of J. Peetre \cite[p.56]{Peetre1976_Besov}, see also \cite{Bourdaud2013_Realization,Moussai2015_Realization_p<1}. 

Denote by $C$ the annulus $\{\xi \in \bR^n:2^{-1} \le |\xi| \le 2^2\}$ and by $2^j C$ the annulus $\{\xi \in \bR^n:2^{j-1} \le |\xi| \le 2^{j+2}\}$ for any $j \in \bZ$. Let $\chi$ be in $\Cc(\bR^n)$ so that $\supp(\chi) \subset C$ and for any $\xi \ne 0$,
\[ \sum_{j \in \bZ} \chi(2^{-j} \xi)=1. \]
Let $\Delta_j$ be the $j$-th \textit{Littlewood--Paley operator} associated with $\chi$, given by
\begin{equation}
    \label{e:L-P-op}
    \Delta_j f := \cF^{-1} (\chi(2^{-j} \cdot) \cF(f)), \quad \forall f \in \scrS'.
\end{equation}
Suppose $s \in \bR$ and $0<p \le \infty$. Let $\frakS_{s,p}$ be the collection of sequences of measu\-rable functions $(f_j)_{j \in \bZ}$ on $\bR^n$  so that
\begin{equation}
    \label{e:frakS_s,p}
    \|(f_j)\|_{\frakS_{s,p}} := \left \| \left ( \sum_j |2^{js} f_j|^2 \right )^{1/2} \right \|_{p} < \infty.
\end{equation}
It is clear that $f_j \in L^p$ for any $j \in \bZ$, whenever $(f_j) \in \frakS_{s,p}$. 

Let $\scrP$ be the space of polynomials $\bC[x_1,\dots,x_n]$, $\scrP_0:=\{0\}$, and $\scrP_m$ be the subspace of $\scrP$ consisting of polynomials of degree less than $m$ for $m \ge 1$. For $p \ne \infty$, the \textit{Triebel--Lizorkin space} $\DotF^s_{p,2}$ consists of $f \in \scrS'/\scrP$ for which
\[ \|f\|_{\DotF^s_{p,2}} := \left\| (\Delta_j f) \right\|_{\frakS_{s,p}} < \infty. \]
For $p=\infty$, the space $\DotF^s_{\infty,2}$ consists of $f \in \scrS'/\scrP$ so that there exists $(f_j) \in \frakS_{s,\infty}$ satisfying that $f=\sum_j \Delta_j f_j$ in $\scrS'/\scrP$. The norm is given by
\[ \|f\|_{\DotF^s_{\infty,2}} := \inf \|(f_j)\|_{\frakS_{s,\infty}}, \]
where the infimum is taken among all $(f_j) \in \frakS_{s,\infty}$ so that $\sum_j \Delta_j f_j = f$ in $\scrS'/\scrP$.

Define $\nu(s,p):=\max\{ 0,[ s-\frac{n}{p} ]+1 \}$. For any $f \in \DotF^s_{p,2}$, the Littlewood--Paley series $\sum_j \Delta_j f$ converges in $\scrS'/\scrP_{\nu(s,p)}$. Moreover, it induces an isometric embedding $\iota: \DotF^s_{p,2} \to \scrS'/\scrP_{\nu(s,p)}$ given by
\[ \iota(f):=\sum_j \Delta_j f. \]

\begin{definition}[Homogeneous Hardy--Sobolev spaces]
    \label{def:Hsp}
    Let $s \in \bR$ and $0<p \le \infty$. The \textit{homogeneous Hardy--Sobolev space} $\DotH^{s,p}$ consists of $f \in \scrS'$ whose class in $\scrS'/\scrP_{\nu(s,p)}$ belongs to $\iota(\DotF^s_{p,2})$. The (quasi-)semi-norm of $\DotH^{s,p}$ is given by
    \[ \|f\|_{\DotH^{s,p}} := \|[f]\|_{\DotF^s_{p,2}}, \]
    where $[f]$ denotes the class of $f$ in $\scrS'/\scrP$.
\end{definition}

Note that for $s=0$, up to equivalent (quasi-)norms, $\DotH^{0,p}$ agrees with the Hardy space $H^p$ if $p \le 1$, the Lebesgue space $L^p$ if $1<p<\infty$, and $\bmo$ if $p=\infty$. It follows by classical results of Littlewood--Paley theory, see \textit{e.g.}, \cite[Theorem 1.3.8]{Grafakos2014_ModernFA} and \cite[\S 5.2.4]{Triebel1983Spaces}. Moreover, for any $s \in \bR$, $\DotH^{s,\infty}/\scrP_{\nu(s,\infty)}$ is isomorphic to $\bmo^s$ introduced by Strichartz \cite{Strichartz1980_BMOs}. 

The reader can refer to \cite[\S 2.5]{Auscher-Egert2023OpAdp} and the references there for basic properties of the spaces $\DotH^{s,p}$, for instance, duality, complex interpolation, lifting property, and Sobolev embedding. In particular, let us mention a dense subspace of $\DotH^{s,p}$. Denote by $\scrS_\infty$ the subspace of $\scrS$ consisting of $\phi \in \scrS$ so that for any multi-index $\alpha$, $\partial^\alpha \widehat{\phi}(0)=0$. It is a dense subspace of $\DotH^{s,p}$ if $p<\infty$, or weak*-dense if $p=\infty$.

\subsection{Tent spaces and slice spaces}
\label{ssec:tent-slice}
We adapt the original definition of (unweighted) tent spaces in \cite{Coifman-Meyer-Stein1985_TentSpaces} to our parabolic setting. For any $\beta \in \bR$ and $p \in (0,\infty)$, the \textit{(parabolic) tent space} $T^p_\beta$ consists of (possibly $\bC^n$-valued) measurable functions $F$ on $\bR^{1+n}_+$ for which
\[ \|F\|_{T^{p}_\beta} := \left ( \int_{\bR^n} \left ( \int_0^\infty \fint_{B(x,t^{1/2})} |t^{-\beta} F(t,y)|^2 \, dtdy \right )^{p/2} dx \right )^{1/p} < \infty. \]
For $0<p \le 1$, the space $T^{\infty}_{\beta,([p,1])}$ consists of measurable functions $F$ on $\bR^{1+n}_+$ for which
\[ \|F\|_{T^{\infty}_{\beta,([p,1])}} := \sup_{B} \frac{1}{|B|^{[p,1]}} \left ( \int_0^{r(B)^2} \fint_B |t^{-\beta} F(t,y)|^2\,  dtdy \right )^{1/2} < \infty,  \]
where the supremum is taken over all balls $B$ of $\bR^n$. We also write $T^{\infty}_{\beta}$ for $T^{\infty}_{\beta,(0)}$. Recall that $[p,1]$ is defined in Section \ref{ssec:notation} by $\frac 1 p - 1$. 
 
The reader can refer to \cite[\S 3.1]{Auscher-Hou2023_SIOTent} for a short summary of basic properties of tent spaces, such as duality and a dense class. Here we also recall two properties:
\begin{enumerate}[label=(\roman*)]
    \item (Interpolation) 
    Let $\beta_{0}, \beta_{1} \in \bR$, $0<p_0, p_1 \le \infty$ and $\theta \in (0,1)$. If   $1/p=(1-\theta)/p_0+\theta/p_1$ and $\beta=(1-\theta)\beta_{0}+\theta\beta_{1}$,  then the complex interpolation space $[T^{p_0}_{\beta_{0}}, T^{p_1}_{\beta_{1}}]_\theta$ (in the sense of Kalton--Mitrea method, see \cite{Kalton-Mitrea1998_ComplexInterpolation}) identifies with $T^p_\beta$. See \cite[Theorem 2.12]{Amenta2018_WeightedTent}.
      \item (Embedding) Let $\beta_0 > \beta_1$ and $0<p_0<p_1 \le \infty$. Suppose
    \[ 2\beta_0-\frac{n}{p_0} = 2\beta_1-\frac{n}{p_1}. \]
    Then $T^{p_0}_{\beta_0}$ embeds into $T^{p_1}_{\beta_1}$. See \cite[Theorem 2.19]{Amenta2018_WeightedTent}.
\end{enumerate}

We shall also use another family of spaces studied in \cite[\S 3]{Auscher-Mourgoglou2019_Ep_delta}, called slice spaces. Let $1 \le p \le \infty$ and $\delta>0$. The \textit{(parabolic) slice space} $E^p_\delta$ consists of measurable (possibly $\bC^n$-valued) functions $f$ on $\bR^n$ for which
\[ \|f\|_{E^p_\delta} := \left \| \left ( \fint_{B(\cdot,\delta^{1/2})} |f(y)|^2 \, dy \right )^{1/2} \right \|_{p} < \infty. \]
We also define the space $E^{1,p}_\delta$ as the collection of $L^2_{\loc}$-functions $g$ on $\bR^n$ so that $\nabla g$ (in the sense of distributions) belongs to $E^p_\delta$, equipped with the semi-norm
\[ \|g\|_{E^{1,p}_\delta} := \|\nabla g\|_{E^p_\delta}. \]
Also define $E^{-1,p}_\delta := \{g \in \scrD'(\bR^n): g=\Div(G) \text{ for some } G \in E^p_\delta\}$ with the norm
\[ \|g\|_{E^{-1,p}_\delta} := \inf \{ \|G\|_{E^p_\delta}: G \in E^p_\delta, ~ g=\Div G\}. \]
The $L^2$-inner product on $\bR^n$ realizes $E^{p'}_{\delta}$ (resp. $E^{-1,p'}_\delta$)  as  the dual of $E^p_{\delta}$ (resp. $E^{1,p}_\delta$) when $1 \le p < \infty$.
\section{Heat equation}
\label{sec:heat}

In this section, we consider the Cauchy problem of the heat equation
\[  \begin{cases}
    \partial_t u - \Delta u = 0 & \quad \text{ in } \scrD'(\bR^{1+n}_+) \\
    u(0)=f & 
    \end{cases}. \]
Recall that any distributional solution to the heat equation on $\bR^{1+n}_+$ is in fact smooth by hypoellipticity, see for instance \cite[\S4.4]{Hormander2003PDOI}. 

\subsection{Heat extension and homogeneous Hardy--Sobolev spaces}
\label{ssec:RW}
For any $f \in \scrS'$, the function $(t,x) \mapsto (e^{t\Delta} f)(x)$ belongs to $C^\infty(\bR^{1+n}_+) \cap C([0,\infty);\scrS')$. We call it the \textit{heat extension} of $f$, denoted by $\cE_{-\Delta}(f)$.

\begin{prop}[Heat extension on $\DotH^{s,p}$]
    \label{prop:HeatExt_cont_Hsp}
    Let $s \in \bR$ and $0<p \le \infty$. Then $\cE_{-\Delta}$  is a bounded and continuous map from $\DotH^{s,p}$ to $C_0([0,\infty);\DotH^{s,p}) \cap C^\infty((0,\infty);\DotH^{s,p})$ with the estimate
    \begin{equation}
        \label{e:HeatExt_Hsp_eq}
        \sup_{t \ge 0} \|\cE_{-\Delta}(f)(t)\|_{\DotH^{s,p}} \eqsim \|f\|_{\DotH^{s,p}}.
    \end{equation}
    Here, for $p=\infty$, the space $\DotH^{s,\infty}$ is equipped with the weak*-topology against elements of $\DotH^{-s,1}$.
\end{prop}

\begin{proof} 
    First consider the case $0<p<\infty$. Let $(\Delta_{j})$ be the Littlewood--Paley operator defined in \eqref{e:L-P-op}. It is well-known that for any $s \in \bR$, there are implicit constants such that for any $f \in \scrS_\infty$,
    \[ \sup_{t \ge 0} \left\| \left( \sum_{j \in \bZ} |2^{js} \Delta_{j}(e^{t\Delta} f)|^2 \right)^{1/2} \right\|_{p} \lesssim  \left\| \left( \sum_{j \in \bZ} |2^{js} \Delta_{j} f|^2 \right)^{1/2} \right\|_{p}. \]
    By density, this inequality extends to all $\DotH^{s,p}$, so we have 
    \[ \sup_{t \ge 0} \|\cE_{-\Delta}(f)(t)\|_{\DotH^{s,p}} \lesssim \|f\|_{\DotH^{s,p}}. \]
    Moreover, one can find that $\cE_{-\Delta}$ is a bounded map from $\scrS_\infty$ to $C_0([0,\infty);\scrS_\infty) \cap \break C^\infty((0,\infty);\scrS_\infty)$. Hence, by density, it is bounded from $\DotH^{s,p}$ to $C_0([0,\infty);\DotH^{s,p}) \cap C^\infty((0,\infty);\DotH^{s,p})$.

    For $p=\infty$, we proceed by weak*-duality to obtain boundedness and regularity. This completes the proof.
\end{proof}

We now provide an intermediate result describing when the heat extension of a tempered distribution belongs to a weighted tent space in terms of its Hardy--Sobolev regularity. To the best of our knowledge, such a precise result and the next ones below do not appear in the literature.

\begin{theorem}[Characterization of $\DotH^{s,p}$ via heat extension]     
    \label{thm:RW_sol}
    Let $s \in \bR$ and $0<p \le \infty$.

    \begin{enumerate}[label=\normalfont(\roman*)]
        \item \label{item:rw_HeatExt(f)<f}
        Suppose $s<0$. If $f \in \DotH^{s,p}$ $($in particular $f \in \scrS')$, then $\cE_{-\Delta}(f)$ belongs to $T^p_{(s+1)/2}$ with the estimate
        \begin{equation}
            \label{e:rw_HeatExt(f)<f}
            \|\cE_{-\Delta}(f)\|_{T^p_{(s+1)/2}} \lesssim \|f\|_{\DotH^{s,p}}.
        \end{equation}

        \item \label{item:rw_f<HeatExt(f)}
        Suppose $f \in \scrS'$ and $\cE_{-\Delta}(f)$ belongs to $T^p_{(s+1)/2}$.
        \begin{enumerate}[label=\normalfont(\arabic*),ref=(\arabic*)]
            \item \label{item:rw_s>0_f=0}
            If $s \ge 0$ and $\frac{n}{n+s} \le p \le \infty$, then $f=0$.
        
            \item \label{item:rw_s<0_finHsp}
            If $s<0$, then $f$ belongs to $\DotH^{s,p}$ with the estimate
            \[ \|f\|_{\DotH^{s,p}} \lesssim \|\cE_{-\Delta}(f)\|_{T^p_{(s+1)/2}}. \]
        \end{enumerate} 
    \end{enumerate}
    Consequently, when $s<0$, a tempered distribution $f \in \scrS'$ lies in $\DotH^{s,p}$ if and only if $\cE_{-\Delta}(f)$ lies in $T^p_{(s+1)/2}$, with the equivalence of norms as
    \[ \|f\|_{\DotH^{s,p}} \eqsim \|\cE_{-\Delta}(f)\|_{T^p_{(s+1)/2}}. \]
\end{theorem}

\begin{remark}
    The restriction on $s$ in \ref{item:rw_HeatExt(f)<f} is as in Triebel--Lizorkin theory. For \ref{item:rw_f<HeatExt(f)}, whenever $s \in \bR$, $\cE_{-\Delta}(f) \in T^p_{(s+1)/2}$ implies that there is a representative in the class of $f$ in $\scrS'/\scrP$ that belongs to $\DotH^{s,p}$. But for $s\ge 0$ and $p<\frac{n}{n+s}$, we do not know any better conclusion.
\end{remark}

In the sequel, for a function $F(t,x)$, we also write $\|F(t)\|_{T^p_\beta}$ for $\|F\|_{T^p_\beta}$ to emphasize the time variable.

\begin{proof}
Let us start from some general facts on relations of Hardy spaces and tent spaces. Let $\psi$ be in $C^\infty(\bR^n)$ satisfying that
\begin{enumerate}[label=(\alph*)]
    \item \label{item:psi_aux-f_mean0}
    $\psi$ is of mean zero, \textit{i.e.}, $\int_{\bR^n} \psi = 0$;
    
    \item \label{item:psi_aux-f_decay}
    there exists some $\varepsilon>0$ so that for any multi-index $\alpha$,
    \begin{equation}
        \label{e:psi_point_est}
        |\partial_x^\alpha \psi(x)| \lesssim  \left( 1+|x| \right)^{-n-\varepsilon-|\alpha|}. 
    \end{equation}
\end{enumerate}
Set $\psi_{t}(x):=t^{-n/2}\psi(t^{-1/2}x)$. One can deduce from classical arguments that the map $Q_{\psi}$ defined by
\[ Q_{\psi}(f)(t,x):=(\psi_{t}\ast f)(x) \]
is bounded from $\DotH^{0,p}$ to $T^p_{1/2}$, for $0<p \le \infty$, or equivalently, from $H^p$ to  $T^p_{1/2}$ when $0<p \le 1$, from $L^p$ to $T^p_{1/2}$ when $1<p<\infty$, and from $\bmo$ to $T^\infty_{1/2}$.
    
Indeed, for $p=2$, it is direct from Fourier transform as 
\[ \|Q_{\psi} f\|_{T^2_{1/2}}^2 \eqsim \|Q_{\psi} f\|_{L^2_{1/2}(\bR^{1+n}_+)}^2 \eqsim \int_{\bR^n} \left( \int_0^\infty |\widehat{\psi}(t^{1/2}\xi)|^2 \frac{dt}{t} \right) |\hatf(\xi)|^2 d\xi \eqsim \|f\|_2^2. \]
For $p=\infty$, it holds by a well-known Carleson measure argument, see for instance, \cite[Theorem 3.3.8(c)]{Grafakos2014_ModernFA}. For $0<p \le 1$, it follows from atomic decomposition of tent spaces (see \cite[Proposition~5]{Coifman-Meyer-Stein1985_TentSpaces}), using decay and regularity of the kernel $t^{-1/2}\psi(t^{-1/2}(x-y))$ given by \eqref{e:psi_point_est} and the moment conditions of $H^p$-atoms. Details are left to the reader. The rest follows by interpolation.
    
Moreover, let $\phi \in \scrS_{\infty}$. For $0<p<\infty$, we infer from \cite[Theorem~6]{Coifman-Meyer-Stein1985_TentSpaces} (adapted to the parabolic scaling) that for any $F \in T^p_{1/2}$, the integral
\[ S_\phi(F) := \int_0^\infty \phi_{t} \ast F(t) \, \frac{dt}t \]
converges in $\scrS'$, and hence induces a bounded operator from $T^p_{1/2}$ to $L^p$ if $1<p<\infty$ and to $H^p$ if $0<p\le 1$. For $p=\infty$, $S_\phi(F)$ is defined in $\bmo$ by testing against $h \in H^1$ using
\[ \langle S_\phi(F), h \rangle := \iint_{\bR^{1+n}_{+}} F(t,x)\, \overline{Q_{\widetilde\phi} h(t,x)} \, \frac{dtdx}t, \] where $\widetilde{\phi}(x):=\overline{\phi(-x)}$. Since $\widetilde{\phi}$ also satisfies the conditions \ref{item:psi_aux-f_mean0} and \ref{item:psi_aux-f_decay}, by duality, one can easily get $S_\phi$ is bounded from $T^\infty_{1/2}$ to $\bmo$. \\

Then we proceed to prove \ref{item:rw_HeatExt(f)<f}, assuming $s<0$. For $0<p<\infty$, we use density. Pick $f \in \scrS_\infty$ and define $g:=(-\Delta)^{-s/2} f$, so we have for any $t>0$,
\[ \cE_{-\Delta}(f)(t) = e^{t\Delta}f = (-\Delta)^{s/2} e^{t\Delta} g = t^{s/2} (-t\Delta)^{-s/2} e^{t\Delta} g = t^{s/2}\psi_{t}\ast g, \]
where $\psi$ denotes the kernel of $(-\Delta)^{-s/2} e^{\Delta}$. One can easily verify that $\psi$ satisfies the conditions \ref{item:psi_aux-f_mean0} and \ref{item:psi_aux-f_decay} with $\varepsilon=-s>0$. Then we get
\[ \|\cE_{-\Delta}(f)\|_{T^p_{(s+1)/2}}= \|Q_{\psi}(g)\|_{T^p_{1/2}} \lesssim \|g\|_{\DotH^{0,p}} = \|(-\Delta)^{-s/2}f\|_{\DotH^{0,p}} \eqsim \|f\|_{\DotH^{s,p}}. \]
Since $\scrS_\infty$ is dense in $\DotH^{s,p}$, this inequality extends to all $f \in \DotH^{s,p}$. 

For $p=\infty$, we use weak*-density. For any $f \in \DotH^{s,\infty}$, pick $(f_{k})$ as a sequence in $\scrS_\infty$ that converges weakly* to $f$ in $\DotH^{s,\infty}$, hence in $\scrS'$, since $\DotH^{s,\infty}$ embeds into $\scrS'$ when $s<0$. Using the above computation, we have
\[ \|\cE_{-\Delta}(f_{k})\|_{T^\infty_{(s+1)/2}} \lesssim \|f_{k}\|_{\DotH^{s,\infty}}. \]
In particular, $(\cE_{-\Delta}(f_k))$ is bounded in $T^\infty_{(s+1)/2}$. Let $F$ be a weak*-accumulation point of $(\cE_{-\Delta}(f_k))$ in $T^\infty_{(s+1)/2}$. Let $G \in \Cc(\bR^{1+n}_{+})$. Then 
\[ \iint_{\bR^{1+n}_{+}} (e^{t\Delta}f_{k})(x) \ovG(t,x) \, {dtdx}= \int_{0}^\infty \langle f_{k}, e^{t\Delta}G(t) \rangle \, dt, \]
where the pairing is in the sense of tempered distributions and Schwartz functions. As $t \mapsto e^{t\Delta}G(t) \in  C((0,\infty); \scrS)$ and the integral is supported on a compact subset of $(0,\infty)$, we obtain at the limit
\[ \iint_{\bR^{1+n}_{+}} F(t,x) \ovG(t,x) \, {dtdx}= \int_{0}^\infty \langle f, e^{t\Delta} G(t) \rangle \, dt. \]
So $F=\cE_{-\Delta}(f)$ and  $\|\cE_{-\Delta}(f)\|_{T^\infty_{(s+1)/2}} \lesssim \|f\|_{\DotH^{s,\infty}}$. This finishes the proof of \ref{item:rw_HeatExt(f)<f}. \\

Next, we prove \ref{item:rw_f<HeatExt(f)} and begin with \ref{item:rw_s<0_finHsp}. Let $s<0$ and $f \in \scrS'$ with $F(t,x) := (t^{-s/2}e^{t\Delta} f)(x) \in T^p_{1/2}$. Pick $\phi \in \scrS_\infty$ with the non-degeneracy condition
\begin{equation}
    \label{e:rw_phi_nondegenerate}
    \int_0^\infty \widehat{\phi}(r \xi) |r\xi|^{-s} e^{-|r\xi|^2} \frac{dr}{r} = 1, \quad \forall \xi \ne 0.
\end{equation}
We know from the above discussion that $g:=S_{\phi}(F)$ lies in $\DotH^{0,p}$. Define $\psi:=(-\Delta)^{s/2}\phi$. Computing in $\scrS'/\scrP$ and using \eqref{e:rw_phi_nondegenerate}, we have 
\[ (-\Delta)^{s/2}g = S_{\psi}(e^{t\Delta}f)=f \quad \text{ in } \scrS'/\scrP. \]
Thus, there is a representative $\tilf \in \scrS'$ in the class of $f$ in $\scrS'/\scrP$ which also lies in $\DotH^{s,p}$. By \ref{item:rw_HeatExt(f)<f}, we know that $e^{t\Delta}\tilf$ lies in $ T^p_{(s+1)/2}$, thus, so does $e^{t\Delta}(\tilf-f)$. As $\tilf-f$ is a polynomial, $e^{t\Delta}(\tilf-f)(x)= P(t,x)$ is also a polynomial, and it is easy to see that $P(t,x)=0$ from $\|P\|_{T^p_{(s+1)/2}}<\infty$. Finally, as $e^{t\Delta}(\tilf-f)\to \tilf-f$ in $ \scrS'$ when $t\to 0$, it follows that $\tilf=f$. This proves \ref{item:rw_s<0_finHsp}. \\

To establish \ref{item:rw_s>0_f=0}, we divide the discussion into two cases. \\

\paragraph{Case 1: $(s,p) \ne (0,\infty)$}
For $0<a<1$, define
\[ I(a):=\fint_a^{2a} \langle e^{t\Delta} f,\phi \rangle dt. \]
We claim that $I(a)$ tends to 0 as $a \to 0$. Meanwhile, $\langle e^{t\Delta} f,\phi \rangle$ tends to $\langle f,\phi \rangle$ as $t \to 0$, so we get $f=0$. Let us verify the claim. For $N \ge 0$, denote by $\cP_N$ the semi-norm on $\scrS$ defined by
\begin{equation}
    \label{e:cPN_S(Rn)}
    \cP_N(\phi):=\sup_{|\alpha|+|\gamma| \le N} \sup_{x \in \bR^n} |x^\alpha \partial^\gamma \phi(x)|.
\end{equation}

For $1<p \le \infty$, pick $N>n/p'$ so that
\begin{align*}
    \|\I_{(a,2a)} \phi\|_{ T^{p'}_{-(s+1)/2} }
    &\lesssim a^{\frac{s}{2}+1} \left( \int_{\bR^n} \langle x \rangle^{-Np'} \Big( \sup_{y \in B(x,(2a)^{1/2})} \langle y \rangle^N |\phi(y)| \Big)^{p'} dx \right)^{1/p'} \\
    &\lesssim a^{\frac{s}{2}+1} \cP_N(\phi),
\end{align*}
where $\langle x \rangle$ is the Japanese bracket, \textit{i.e.}, $\langle x \rangle := (1+|x|^2)^{1/2}$. We get
\begin{align*}
    I(a) 
    &\lesssim a^{-1} \|\I_{(a,2a)} e^{t\Delta} f\|_{T^p_{(s+1)/2}} \|\I_{(a,2a)} \phi\|_{T^{p'}_{-(s+1)/2}} \\
    &\lesssim a^{ \frac{s}{2} } \cP_N(\phi) \|\I_{(a,2a)} e^{t\Delta} f\|_{T^p_{(s+1)/2}}.
\end{align*}
When $s>0$, it follows that $I(a)$ tends to 0 if $a \to 0$ as $\|\I_{(a,2a)} e^{t\Delta} f\|_{T^p_{(s+1)/2}} \le \|e^{t\Delta} f\|_{T^p_{(s+1)/2}}$. When $s=0$ but $p \ne \infty$, it also holds as $\|\I_{(a,2a)} e^{t\Delta} f\|_{T^p_{(s+1)/2}}$ tends to 0 if $a \to 0$. 

For $\frac{n}{n+s} \le p \le 1$, \textit{i.e.}, $s-n[p,1] \ge 0$, note that
\[ \|\I_{(a,2a)} \phi\|_{T^\infty_{-({(s+1)}/{2}),([p,1])}} \lesssim a\cP_0(\phi), \]
so the claim follows as
\[ I(a) \lesssim a^{-1} \|\I_{(a,2a)} e^{t\Delta} f\|_{T^p_{ {(s+1)}/{2} }} \|\I_{(a,2a)} \phi\|_{T^\infty_{-({(s+1)}/{2}),([p,1])}} \lesssim_\phi \|\I_{(a,2a)} e^{t\Delta} f\|_{T^p_{(s+1)/2}}, \]
which also tends to 0 when $a \to 0$. \\

\paragraph{Case 2: $s=0$ and $p=\infty$}
For any $\phi \in \Cc(\bR^n)$, pick a ball $B \subset \bR^n$ containing $\supp(\phi)$. Note that
\[ \int_0^{r(B)^2} |\langle e^{t\Delta} f, \phi \rangle|^2 \frac{dt}{t} \lesssim \|e^{t\Delta} f\|_{T^\infty_{1/2}}^2 \cP_0(\phi)^2 |B|^2 < \infty. \]
The fact that $\langle e^{t\Delta} f, \phi \rangle$ tends to $\langle f, \phi \rangle$ if $t \to 0$ forces $f=0$ for the integral to converge, so $f=0$ in $\scrD'(\bR^n)$. This completes the proof.
\end{proof}

Replacing the heat extension by its gradient gives us the following 

\begin{cor}[Characterization of $\DotH^{s,p}$ via gradient of heat extension] 
    \label{cor:RW_grad}
    Let $s \in \bR$ and $0<p \le \infty$.

    \begin{enumerate}[label=\normalfont(\roman*)]
        \item \label{item:rw_grad_nabla-HeatExt(f)<f}
        Suppose $s<1$ and $f \in \DotH^{s,p}$. Then $\nabla \cE_{-\Delta}(f)$ lies in $T^p_{s/2}$ with
        \[ \|\nabla \cE_{-\Delta}(f)\|_{T^p_{s/2}} \lesssim \|f\|_{\DotH^{s,p}}. \] 

        \item \label{item:rw_grad_f<nabla-HeatExt(f)}
        Suppose $f \in \scrS'$ with $\nabla \cE_{-\Delta}(f) \in T^p_{s/2}$.
        \begin{enumerate}[label=\normalfont(\arabic*),ref=(\arabic*)]
            \item \label{item:rw_grad_f<nabla-HeatExt(f)_s>1}
            If $s \ge 1$ and $\frac{n}{n+s-1} \le p \le \infty$, then $f$ is a constant.
        
            \item \label{item:rw_grad_f<nabla-HeatExt(f)_s<1}
            If $s<1$, then there exists some constant $c \in \bC$ so that $f-c \in \DotH^{s,p}$ with
            \[ \|f-c\|_{\DotH^{s,p}} \eqsim \|\nabla \cE_{-\Delta}(f)\|_{T^p_{s/2}}. \]
        \end{enumerate}
    \end{enumerate}
    Consequently, when $s<1$, a tempered distribution $f \in \scrS'$ belongs to $\DotH^{s,p}+\bC$ if and only if $\nabla \cE_{-\Delta}(f)$ lies in $T^p_{s/2}$, with the equivalence of norms as
    \[ \inf_{c \in \bC} \|f-c\|_{\DotH^{s,p}} \eqsim \|\nabla \cE_{-\Delta}(f)\|_{T^p_{s/2}}. \]
\end{cor}

\begin{proof}
    A tempered distribution $f \in \scrS'$ lies in $\DotH^{s,p}+\bC$ if and only if $\nabla f \in \DotH^{s-1,p}$, with the equivalence $\inf_{c\in \bC}\|f-c\|_{\DotH^{s,p}} \eqsim \|\nabla f\|_{\DotH^{s-1,p}}$. So the corollary follows from applying Theorem~\ref{thm:RW_sol} to $\nabla f$ instead of $f$ and using the above equivalence.  
\end{proof}

In fact, working modulo constants, one can equip $\DotH^{s,p}+\bC$ with the ``norm modulo constants'' given by $\inf_{c \in \bC}\|f-c\|_{\DotH^{s,p}}$. When $\nu(s,p) \ge 0$, constants are contained in $\DotH^{s,p}$, so $\DotH^{s,p}+\bC=\DotH^{s,p}$, and for any $c \in \bC$, $\|f-c\|_{\DotH^{s,p}} = \|f\|_{\DotH^{s,p}}$.

\subsection{Representation of heat solutions}
\label{ssec:rep_heat} 

We prove here Theorem \ref{thm:heat_Hsp}. In Corollary \ref{cor:RW_grad}, we started from an element of $ \DotH^{s,p}$ and described in which range of exponents this is equivalent to the gradient of its heat extension being in a  weighted tent space. The question is now whether all heat solutions are of this form. We prove this is true in some (different) range of exponents, by establishing a trace result followed by a representation via heat semigroup.

\begin{prop}[Representation of heat solutions]
    \label{prop:rep_heat}
    Let $s>-1$ and $\frac{n}{n+s+1} \le p \le \infty$. Let $u$ be a distributional solution to the heat equation on $\bR^{1+n}_+$ with $\nabla u \in T^p_{{s}/{2}}$. Then there exists a unique $u_0 \in \scrS'$ so that $u(t)=e^{t\Delta}u_0$ for all $t>0$. 
\end{prop}

\begin{remark}
    Note that we work in $\scrS'$, and not just modulo polynomials. This induces restrictions on $s$ and $p$ in the proof.
\end{remark}

Let us start the proof by a lemma on the growth of local $L^2$-norms of functions with $\nabla u \in T^p_{{s}/{2}}$.
\begin{lemma}
    \label{lemma:size_nabla_u_tent}
    Let $s \in \bR$ and $0<p \le \infty$. Let $u$ be in $L^2_{\loc}(\bR^{1+n}_+)$ with $\nabla u \in T^p_{{s}/{2}}$. Then for $0<a<b<\infty$ and $R>1$,
    \begin{equation}
        \label{e:heat_u_L2(a,b)B(0,R)}
        \int_a^b \int_{B(0,R)} |u|^2 \lesssim_{a,b,p,s} R^{3n+2} \left( \|\nabla u\|_{T^p_{{s}/{2}}}^2 + \|u\|_{L^2((a,b) \times B(0,1))}^2 \right).
    \end{equation}
\end{lemma}

\begin{proof}
    Poincar\'e's inequality yields
    \[ \int_{B(0,R)} \bigg| u-\fint_{B(0,R)} u \bigg|^2 \lesssim R^2 \int_{B(0,R)} |\nabla u|^2. \]
    Here and in the sequel, unspecified measures in the integral are unweighted Lebesgue measure. One can also easily verify that
    \[ \bigg| \fint_{B(0,R)} u - \fint_{B(0,1)} u \bigg| \lesssim R \int_{B(0,R)} |\nabla u|. \]
    Then we have
    \begin{align*}
        \int_a^b \int_{B(0,R)} |u|^2 
         &\lesssim \int_a^b \int_{B(0,R)}\bigg| u-\fint_{B(0,R)} u \bigg |^2 \\
         &\quad+ \int_a^b \int_{B(0,R)}\bigg| \fint_{B(0,R)} u-\fint_{B(0,1)} u \bigg |^2 + R^n \int_a^b \fint_{B(0,1)} |u|^2 \\ 
         &\lesssim R^{2n+2} \int_a^b \int_{B(0,R)} |\nabla u|^2 + R^n \int_a^b \int_{B(0,1)} |u|^2.
    \end{align*}
    Then \eqref{e:heat_u_L2(a,b)B(0,R)} follows from the claim that for any $R>1$,
    \begin{equation}
        \int_a^b \int_{B(0,R)} |\nabla u|^2 \lesssim_{a,b,p,s} R^n \|\nabla u\|_{T^p_{{s}/{2}}}^2.
        \label{e:heat_est_size_nabla_u}
    \end{equation}
    Let us verify the claim. For $2 \le p \le \infty$, \cite[Lemma 4.9]{Auscher-Hou2023_SIOTent} implies
    \[ \left \| \left ( \int_a^b \fint_{B(\cdot,b^{1/2})} |\nabla u(t,y)|^2 dtdy \right )^{1/2} \right \|_p \lesssim_{a,b,s} \|\nabla u\|_{T^p_{{s}/{2}}}. \]
    Denote by $r$ the H\"older conjugate of $p/2$, and \eqref{e:heat_est_size_nabla_u} follows from
    \begin{align*}
        \int_a^b \int_{B(0,R)} |\nabla u|^2 
        &\lesssim \int_{B(0,R+b^{1/2})} dx \int_a^b \fint_{B(x,b^{1/2})} |\nabla u(t,y)|^2 dtdy \\
        &\lesssim_{a,b,p,s} R^{\frac{n}{r}} \|\nabla u\|_{T^p_{{s}/{2}}}^2.
    \end{align*}
    For $0<p \le 2$, \eqref{e:heat_est_size_nabla_u} follows by \cite[Lemma 4.8]{Auscher-Hou2023_SIOTent} as
    \[ \int_a^b \int_{B(0,R)} |\nabla u|^2 \lesssim_{a,b,p,s} \|\nabla u\|_{T^p_{{s}/{2}}}^2. \]
    This completes the proof.
\end{proof}

\begin{proof}[Proof of Proposition \ref{prop:rep_heat}]
    The conclusion holds from applying \cite[Theorem 1.1]{Auscher-Hou2023_RepHeat}. Let us verify the size condition and the uniform control condition there. Recall that $u$ is smooth.  Lemma \ref{lemma:size_nabla_u_tent} implies that for $0<a<b<\infty$, there exists $\gamma \in (0,1/4)$ so that for any $R>0$,
    \[ \left( \int_a^b \int_{B(0,R)} |u|^2 \right)^{1/2} \lesssim_{a,b,\gamma} \exp \left( \frac{\gamma R^2}{b-a} \right), \]
    which is exactly the size condition. For the uniform control condition, we show that there exists $N>0$ so that for any $\phi \in \Cc(\bR^n)$,
    \begin{equation}
        \label{e:heat_verify_UC}
        \sup_{0<t<1/2} |\langle u(t),\phi \rangle| \lesssim \cP_N(\phi) \left( \|\nabla u\|_{T^p_{{s}/{2}}} + \|u\|_{L^2((1,2) \times B(0,1))} \right),
    \end{equation}
    where $\cP_N(\phi)$ is defined in the proof of Theorem \ref{thm:RW_sol}. To this end, fix $0<t<1/2<1<t'<2$ and $\phi \in \Cc(\bR^n)$. Note that 
    \[ \int_1^2 \int_{\bR^n} |u| |\phi| \le \int_1^2 \int_{|x|<1} |u| |\phi| + \sum_{k=1}^\infty \int_1^2 \int_{2^{k-1} \le |x| < 2^k} |u| |\phi|. \]
    Denote by $I_k$ the $k$-th term for $k \ge 0$. For $k=0$, we have
    \[ I_0 \lesssim \|u\|_{L^2((1,2) \times B(0,1))} \cP_0(\phi). \]
    For $k \ge 1$, using \eqref{e:heat_u_L2(a,b)B(0,R)}, we get
    \begin{align*}
        I_k
         &\le \left ( \int_1^2 \int_{|x|<2^k} |u|^2 \right )^{1/2} \left ( \int_1^2 \int_{2^{k-1} \le |x|<2^k} |\phi|^2 \right )^{1/2} \\ 
         &\lesssim 2^{k(2n+1)} \left( \|\nabla u\|_{T^p_{{s}/{2}}} + \|u\|_{L^2((1,2) \times B(0,1))} \right) \sup_{2^{k-1} \le |x| < 2^k} |\phi(x)| \\ 
         &\lesssim 2^{-k} \cP_{N_1}(\phi) \left( \|\nabla u\|_{T^p_{{s}/{2}}} + \|u\|_{L^2((1,2) \times B(0,1))} \right)
    \end{align*}
    for some $N_1>0$. We hence obtain
    \[ \int_1^2 \int_{\bR^n} |u| |\phi| \lesssim \cP_{N_1}(\phi) \left( \|\nabla u\|_{T^p_{{s}/{2}}} + \|u\|_{L^2((1,2) \times B(0,1))} \right). \]
    
    Next, we claim that there exists $N_2>0$ so that 
    \begin{equation}
        \label{e:heat_uc_nabla_u_N2}
        \int_0^2 \int_{\bR^n} |\nabla u| |\nabla \phi| \lesssim \cP_{N_2}(\phi) \|\nabla u\|_{T^p_{{s}/{2}}}.
    \end{equation}
    Then \eqref{e:heat_verify_UC} follows by using the mean-value inequality in time and the equation for $u$. Indeed, pick $N>\max\{N_1,N_2\}$, and we get
    \begin{align*}
        | \langle u(t),\phi \rangle| 
        &\le \fint_1^2 |\langle u(t'),\phi \rangle| \, dt' + \fint_1^2 |\langle u(t'),\phi \rangle - \langle u(t),\phi \rangle| \, dt' \\
        &\le \fint_1^2 |\langle u(t'),\phi \rangle| \, dt' + \int_0^2 \int_{\bR^n} |\nabla u| |\nabla \phi| \\
        &\lesssim \cP_N(\phi) \left( \|\nabla u\|_{T^p_{{s}/{2}}} + \|u\|_{L^2((1,2) \times B(0,1))} \right)
    \end{align*}
    as desired. 
    
    Let us verify \eqref{e:heat_uc_nabla_u_N2}. We use duality of tent spaces. Let $s>-1$. For $1<p \le \infty$, pick $N_2>n/p'+1$ and we get
    \begin{align*} \|\I_{(0,2)}\nabla \phi\|_{T^{p'}_{-{s}/{2}}} &\lesssim {\bigg( \int_0^2 t^s dt \bigg)^{1/2}
     \left( \int_{\bR^n} \bigg( \sup_{y \in B(x,2)} |\nabla \phi(y)| \bigg)^{p'} dx \right)^{1/p'} }\\
     &\lesssim \bigg ( \int_{\bR^n} \langle x \rangle^{-(N_2-1)p'} \bigg ( \sup_{y \in B(x,2)} \langle y \rangle^{N_2-1} |\nabla \phi(y)| \bigg)^{p'} dx \bigg)^{1/p'} \lesssim \cP_{N_2}(\phi). \end{align*}
    For $\frac{n}{n+s+1} \le p \le 1$, \textit{i.e.}, $s+1-n[p,1] \ge 0$, we have
    \[ \|\I_{(0,2)} \nabla\phi\|_{T^\infty_{-{s}/{2}, ([p,1])}} \lesssim \|\nabla \phi\|_\infty \le \cP_{N_2}(\phi). \]
    This proves \eqref{e:heat_uc_nabla_u_N2} and hence completes the proof.
\end{proof}

\begin{proof}[Proof of Theorem \ref{thm:heat_Hsp}]
	Assertion \ref{item:heat_tent_est} has been shown in Corollary \ref{cor:RW_grad}. For \ref{item:heat_rep}, the existence of $u_0 \in \scrS'$ comes from Proposition \ref{prop:rep_heat}.  Corollary~\ref{cor:RW_grad} and Proposition~\ref{prop:HeatExt_cont_Hsp} show the desired properties  \ref{item:heat_rep_s>1} and \ref{item:heat_rep_s<1}.
\end{proof}
\section{Basic facts about weak solutions}
\label{sec:basicfactsweaksol}

The goal of the rest of the paper is to study the Cauchy problems
\begin{equation}
    \label{e:HCLions-L2}
    \begin{cases}
    \partial_t u - \Div(A\nabla u) = \Div F \\
    u(0)=u_0
    \end{cases},
\end{equation}
where $A$ satisfies \eqref{e:uniformly_elliptic}. In this section, we review the definition of  weak solutions, prove \textit{a priori} energy inequalities, and recall elements of the $L^2$-theory. Denote by $W^{1,2}$ the inhomogeneous Sobolev space with the norm $\|f\|_{W^{1,2}} := \|f\|_2+\|\nabla f\|_2$.

\begin{definition}[Weak solutions]
    \label{def:weaksol}
    Let $0 \leq a < b \leq \infty$, $\Omega$ be an open subset of $\bR^n$, and $Q:=(a,b) \times \Omega$. Let $f$ and $F$ be in $\scrD'(Q)$. A function $u \in L^2_{\loc}((a,b);W^{1,2}_{\loc}(\Omega))$ is called a \textit{weak solution} to the equation
    \[ \partial_t u - \Div(A\nabla u) =f+ \Div F \]
    with \textit{source term} $f+ \Div F$,
    if for any $\phi \in \Cc(Q)$,
    \begin{equation}
        \label{e:weak-sol}
        -\iint_Q u ~ {\partial_t \phi} + \iint_Q (A \nabla u) \cdot {\nabla \phi} = (f,{\phi}) -(F,{\nabla \phi}).
    \end{equation}
    The pairs on the right-hand side are understood as pairings of distributions and test functions on $Q$. We say $u$ is a \textit{global weak solution} if \eqref{e:weak-sol} holds for $Q=\bR^{1+n}_+$.
    
    Let $a=0$ and $u_0 \in \scrD'(\Omega)$.  By the \textit{initial condition} $u(0)=u_{0}$, we further impose that the weak solution $u(t)$ converges to $u_{0}$ in $\scrD'(\Omega)$ as $t \to 0$.
\end{definition}
Similarly, there is a corresponding definition of weak solutions for the dual backward equation $ -\partial_t u - \Div(A^*\nabla u) =f+ \Div F$.

\subsection{Energy inequalities}
\label{sec:energyinequalities}
Let us recall without proof a form of the classical Caccioppoli's inequality.
\begin{lemma}
    \label{lemma:lions_Caccioppoli}
    Let $0<a<b<\infty$, $B \subset \bR^n$ be a ball, and $f$, $F$ be in $L^2((a,b) \times 2B)$. Let $u \in L^2((a,b);W^{1,2}(2B))$ be a weak solution to the equation $\partial_t u - \Div(A\nabla u) = f+\Div F$ in $(a,b) \times 2B$. Then $u$ lies in $C([a,b];L^2(B))$ with  
    \begin{align*}
        \|u(b)\|_{L^2(B)}^2 \lesssim
        &\left ( \frac{1}{r(B)^2}+\frac{1}{b-a} \right ) \int_a^b \|u(s)\|_{L^2(2B)}^2\, ds \\
        &+ r(B)^2 \int_a^b \|f(s)\|_{L^2(2B)}^2 \,ds + \int_a^b \|F(s)\|_{L^2(2B)}^2\, ds .
    \end{align*}
    Moreover, for any $c \in (a,b)$, it holds that
    \begin{align*}
        \int_c^b \|\nabla u(s)\|_{L^2(B)}^2&\, ds \lesssim
        \frac{1}{c-a} \left ( 1+\frac{b-a}{r(B)^2} \right ) \int_a^b \|u(s)\|_{L^2(2B)}^2 \, ds \\
        &+ \frac{r(B)^2 (b-a)}{c-a} \int_a^b \|f(s)\|_{L^2(2B)}^2 \,ds + \frac{b-a}{c-a} \int_a^b \|F(s)\|_{L^2(2B)}^2 \, ds.
    \end{align*}
    The implicit constants are independent of $a$, $b$, $c$, and $B$.
\end{lemma}

There is also a corresponding version for weak solutions to the  backward equation. In the sequel, we refer to ``Caccioppoli's inequality" in both cases.

For any $(t,x) \in \bR^{1+n}_+$, the set $W(t,x):=(t,2t) \times B(x,t^{1/2})$ is called a \textit{(parabolic) Whitney cube} at $(t,x)$. 

\begin{cor}
    \label{cor:energyinequalitytent space} 
    Let $\beta \in \bR$ and $0<p\le \infty$. For any global weak solution $u$ on $\bR^{1+n}_+$ to $\partial_t u - \Div(A\nabla u) = f+\Div F$, the \textit{a priori} energy inequality holds as
    \begin{equation}
        \label{eq:energyinequalitytentspace}
        \|\nabla u \|_{T^p_{\beta+1/2}} \lesssim  \| u \|_{T^p_{\beta+1}} + \|F \|_{T^p_{\beta+1/2}}+ \|f \|_{T^p_{\beta}}. 
    \end{equation}
    The same inequality occurs for global weak solutions to the backward equation.
\end{cor}

\begin{proof}
When $p<\infty$, using Caccioppoli's inequality on local Whitney cubes and the change of aperture property of tent-space norms, we get
\begin{align*}
    \|\nabla u\|_{T^{p}_{\beta+1/2}}
    &\eqsim \left ( \int_{\bR^n} \left ( \int_0^\infty t^{-2\beta-1} dt \fint_{t/2}^t \fint_{B(x,t^{1/2})} |\nabla u|^2 \right )^{p/2} dx \right )^{1/p} \\
    &\lesssim \left ( \int_{\bR^n} \left ( \int_0^\infty t^{-2\beta-2} dt \fint_{t/4}^{t} \fint_{B(x,2t^{1/2})} |u|^2 \right )^{p/2} dx \right )^{1/p} \\
    &\quad + \left ( \int_{\bR^n} \left ( \int_0^\infty t^{-2\beta-1} dt \fint_{t/4}^{t} \fint_{B(x,2t^{1/2})} |F|^2 \right )^{p/2} dx \right )^{1/p} \\
    &\quad + \left ( \int_{\bR^n} \left ( \int_0^\infty t^{-2\beta} dt \fint_{t/4}^{t} \fint_{B(x,2t^{1/2})} |f|^2 \right )^{p/2} dx \right )^{1/p} \\
    &\lesssim \|u\|_{T^{p}_{\beta+1}} + \|F\|_{T^{p}_{\beta+1/2}}+ \|f\|_{T^{p}_{\beta}}.
\end{align*}
When $p=\infty$, we take a covering of the Carleson boxes $(0,{r(B)^2})\times B$ by local Whitney cubes and apply Caccioppoli's inequality on each, noting that the enlarged local Whitney cubes are contained in $(0,16r(B)^2) \times 4B$ with bounded overlapping. Detailed verification is left to the reader.
\end{proof}

\subsection{\texorpdfstring{$L^2$}{L2}-theory}
\label{ssec:L2}
In this section, we summarize the $L^2$-theory as a starting point. The weak solutions are given by Duhamel's formula and semigroup theory, see Section \ref{sssec:criticalexponents} for the definition of operators $L$ and $\cE_L$. We also define the \textit{Lions' operator} $\cR^L_{1/2}: L^2(\bR^{1+n}_+) \to L^2_{\loc}(\bR^{1+n}_+)$ by the $L^2$-valued Bochner integrals (where $L^2$ denotes $L^2(\bR^n)$)
\begin{equation}
    \label{e:R^L_1/2}
    \cR_{1/2}^L(F)(t) := \int_0^t e^{-(t-s)L} \Div F(s)\, ds, \quad \forall t>0.
\end{equation}

\begin{prop}[$L^2$-theory]
    \label{prop:wpL2}
    Let $u_0 \in L^2$ and $F \in L^2(\bR^{1+n}_+)$. Then there exists a unique global weak solution $u$ to the Cauchy problem \eqref{e:HCLions-L2} with $\nabla u \in L^2(\bR^{1+n}_+)$. Moreover,
    \begin{enumerate}[label=\normalfont(\roman*)]
        \item $u$ belongs to $C_0([0,\infty);L^2)$ with the estimate
        \[ \sup_{t \ge 0} \|u(t)\|_{2} + \|\nabla u\|_{L^2(\bR^{1+n}_+)} \lesssim \|u_{0}\|_{2} +  \|F\|_{L^2(\bR^{1+n}_+)}.  \]
        
        \item {\normalfont (Duhamel's formula)}   
        $u=\cE_L(u_0)+\cR^L_{1/2}(F)$.
    \end{enumerate}
\end{prop}

Existence originates from the work of J.-L. Lions \cite[Th\'eor\`eme II.3.1]{Lions1957_L2}, and uniqueness in this larger (likely the largest) class is established in \cite[Theorem 3.11]{Auscher-Monniaux-Portal2019Lp}. The reader can refer to \cite[Theorem 2.2]{Auscher-Portal2023Lions} for a detailed survey of different proofs of this theorem.

This also allows us to obtain several equivalent expressions of these operators, using different interpretations of the equation. This  will come into play when studying their extensions. Denote by $\bI$ the identity matrix in $\mat_n(\bC)$. Define the operator $\cL_1^L: L^2(\bR^{1+n}_+) \to L^2_{\loc}((0,\infty);L^2)$ by the $L^2$-valued Bochner integrals
\begin{equation}
    \label{e:cL^L_1}
    \cL_1^L(f)(t) := \int_0^t e^{-(t-s)L} f(s)\,  ds, \quad \forall t>0.
\end{equation}

\begin{cor}[Explicit formulae]
    \label{cor:formula_L2}
    Let $u_0 \in L^2$ and $F \in L^2(\bR^{1+n}_+)$. 
    \begin{enumerate}[label=\normalfont(\roman*)]
        \item \label{item:formula_L2_RL1/2}
        Define $\tilF:=(A-\bI) \nabla \cR^L_{1/2}(F) + F$ on $\bR^{1+n}_+$. Then
        \begin{equation}
            \label{e:formula_lions_u}
            \cR^L_{1/2}(F) = \cR^{-\Delta}_{1/2}(\tilF) = \Div \cL_1^{-\Delta}(\tilF) \quad \text{ in } \scrD'(\bR^{1+n}_+).
        \end{equation}
    
        \item \label{item:formula_L2_EL}
        It holds that
        \begin{align}
            \cE_L(u_0) 
            &= \cE_{-\Delta}(u_0) + \cR^L_{1/2}((A-\bI) \nabla \cE_{-\Delta}(u_0))
            \label{e:formula_EL_RL1/2} \\
            &= \cE_{-\Delta}(u_0) - \cR^{-\Delta}_{1/2}((A-\bI) \nabla \cE_L(u_0)).
            \label{e:formula_EL_RDelta1/2}
        \end{align}
    \end{enumerate}
\end{cor}

\begin{proof}
    First consider \ref{item:formula_L2_RL1/2}. Proposition \ref{prop:wpL2} says that  $u:=\cR^L_{1/2}(F)$  is a global weak solution to the Cauchy problem
    \begin{equation}
        \label{e:lions} \tag{L} 
        \begin{cases}
        \partial_t u - \Div(A \nabla u) = \Div F \\ 
        u(0)=0
        \end{cases},
    \end{equation}
    with $\nabla u \in L^2(\bR^{1+n}_+)$. We hence infer that $\tilF$ lies in $L^2(\bR^{1+n}_+)$, and then $\tilu:=\cR^{-\Delta}_{1/2}(\tilF)$ is a global weak solution to the Cauchy problem
    \[ \begin{cases}
        \partial_t \tilu - \Delta \tilu = \Div \tilF = \Div((A-\bI) \nabla u) + \Div F \\
        \tilu(0)=0
    \end{cases} \]
    with $\nabla \tilu \in L^2(\bR^{1+n}_+)$. Therefore, $w:=u-\tilu$ is a global weak solution to the heat equation with null source term and null initial data. Since $\nabla w \in L^2(\bR^{1+n}_+)$, we deduce from uniqueness in Proposition \ref{prop:wpL2} for the heat equation that $w=0$, proving the first equality in \eqref{e:formula_lions_u}. For the last, as $\tilF \in L^2(\bR^{1+n}_+)$, for a.e. $t>0$, we have
    \begin{align*}
        \cR^{-\Delta}_{1/2}(\tilF)(t) 
        &= \int_0^t e^{(t-s)\Delta} \Div \tilF(s)\, ds \\
        &= \Div \int_0^t e^{(t-s)\Delta} \tilF(s) \, ds = \Div \cL_1^{-\Delta}(\tilF)(t)
    \end{align*}
    in the sense of distributions in $\bR^n$ and  \ref{item:formula_L2_RL1/2} follows.

    Next, we consider \ref{item:formula_L2_EL}. Proposition \ref{prop:wpL2} asserts that $v:=\cE_L(u_0)$ is the unique global solution to the homogeneous Cauchy problem
    \begin{equation*} 
        \tag{HC}
        \begin{cases}
        \partial_t v-\Div(A\nabla v) = 0\\ 
        v(0)=u_0
        \end{cases},
    \end{equation*}
    with $\nabla v \in L^2(\bR^{1+n}_+)$. Meanwhile, as $\nabla \cE_{-\Delta}(u_0) \in L^2(\bR^{1+n}_+)$, it also shows that $\tilv := \cR^L_{1/2}((A-\bI) \nabla \cE_{-\Delta}(u_0))$ is a global solution to the Cauchy problem
    \[ \begin{cases}
        \partial_t \tilv - \Div(A\nabla \tilv) =  \Div((A-\bI) \nabla \cE_{-\Delta}(u_0)) \\
        \tilv(0)=0
    \end{cases}. \] 
    Notice that $\cE_{-\Delta}(u_0) + \tilv$ is a global weak solution to the homogeneous Cauchy problem \eqref{e:hc} in the same class as $\cE_L(u_0)$. Then by uniqueness, we deduce $\cE_{L}(u_0)= \cE_{-\Delta}(u_0) + \tilv$ as wanted for \eqref{e:formula_EL_RL1/2}.

    One can also verify \eqref{e:formula_EL_RDelta1/2} via an analogous argument. Details are left to the reader. This completes the proof.
\end{proof}
\section{Properties of the Lions' operator}
\label{sec:lions} 

We provide weak solutions to Lions' equation \eqref{e:lions}, as asserted in Theorem \ref{thm:wp_lions}, by studying properties of the Lions' operator $\cR^L_{1/2}$ given by \eqref{e:R^L_1/2}. We state a general result here. Proof of the bounded extension is addressed in this section, while that of trace and continuity properties is  in the next one.

\subsection{Main properties of the Lions' operator}
\label{ssec:lions_results} 
We use the the critical exponent $\phcl{L}$ defined in \eqref{e:phc-p-(L)>1} and \eqref{e:phc-p-(L)<1}. 

\begin{theorem}[Extension of $\cR^L_{1/2}$]
    \label{thm:lions_ext}
    Let $\beta>-1$ and $\phcl{L} < p \le \infty$. Then $\cR_{1/2}^L$ extends to a bounded operator from $T^p_{\beta+{1}/{2}}$ to $T^p_{\beta+1}$, also denoted by $\cR_{1/2}^L$. Moreover, the following properties hold for any $F \in T^p_{\beta+{1}/{2}}$ and $u:=\cR^L_{1/2}(F)$.
    \begin{enumerate}[label=\normalfont(\alph*)]
        \item \label{item:lions_reg} {\normalfont (Regularity)}
        $u$ lies in $T^p_{\beta+1}$ and $\nabla u$ lies in $T^p_{\beta+1/2}$ with 
        \[ \|u\|_{T^{p}_{\beta+1}} \lesssim \|F\|_{T^{p}_{\beta+{1}/{2}}}, \quad \|\nabla u\|_{T^{p}_{\beta+{1}/{2}}} \lesssim \|F\|_{T^{p}_{\beta+{1}/{2}}}. \]

        \item \label{item:lions_formula} {\normalfont (Explicit formulae)}
        Define $\tilF:=(A-\bI)\nabla \cR^L_{1/2}(F)+F$. Then 
        \[ u = \cR^{-\Delta}_{1/2}(\tilF) = \Div \cL_1^{-\Delta}(\tilF) \quad \text{ in } \scrD'(\bR^{1+n}_+), \]
        where $\cL^{-\Delta}_1$ is defined in \eqref{e:cL^L_1}.

        \item \label{item:lions_sol}
        $u$ is a global weak solution to the equation $\partial_t u - \Div(A\nabla u) = \Div F$.

        \item \label{item:lions_cont} {\normalfont (Continuity and trace)}
        $u \in C([0,\infty);\scrS')$ with $u(0)=0$. 
    
        When $t \to 0$, the convergence occurs in the following spaces shown in Table \ref{tab:lions_cont}, with arbitrary parameters $\delta>0$, $q \in [p,\infty]$, and $s \in [-1,2\beta+1]$, and a fixed one $r:=\frac{np}{n-(2\beta+2)p}>1$.

        \begin{table}[!ht]
            \centering
            \caption{Spaces for convergence of $\cR^L_{1/2}(F)(t)$ as $t \to 0$.}
            \label{tab:lions_cont}
            \begin{tabular}{c|c|c}
            Conditions 
            & $\phcl{L}<p \le 2$ 
            & $2<p \leq \infty$ \\
            \hline
         
            $\beta \ge -1/2$ 
            & $\begin{cases}
               \scrS' & \text{ if } p<1\\ 
               L^p  & \text{ if } 1 \le p \le 2 
            \end{cases}$ 
            & $E^{-1,q}_\delta$ \\
            \hline
        
            $-1<\beta<-1/2$ 
            & $\begin{cases}
                \DotH^{-1,r} & \text{ if } p \le 1\\ 
                \DotH^{s,p} & \text{ if } 1<p \le 2 
            \end{cases}$
            &  $E^{-1,q}_\delta$
            \end{tabular}
        \end{table}
    \end{enumerate}
    Consequently, $u$ is a global weak solution to the Lions' equation \eqref{e:lions}.
\end{theorem}

We provide the proof of \ref{item:lions_reg} in this section, and the proof of \ref{item:lions_formula}, \ref{item:lions_sol}, and \ref{item:lions_cont} is deferred to Section \ref{sec:cont}. Let us first give some remarks. 
\begin{remark}
    Property \ref{item:lions_cont} allows us to make sense of $u(t)$ for any $t \ge 0$ in $\scrS'$. Besides, when $\beta \ge -1/2$ and $\phcl{L}<p<1$, we also have trace in $L^p$, but this topology is not compatible with that of $\scrS'$, so one cannot identify the limits.
\end{remark}

\begin{remark}[Whitney trace]
    When $\beta>-1/2$, there is another notion of trace (valid for any function $u \in T^p_{\beta+1}$ for $0<p\le \infty$) in the sense of taking the limit of averages on Whitney cubes as $t \to 0$ for a.e. $x \in \bR^n$. More precisely, when $\beta>-1/2$, for a.e. $x \in \bR^n$, we have
    \[ \lim_{t \rightarrow 0+} \left( \fint_{W(t,x)} |u(s,y)|^2 \, dsdy \right)^{1/2} = 0. \]
    The reader can refer to \cite[Lemma 5.3]{Auscher-Hou2023_SIOTent} for the proof. 
\end{remark}

\begin{remark}
    We mention that the bounded extension of $\cR^L_{1/2}$ on tent spaces $T^p_{0}$ was studied in \cite[Proposition 2.12]{Auscher-Monniaux-Portal2019Lp}. It was claimed there that $\cR^L_{1/2}$ extends to a bounded operator from $T^p_0$ to the Kenig--Pipher space $X^p$ for $0<p \le \infty$. However, the proof has a gap and we do not know if their full conclusion is correct. Yet, as $T^p_{1/2}$ embeds into $X^p$, our Theorem \ref{thm:lions_ext} ensures that their proposition is still valid in some range of $p$,  which in fact covers the usage of this proposition (only when $L=-\Delta$ or $p=2$) in that paper.
\end{remark}

\begin{remark}
	As we shall see in Section \ref{sec:cont}, the proofs of \ref{item:lions_formula}, \ref{item:lions_sol}, and \ref{item:lions_cont} work for any $\beta>-1$ and $\frac{n}{n+2\beta+2}<p \le \infty$, for which the operators $\cR^L_\kappa$ are bounded from $T^p_{\beta+1/2}$ to $T^p_{\beta+1/2+\kappa}$ for $\kappa=0,1/2$ ($\cR^L_0$ is defined right below).
\end{remark}

We prove Theorem \ref{thm:lions_ext} \ref{item:lions_reg} using  the following two lemmas requiring different methods. Define the operator $\cR^L_0 \colon L^1_\ssfc((0,\infty);W^{2,2}) \to L^\infty_{\loc}((0,\infty);L^2)$ by the $L^2$-valued Bochner integrals 
\begin{equation}
    \label{e:lions_T0}
    \cR_0^L(F)(t) := \int_0^t \nabla e^{-(t-s)L} \Div F(s) \, ds, \quad \forall t>0.  
\end{equation}
We also use $p_{L}(\beta)$ defined in \eqref{e:pLbeta} and set 
\[ \tilp_{L}^{\,\flat}(\beta) :=\frac{nq_+(L^\ast)'}{n+2(\beta+1)q_+(L^\ast)'}. \]

\begin{lemma}
    \label{lemma:RL-tent_ext_sio}
    Let $\beta>-1$ and $\tilp_{L}^{\,\flat}(\beta) < p \le \infty$. The operator $\cR^L_{1/2}$ (resp. $\cR^L_0$) extends to a bounded operator from $T^p_{\beta+{1}/{2}}$ to $T^p_{\beta+1}$ (resp. $T^p_{\beta+{1}/{2}}$), also denoted by $\cR^L_{1/2}$ (resp. $\cR^L_0$).
\end{lemma}

\begin{lemma}
    \label{lemma:RL-tent_ext_pde}
    \begin{enumerate}[label=\normalfont(\roman*)]
        \item \label{item:RL-tent_ext_pde_p-(L)>1}
        Suppose $p_-(L) \ge 1$. Then the properties of $\cR^L_{1/2}$ and $\cR^L_0$ in Lemma \ref{lemma:RL-tent_ext_sio} are also valid for $\beta>-1/2$ and $p_{L}(\beta)<p \le \infty$. 
        
        \item \label{item:RL-tent_ext_pde_p-(L)<1}
        Suppose $p_-(L)<1$. Then the properties of $\cR^L_{1/2}$ and $\cR^L_0$ in Lemma \ref{lemma:RL-tent_ext_sio} are also valid for $\beta>\beta(L)$ and $p_{L}(\beta)<p \le 1$.
    \end{enumerate}
\end{lemma}

The proofs are provided in Sections \ref{ssec:lions_ext_RL1/2_sio} and \ref{ssec:lions_ext_PDE}, respectively. Let us first show that given these two lemmas, Theorem \ref{thm:lions_ext} \ref{item:lions_reg} holds.

\begin{proof}[Proof of Theorem \ref{thm:lions_ext} \ref{item:lions_reg} assuming Lemmas \ref{lemma:RL-tent_ext_sio} and \ref{lemma:RL-tent_ext_pde}]
    Boundedness of $\cR^L_{1/2}$ and $\cR^L_{0}$ comes from interpolation of tent spaces, thanks to Lemmas \ref{lemma:RL-tent_ext_sio} and \ref{lemma:RL-tent_ext_pde}. The number $\phcl{L}$ is exactly designed for this interpolation argument. Using Proposition \ref{prop:lions_RL1/2-0_sio} \ref{item:sio-nabla-RL1/2=RL0}, 
    the equality $ \nabla \cR^L_{1/2}(F) = \cR^L_0(F)$ in $\scrD'(\bR^{1+n}_+)$ extends to all $F \in T^p_{\beta+1/2}$ by density (or weak*-density if $p=\infty$), so the estimate of $\nabla \cR^L_{1/2}(F)$ follows.
\end{proof}

We begin to verify the lemmas. An important tool for the proof of both lemmas is that the semigroup $(e^{-tL})$ satisfies the \textit{$L^p-L^q$ off-diagonal estimates}, see \cite[Proposition 3.15]{Auscher2007Memoire}. More precisely, for $\max\{p_-(L), 1\}<p<q<p_+(L)$, there are constants $c,C>0$ so that for any $t>0$, $E,F \subset \bR^n$ as Borel sets, and $f \in L^2 \cap L^p$,
\begin{equation*}
    \|\I_E e^{-tL} \I_F f\|_q \le C t^{-\frac{n}{2}[p,q]} \exp \left( -c\frac{\dist(E,F)^2}{t} \right) \|\I_F f\|_p.
\end{equation*}
The same also holds for the family $(t^{1/2} \nabla e^{-tL})_{t>0}$ if replacing $p_\pm(L)$ by $q_\pm(L)$.

\subsection{Proof of Lemma \ref{lemma:RL-tent_ext_sio}}
\label{ssec:lions_ext_RL1/2_sio}

\begin{proof}[Proof of Lemma \ref{lemma:RL-tent_ext_sio}]
    It is a direct consequence of   \cite[Proposition 3.3 \& Corollary 3.5]{Auscher-Hou2023_SIOTent}. The assumptions  there are verified in Proposition \ref{prop:lions_RL1/2-0_sio} just below.
\end{proof}

\begin{prop}[SIO properties of $\cR^L_{1/2}$ and $\cR^L_0$]
    \label{prop:lions_RL1/2-0_sio}
    Let $\beta>-1$.
    \begin{enumerate}[label=\normalfont(\roman*)]
        \item \label{item:sio-RL1/2}
        The operator $\cR^L_{1/2}$ given by \eqref{e:R^L_1/2} belongs to $\sio^{{1}/{2}+}_{2,q,\infty}$ when $q_+(L^\ast)' < q < p_+(L)$.
        It is bounded from $L^2_{\beta+{1}/{2}}(\bR^{1+n}_+)$ to $L^2_{\beta+1}(\bR^{1+n}_+)$.
        
        \item \label{item:sio-RL0}
        The operator $\cR^L_0$ given by \eqref{e:lions_T0} extends to an operator  (also denoted by $\cR^L_0$) in $\sio^{0+}_{2,q,\infty}$ for $q_+(L^\ast)' < q < q_+(L)$ that is bounded on $L^2_{\beta+{1}/{2}}(\bR^{1+n}_+)$.

        \item \label{item:sio-nabla-RL1/2=RL0}
        For any $F \in L^2_{\beta+{1}/{2}}(\bR^{1+n}_+)$,
        \[ \nabla \cR^L_{1/2}(F) = \cR^L_0(F) \quad \text{ in } \scrD'(\bR^{1+n}_+). \]
    \end{enumerate}
\end{prop}

\begin{proof}
    For \ref{item:sio-RL1/2}, the off-diagonal estimates of $(t^{1/2} \nabla e^{-t L^\ast})_{t>0}$ imply that  the function $(t,s) \mapsto \I_{\{s>t\}}(t,s) \nabla e^{-(s-t)L^\ast}$ lies in $\sk^{1/2}_{2,q,\infty}$ for $\max\{q_-(L^\ast), 1\} < q < q_+(L^\ast)$ by definition of this class. Using duality of singular kernels (see \cite[Corollary 2.5]{Auscher-Hou2023_SIOTent}), we get that the kernel of $\cR^L_{1/2}$, $K_{1/2}(t,s) := \I_{\{t>s\}}(t,s) e^{-(t-s)L} \Div$ belongs to $\sk^{1/2}_{2,q,\infty}$ for $q_+(L^\ast)' < q < \max\{q_-(L^\ast), 1\}' = p_+(L)$. Applying \cite[Lemma 2.6]{Auscher-Hou2023_SIOTent} yields that  $\cR^L_{1/2}$ is bounded from $L^2_{\beta+1/2}(\bR^{1+n}_+)$ to $L^2_{\beta+1}(\bR^{1+n}_+)$ as $\beta+1/2>-1/2$. This proves \ref{item:sio-RL1/2}.

    For \ref{item:sio-RL0}, we construct the extension as follows. Define the operator $\cL_0$ from $ L^2((0,\infty);D(L))$ to $L^\infty_{\loc}((0,\infty);L^2)$ by the $L^2$-valued Bochner integrals
    \[ \cL_0(f)(t) := \int_0^t Le^{-(t-s)L} f(s) \, ds, \quad \forall t>0. \]
    De Simon's theorem states $\cL_0$ extends to a bounded operator $\Tilde{\cL}_0$ on $L^2(\bR^{1+n}_+)$ \cite{deSimon1964_MRL2}. Moreover, \cite[Theorem 1.3]{Auscher-Axelsson2011_MR_L2beta} shows that $\Tilde{\cL}_0$ is bounded on $L^2_{\beta+1/2}(\bR^{1+n}_+)$ for any $\beta>-1$, and \cite[Proposition 2.5]{Auscher-Monniaux-Portal2019Lp} reveals
    \[ \cR^L_0(F) = \nabla L^{-1/2} \tilde{\cL_0} L^{-1/2} \Div F \]
    for any $F \in L^1((0,\infty);W^{2,2})$. Recall that $\nabla L^{-1/2}$ and $L^{-1/2} \Div$ are bounded on $L^2$ (see \cite{Auscher-Hofmann-Lacey-McIntosh-Tchamitchian2002_Kato}), so we define the bounded extension of $\cR^L_0$ on $L^2_{\beta+1/2}(\bR^{1+n}_+)$ by
    \[ \cR^L_0 := \nabla L^{-1/2} \tilde{\cL}_0 L^{-1/2} \Div. \]
    
    Next, let us verify that its kernel $K_0(t,s):=\I_{\{t>s\}}(t,s)\nabla e^{-(t-s)L^\ast} \Div$ belongs to $ \sk^0_{2,q,\infty}$ for $q_+(L^\ast)'<q<q_+(L)$. Indeed, it can be written as
    \[ K_0 (t,s) = \I_{\{t>s\}} (t,s) \nabla e^{-\frac{1}{2}(t-s)L} \I_{\{t>s\}} (t,s) e^{-\frac{1}{2}(t-s)L} \Div, \]
    where $(t,s) \mapsto \I_{\{t>s\}}(t,s) e^{-\frac{1}{2}(t-s)L} \Div$ lies in $\sk^{1/2}_{2,q,\infty}$ for $q_+(L^\ast)' < q < p_+(L)$, and $(t,s) \mapsto \I_{\{t>s\}}(t,s) \nabla e^{-\frac{1}{2}(t-s)L}$ lies in $\sk^{1/2}_{2,q,\infty}$ for $\max\{q_-(L), 1\} < q < q_+(L)$. Hence,  we obtain $K_0 \in \sk^0_{2,q,\infty}$ for $q_+(L^\ast)'<q<q_+(L)$ by composition.
    
    Finally, for the representation, we have to show that for any $F \in L^2_\ssfb(\bR^{1+n}_+)$ and a.e. $(t,x) \in (\bR^{1+n}_+ \setminus \pi(F))$,
    \begin{equation}
        \cR^L_0(F)(t,x) = \int_0^t (\nabla e^{-(t-s)L} \Div F(s))(x) \, ds.
        \label{e:T0_representation}
    \end{equation}
    The proof of \eqref{e:T0_representation}, as well as the identity in \ref{item:sio-nabla-RL1/2=RL0} follows by a verbatim adaptation of \cite[Lemma 5.1]{Auscher-Hou2023_SIOTent}. Details are left to the reader.
\end{proof} 

\subsection{Proof of Lemma \ref{lemma:RL-tent_ext_pde}}
\label{ssec:lions_ext_PDE}
It consists of two parts. For $p>1$, we argue by duality. For $p \le 1$, in the spirit of  \cite[Theorem 3.2]{Auscher-Portal2023Lions}, we use atomic decomposition in weighted tent spaces. Recall that for $0<p \leq 1$, a measurable function $a$ on $\bR^{1+n}_+$ is called a \textit{$T^p_\beta$-atom}, if there exists a ball $B \subset \bR^n$ so that $\supp(a) \subset [0,r(B)^2] \times B$, and
\[ \|a\|_{L^2_\beta(\bR^{1+n}_+)} \leq |B|^{-[p,2]}. \]
Such a ball $B$ is said to be \textit{associated} to $a$. Write $r:=r(B)$, $C_0:=2B$, and $C_j:=2^{j+1}B \setminus 2^j B$ for $j \ge 1$. For $j \ge 4$, define
\begin{align*}
     M^{(1)}_j &:= \left( 0,(2^3 r)^2 \right] \times C_j \\ 
     M^{(2)}_j &:= \left( (2^3 r)^2,(2^j r)^2 \right) \times C_j  \\ 
     M^{(3)}_j &:= \left[ (2^j r)^2, (2^{j+1} r)^2 \right) \times 2^{j+1} B.
\end{align*}

The next lemma shows molecular decay for $\cR^L_{1/2}$ acting on atoms.

\begin{lemma}[Molecular decay when $p_-(L) \ge 1$]
    \label{lemma:lions_atom-mol-decay} 
    Assume $p_{-}(L)\ge 1$. Let $\beta>-1/2$, $0<p \le 1$, and $p_{-}(L) < q < 2$. There exists a constant $c>0$ depending on $L$ and $q$, so that for any $T^p_{\beta+{1}/{2}}$-atom $a$ with an associated ball $B \subset \bR^n$, the following estimates hold for $u:=\cR^L_{1/2}(a)$ and  any $j \ge 4$,
    \begin{align}
        &\label{e:lions_atom_Mj1}
        \|u\|_{ L^2_{\beta+1}(M_j^{(1)}) } \lesssim 2^{jn[p,2]} e^{-c2^{2j}} |2^{j+1} B|^{[2,p]},  \\
        &\label{e:lions_atom_Mj2}
        \|u\|_{ L^2_{\beta+1}(M_j^{(2)}) }  \lesssim 2^{-j(2\beta+1+n[q,p])} |2^{j+1} B|^{[2,p]}, \\
        &\label{e:lions_atom_Mj3}
        \|u\|_{ L^2_{\beta+1}(M_j^{(3)}) } \lesssim 2^{-j(2\beta+1+n[q,p])} |2^{j+1} B|^{[2,p]}.
    \end{align}
    The implicit constants are independent of $j$ and $B$.
\end{lemma}

\begin{proof}
Note that since $a \in L^2_{\beta+{1}/{2}}(\bR^{1+n}_+)$, Proposition \ref{prop:lions_RL1/2-0_sio}~\ref{item:sio-RL1/2} ensures that $u:=\cR^L_{1/2}(a)$ lies in $L^2_{\beta+1}(\bR^{1+n}_+)$, and for a.e. $(t,x) \in \bR^{1+n}_+$, 
\[ u(t,x) = \int_0^t (e^{-(t-s)L} \Div a(s))(x)\,  ds. \]
This allows us to use duality to get estimates of $u(t)$ for a.e. $t>0$. To this end, we fix $\phi \in \Cc(\bR^n)$ and set $v_t(s,y):=(e^{-(t-s)L^\ast} \phi)(y)$. We get
\[ \langle u(t), \phi \rangle = \int_0^t \int_{\bR^n} a(s,y) \cdot \overline{\nabla v_t(s,y)} \, dsdy. \]
Using the properties of $a$ and Cauchy--Schwarz inequality, we have
\begin{equation}
    \label{e:lions_atom_u(t)_dual}
    |\langle u(t),\phi \rangle| \le |B|^{[2,p]} \left( \int_0^{\min\{r^2, t\}} \int_B s^{2\beta+1} |\nabla v_t(s,y)|^2\,  dsdy \right)^{1/2}.
\end{equation}

The first inequality \eqref{e:lions_atom_Mj1} follows from the estimate
\begin{equation}
    \label{e:lions_atom_u(t)_L2Cj_Mj1}
    \|u(t)\|_{L^2(C_j)}  \lesssim |B|^{[2,p]} t^{\beta+{1}/{2}} e^{-\frac{c(2^j r)^2}{t}} \quad  0 <t \le (2^3 r)^2,
\end{equation}
by integrating the square of both sides over $0<t \le (2^3 r)^2$. To prove \eqref{e:lions_atom_u(t)_L2Cj_Mj1}, one applies \eqref{e:lions_atom_u(t)_dual} for $\phi \in \Cc(\bR^n)$ and $\supp(\phi) \subset C_j$. The $L^2-L^2$ off-diagonal estimates of $(t^{1/2} \nabla e^{-tL^\ast})$ yield that there exists $c_1>0$ only depending on $L$ so that for any $c \in (0,c_1)$,
\[ \int_0^t \int_B s^{2\beta+1} |\nabla v_t(s,y)|^2 \, dsdy \lesssim t^{2\beta+1}e^{-\frac{2c(2^j r)^2}{t}} \|\phi\|_2^2. \]
Thus, one obtains \eqref{e:lions_atom_u(t)_L2Cj_Mj1}.

The second inequality \eqref{e:lions_atom_Mj2} is obtained similarly from
\[ \|u(t)\|_{L^2(C_j)} \lesssim |B|^{[q,p]} r^{2\beta+1} t^{-\frac{n}{2}[2,q']} e^{-\frac{c(2^j r)^2}{t}}, \quad (2^3 r)^2 < t < (2^j r)^2. \]
To this end, we use again \eqref{e:lions_atom_u(t)_dual} for $\phi \in \Cc(\bR^n)$ and $\supp(\phi) \subset C_j$. Note that $v_t$ is a weak solution to the backward equation $-\partial_s v - \Div(A^\ast \nabla v) = 0$ on $(-\infty,t) \times \bR^n$. As $\beta>-1/2$, we infer from Caccioppoli's inequality (cf. Lemma \ref{lemma:lions_Caccioppoli}) that
\begin{align*}
    \int_0^{r^2} \int_B s^{2\beta+1} |\nabla e^{-(t-s)L^\ast} \phi (y)|^2 \,dsdy
    &\le (r^2)^{2\beta+1} \int_0^{r^2} \int_B |\nabla e^{-(t-s)L^\ast} \phi (y)|^2 \,dsdy \\
    &\lesssim (r^2)^{2\beta} \int_0^{2r^2} \int_{2B} |e^{-(t-s)L^\ast} \phi (y)|^2 \,dsdy.
\end{align*}
Using H\"older's inequality and $L^2-L^{q'}$ off-diagonal estimates of $(e^{-tL^\ast})$, we get $c_2>0$, depending on $L$ and $q$, so that for any $c \in (0,c_2)$,
\begin{align*}
    \|\I_{2B} e^{-(t-s)L^\ast} \phi\|_2 
    &\lesssim |B|^{[2,q']} \|\I_{2B} e^{-(t-s)L^\ast} \phi\|_q \\
    &\lesssim |B|^{[2,q']} t^{-\frac{n}{2}[2,q']} e^{-\frac{c(2^j r)^2}{t}} \|\phi\|_2.
\end{align*}
Hence, if  $\|\phi\|_2=1$, then plugging the above estimates in \eqref{e:lions_atom_u(t)_dual} implies
\begin{align*}
    |\langle u(t),\phi \rangle|^2 
    &\le |B|^{2[2,p]} \int_0^{r^2} \int_B s^{2\beta+1} |\nabla v_t(s,y)|^2\,  dsdy \\
    &\lesssim r^{4\beta} \int_0^{2r^2} \|\I_{2B} e^{-(t-s)L^\ast} \phi\|_2^2 \, ds \lesssim |B|^{2[q,p]} r^{2(2\beta+1)} t^{-n[2,q']} e^{-\frac{2c(2^j r)^2}{t}}
\end{align*}
as desired.

The third inequality \eqref{e:lions_atom_Mj3} follows from a similar argument as that of \eqref{e:lions_atom_Mj2} with the same constant $c_2$. Details are left to the reader. The assertions hence hold for $0<c<\min\{c_1,c_2\}$, which only depends on $L$ and $q$.
\end{proof}

When $p_-(L)<1$, we have even better decay for all $\beta>-1$, using  pointwise estimates and H\"older continuity of weak solutions to the backward equation.

\begin{cor}[Molecular decay when $p_-(L)<1$]
    \label{cor:lions_atom-mol-decay_p-L<1}
    Suppose $p_-(L)<1$. Let $\beta>-1$, $0<p \le 1$, and $0<\eta<n(\frac{1}{p_-(L)}-1)$. There exists a constant $c>0$ only depending on $L$, so that for any $T^p_{\beta+{1}/{2}}$-atom $a$ with an associated ball $B \subset \bR^n$, the following estimates hold for $u:=\cR^L_{1/2}(a)$ and any $j \ge 4$,
    \begin{align*}
        &\|u\|_{ L^2_{\beta+1}(M_j^{(1)}) } \lesssim 2^{jn[p,2]} e^{-c2^{2j}} |2^{j+1} B|^{[2,p]},  \\
        &\|u\|_{ L^2_{\beta+1}(M_j^{(2)}) }  \lesssim 2^{-j(2\beta+1+\eta-n[p,1])} |2^{j+1} B|^{[2,p]}, \\
        &\|u\|_{ L^2_{\beta+1}(M_j^{(3)}) } \lesssim 2^{-j(2\beta+1+\eta-n[p,1])} |2^{j+1} B|^{[2,p]}.
    \end{align*}
    The implicit constants are independent of $j$ and $B$.
\end{cor}

\begin{proof}
Fix $c>0$ as the constant given by the $L^2-L^2$ off-diagonal estimates of $(t^{1/2} \nabla e^{-tL^\ast})$, which only depends on $L$. The first inequality is the same as in Lemma \ref{lemma:lions_atom-mol-decay}. We only need to show the second, and the third follows analogously. The second inequality is also obtained by integrating the square of 
\begin{equation}
    \label{e:lions_atom_Mj2_u(t)L2_Cj}
    \|u(t)\|_{L^2(C_j)} \lesssim |B|^{[2,p]} r^{\frac{n}{2} + 2\beta+1+\eta} t^{-(\frac{n}{4}+\frac{\eta}{2})} e^{-\frac{c(2^j r)^2}{t}}, \quad (2^3 r)^2 < t < (2^j r)^2.
\end{equation} 
To prove it, we use \eqref{e:lions_atom_u(t)_dual} again for $\phi \in \Cc(\bR^n)$ with $\|\phi\|_2=1$ and $\supp(\phi) \subset C_j$. When $p_{-}(L)<1$, we know from \cite[Chapter 14]{Auscher-Egert2023OpAdp} that $(e^{-tL^*})$ satisfies $L^2-L^\infty$ off-diagonal estimates, and it is uniformly bounded from $L^2$ to $\Dot{\Lambda}^\eta$ for $0<\eta<n(\frac{1}{p_-(L)}-1)$, where $\Dot{\Lambda}^\eta$ is the homogeneous $\eta$-H\"older space. In particular, $v_t(s,y):=(e^{-(t-s)L^\ast} \phi)(y)$ has pointwise values. Set $\tilv_t(s,y):=v_t(s,y)-v_t(0,0)$. Write $\tilv_t(s,y)=  (v_t(s,y)-v_t(s,0))+ (v_t(s,0)-v_t(0,0))$ and we estimate each term separately. For the first term, we use  H\"older continuity. For the second, we write it as $\int_{0}^s (L^\ast e^{-(t-\tau) L^\ast} \phi )(0)\, d\tau$  and use $L^2-L^\infty$ boundedness of $(tL^\ast e^{-tL^\ast})$. In summary, as $\dist(\supp(\phi),B) \sim 2^j r$, we obtain that for any $0<\eta<n(\frac{1}{p_-(L)}-1)$,
\begin{equation}
    \label{e:lions_atom_tilvt_point_est}
    \sup_{0<s<2r^2, y \in 2B} |\tilv_t(s,y)| \lesssim t^{-(\frac{n}{4}+\frac \eta 2)} e^{-\frac{c(2^j r)^2}{t}} r^\eta \|\phi\|_2.
\end{equation}
By analyticity of the semigroup again, one finds that $\partial_{s} \tilv_{t}(s,y)$ satisfies the similar estimate with an extra factor $(t-s)^{-1} \sim t^{-1}$.  

Let $\chi$ be in $C^\infty([0,\infty))$ with $\supp(\chi) \subset [0,\frac{3}{2} r^2]$, $0 \le \chi \le 1$, $\chi=1$ on $[0,r^2]$, and $\|\chi'\|_{\infty}\lesssim r^{-2}$. Using integration by parts in variable $s$, we get
\begin{align*}
    &\int_0^{r^2} \int_B  |\nabla \tilv_t(s,y)|^2 s^{2\beta+1}\, dsdy 
    \le \int_0^{\frac{3}{2} r^2} \int_B \chi(s) |\nabla \tilv_t(s,y)|^2 s^{2\beta+1}\, dsdy \\ 
    &\qquad \quad \lesssim \int_0^{\frac{3}{2} r^2} \int_B |\chi'(s)| |\nabla \tilv_t(s,y)|^2 s^{2\beta+2} \, dsdy \\ 
    &\qquad \qquad + \int_0^{\frac{3}{2} r^2} \int_B \chi(s) |\nabla \tilv_t(s,y)| |\nabla \partial_s \tilv_t(s,y)| s^{2\beta+2} \, dsdy =: I_1 + I_2.
\end{align*}
For $I_1$, note that $\tilv_t$ is also a weak solution to the backward equation $-\partial_s v - \Div(A^\ast \nabla v) = 0$ on $(-\infty,t) \times \bR^n$. As $|\chi'| \lesssim r^{-2}$ and $2\beta+2>0$, we get from Caccioppoli's inequality and \eqref{e:lions_atom_tilvt_point_est} that
\begin{align*}
    I_1
    &\lesssim r^{4\beta+2} \int_{0}^{\frac{3}{2} r^2} \int_B |\nabla \tilv_t|^2 \lesssim r^{4\beta} \int_0^{2r^2} \int_{2B} |\tilv_t|^2 \\ 
    &\lesssim |B| r^{4\beta +2+  2 \eta} t^{-(\frac n 2+\eta)} e^{-\frac{2c(2^j r)^2}{t}} \|\phi\|_2^2.
\end{align*}
For $I_2$, using again Caccioppoli's inequality and the above estimate for $\partial_s \tilv_t(s,y)$, we find
\begin{align*}
    I_2 
    &\lesssim r^{4\beta+4} \int_0^{\frac{3}{2} r^2} \int_B |\nabla \partial_s \tilv_t|^2  
    \lesssim  r^{4\beta+2} \int_0^{{2} r^2} \int_{2B} | \partial_s \tilv_t|^2 \\
    &\lesssim |B| r^{4\beta+4+2{\eta}} t^{-(\frac{n}{2}+1+{\eta})} e^{-\frac{2c(2^j r)^2}{t}} \|\phi\|_2^2. 
\end{align*}
Gathering these estimates, we easily obtain \eqref{e:lions_atom_Mj2_u(t)L2_Cj}.
\end{proof}

\begin{proof}[Proof of Lemma \ref{lemma:RL-tent_ext_pde}]
We begin with \ref{item:RL-tent_ext_pde_p-(L)>1} and split the discussion into two cases. \\

\paragraph{Case 1: $\max\{p_L(\beta),1\}<p \le \infty$}
First consider $\cR^L_{1/2}$. By density (or weak$^\star$-density when $p=\infty$), it suffices to show that for any $F \in T^p_{\beta+1/2} \cap L^2(\bR^{1+n}_+)$, 
\begin{equation}
    \label{e:RL-tent-ext-pde-bdd}
    \|\cR^L_{1/2}(F)\|_{T^p_{\beta+1}} \lesssim \|F\|_{T^p_{\beta+1/2}}.
\end{equation}
To prove it, we use duality. Fix $\phi \in \Cc(\bR^{1+n}_+)$ and Fubini's theorem ensures that
\[ \langle \cR^L_{1/2}(F),\phi \rangle_{L^2(\bR^{1+n}_+)} = \int_0^\infty \int_{\bR^n} F(s,y) \left( \int_s^\infty \overline{\nabla e^{-(t-s)L^\ast} \phi(t) }(y) \, dt \right) dsdy. \]
Define
\[ v(s,y) := \int_s^\infty e^{-(t-s)L^\ast} \phi(t) \, dt. \]
As $1 \le p'<\infty$, by duality of tent spaces, \eqref{e:RL-tent-ext-pde-bdd} follows the claim  
\begin{equation}
    \label{e:RL-tent-ext-pde-bdd-dual}
    \|\nabla v\|_{T^{p'}_{-(\beta+1/2)}} \lesssim \|\phi\|_{T^{p'}_{-(\beta+1)}}.
\end{equation}
Indeed, observe that $v \in L^2_{\loc}((0,\infty);W^{1,2}_{\loc})$ is a weak solution to the backward equation $-\partial_s v - \Div(A^\ast \nabla v) = \phi$ on $ \bR^{1+n}$. Applying Corollary \ref{cor:energyinequalitytent space} leads to 
\[ \|\nabla v\|_{T^{p'}_{-(\beta+1/2)}} \lesssim \|v\|_{T^{p'}_{-\beta}} + \|\phi\|_{T^{p'}_{-\beta-1}}. \]
As we assume $p_{-}(L)\ge 1$, we claim that for $\beta>-1/2$ and $0 < p' < \infty$,
\[ \|v\|_{T^{p'}_{-\beta}} \lesssim \|\phi\|_{T^{p'}_{-\beta-1}}. \]
Namely, apply the results in  \cite{Auscher-Hou2023_SIOTent}: first, the  operator $\phi \mapsto v$ lies in $\sio^{1-}_{2,q,\infty}$ for $\max\{p_-(L^\ast), 1\} <q<p_+(L^\ast)$ by (4.14) and Proposition 2.11, then apply Proposition 3.4 and Corollary 3.6. This, together with the previous inequality, proves the claim \eqref{e:RL-tent-ext-pde-bdd-dual}. The bounded extension of $\cR^L_{1/2}$ hence follows.

Next, consider $\cR^L_0$. Fix $F \in T^p_{\beta+1/2} \cap L^2(\bR^{1+n}_+)$. Proposition \ref{prop:lions_RL1/2-0_sio} \ref{item:sio-nabla-RL1/2=RL0} says $\cR^L_0(F)=\nabla \cR^L_{1/2}(F)$ in $L^2(\bR^{1+n}_+)$. Meanwhile, by Proposition \ref{prop:wpL2}, $u=\cR^L_{1/2}(F)$ is a global weak solution to the equation $\partial_t u - \Div(A\nabla u) = \Div F$. Applying Corollary \ref{cor:energyinequalitytent space} and \eqref{e:RL-tent-ext-pde-bdd} yields that for any $F \in T^p_{\beta+1/2} \cap L^2(\bR^{1+n}_+)$,
\[ \|\cR^L_0(F)\|_{T^p_{\beta+1/2}} = \|\nabla \cR^L_{1/2}(F)\|_{T^p_{\beta+1/2}} \lesssim \|\cR^L_{1/2}(F)\|_{T^p_{\beta+1}} + \|F\|_{T^p_{\beta+1/2}} \lesssim \|F\|_{T^p_{\beta+1/2}}. \]
Then the bounded extension of $\cR^L_0$ follows by density. \\

\paragraph{Case 2: $p_L(\beta)<p \le1$}
Note that by definition of $\beta(L)$ (cf.~\eqref{e:beta_0}), this case only occurs when $\beta>\beta(L)$. We first consider the extension of $\cR^L_{1/2}$. Thanks to \cite[Lemma 3.2]{Auscher-Hou2023_SIOTent}, it is enough to show $\cR^L_{1/2}$ is uniformly bounded on $T^p_{\beta+{1}/{2}}$-atoms. 

Let us verify this. Let $a$ be a $T^p_{\beta+{1}/{2}}$-atom, $B \subset \bR^n$ be a ball associated to $a$, and $u:=\cR^L_{1/2}(a)$. Denote by $Q_0:=\left( 0,(2^3 r(B))^2 \right) \times 2^3 B$. As $\beta+1/2>-1/2$, Proposition \ref{prop:lions_RL1/2-0_sio} \ref{item:sio-RL1/2} implies
\begin{equation}
    \label{e:lions_atom_u_Q0}
    \|u\|_{L^2_{\beta+1}(Q_0)} \lesssim \|a\|_{L^2_{\beta+{1}/{2}}(\bR^{1+n}_+)} \lesssim |2^3 B|^{[2,p]}.
\end{equation}
Pick $q<2$ sufficiently close to $p_{-}(L)$ so that
\begin{equation}
    \label{e:lions_atom_p-q'_p-L>1}
    \frac{np_-(L)}{n+(2\beta+1)p_-(L)} < \frac{nq}{n+(2\beta+1)q} < p \le 1.
\end{equation}
As $\beta>\beta(L) \ge -1/2$, we may apply  Lemma \ref{lemma:lions_atom-mol-decay} as well as the estimates given by \eqref{e:lions_atom_u_Q0} to get 
\[ \|u\|_{T^p_{\beta+1}}^p \lesssim 1+\sum_{j \ge 4} 2^{-jp(2\beta+1+n[q,p])} \le C \]
for some constant $C<\infty$ by \eqref{e:lions_atom_p-q'_p-L>1}. The assertion hence holds in this case.

Next, we consider the bounded extension of $\cR^L_0$. In fact, one can apply Caccioppoli's inequality as in Corollary \ref{cor:energyinequalitytent space} to establish the same molecular decay estimates as in Lemma \ref{lemma:lions_atom-mol-decay} for $\nabla u$ in $L^2_{\beta+{1}/{2}}(M_j^{(i)})$ $(1 \le i \le 3)$. Meanwhile, using Proposition \ref{prop:lions_RL1/2-0_sio} \ref{item:sio-RL0}, one can also get as an alternative of \eqref{e:lions_atom_u_Q0} 
\[ \|\nabla u\|_{L^2_{\beta+{1}/{2}}(Q_0)} \lesssim \|a\|_{L^2_{\beta+{1}/{2}}(\bR^{1+n}_+)} \lesssim |2^3 B|^{[2,p]}. \]
It is then enough to repeat the above computation. Details are left to the reader. This proves \ref{item:RL-tent_ext_pde_p-(L)>1}. \\

To prove \ref{item:RL-tent_ext_pde_p-(L)<1}, as $p \le 1$, we can use  the argument in Case 2, using \eqref{e:lions_atom_u_Q0} directly and replacing the estimates of  Lemma \ref{lemma:lions_atom-mol-decay} by those in Corollary \ref{cor:lions_atom-mol-decay_p-L<1}. Picking $\eta>0$ sufficiently close to $n(\frac{1}{p_-(L)}-1)$ so that
\begin{equation*}
    \frac{np_-(L)}{n+(2\beta+1)p_-(L)} < \frac{n}{n+2\beta+1+\eta} < p \le 1,
\end{equation*} 
we obtain some constant $C<\infty$ so that
\[ \|u\|_{T^p_{\beta+1}}^p \lesssim 1+\sum_{j \ge 4} 2^{-jp(2\beta+1+\eta-n[p,1])} \le C. \] 
This proves bounded extension of $\cR^L_{1/2}$.

Bounded extension of $\cR^L_0$ follows by the same argument as in the above Case 2, using Corollary \ref{cor:lions_atom-mol-decay_p-L<1}. This completes the proof.
\end{proof}
\section{Traces and continuity}
\label{sec:cont}

In this section, we first prove Theorem \ref{thm:lions_ext} \ref{item:lions_formula}, \ref{item:lions_sol}, and \ref{item:lions_cont}, and then obtain a trace for arbitrary weak solutions with gradient controlled in weighted tent spaces in the largest possible range of exponents $\beta$ and $p$ (as it is the same as the one for  $L=-\Delta$).

\begin{proof}[Proof of Theorem \ref{thm:lions_ext}, \ref{item:lions_formula} and \ref{item:lions_sol}]
	First, Corollary \ref{cor:formula_L2} \ref{item:formula_L2_RL1/2} says that \ref{item:lions_formula} holds for $F \in T^p_{\beta+{1}/{2}} \cap L^2(\bR^{1+n}_+)$. Also note that all the operators involved have bounded extensions to $T^p_{\beta+1/2}$. Indeed, for $\cR^L_{1/2}$, $\nabla \cR^L_{1/2}$, and $\cR^{-\Delta}_{1/2}$, it has been shown in \ref{item:lions_reg}; and we invoke \cite[Theorem 4.2]{Auscher-Hou2023_SIOTent} to get boundedness of $\cL_1^{-\Delta}$ from $T^p_{\beta+{1}/{2}}$ to $T^p_{\beta+{3}/{2}}$. We conclude by density (or weak*-density if $p=\infty$), since all the weighted tent spaces embed into $L^2_{\loc}(\bR^{1+n}_+)$, hence into $\scrD'(\bR^{1+n}_+)$.

    For \ref{item:lions_sol}, we infer from \ref{item:lions_reg} that for any $F \in T^p_{\beta+1/2}$, both $u=\cR^L_{1/2}(F)$ and $\nabla u$ lie in $L^2_{\loc}(\bR^{1+n}_+)$, so $u \in L^2_{\loc}((0,\infty);W^{1,2}_{\loc})$. Recall that Proposition \ref{prop:wpL2} says for any $F \in T^p_{\beta+1/2} \cap L^2(\bR^{1+n}_+)$, $u$ is a global weak solution to $\partial_t u - \Div(A\nabla u) = \Div F$. The same density argument concludes the verification for any $F \in T^p_{\beta+1/2}$.
\end{proof}

To prove \ref{item:lions_cont}, we need a lemma allowing us to reduce the proof to the case $p>1$.
\begin{lemma}
    \label{lemma:unique_embed}
    Let $\beta>-1$ and $\phcl{L} < p \le 1$. There exist $\beta_0>-1$ and $\max\{\tilp_L(\beta_0),1\}<p_0<2$ so that $T^p_{\beta+1/2}$ embeds into $T^{p_0}_{\beta_0+1/2}$.
\end{lemma}

\begin{proof}
The crucial point is the embedding theorem for weighted tent spaces, recalled in Section \ref{ssec:tent-slice}. In parabolic scaling it asserts that $T^p_{\beta+{1}/{2}}$ embeds into $T^{p_0}_{\beta_0+{1}/{2}}$ if $\beta_0<\beta$, $p<p_0$, and $2\beta_0+1-\frac{n}{p_0} = 2\beta+1-\frac{n}{p}$.
Let $\gamma<\beta$ be the number so that $2\gamma+1 -n = 2\beta+1 - \frac{n}{p}$. We claim that $\gamma>-1$ and $\tilp_L(\gamma)<1$. If it holds, then by perturbation, we can take $\beta_0>-1$ sufficiently close to $\gamma$ and the corresponding $p_0>1$. The claim is elementary using the definitions of the exponents in the introduction but we provide an argument for the sake of completeness.
  
By definition (see also Figure \ref{fig:exponents}), we already observe that:
\begin{enumerate}[label=\normalfont(\roman*)]
    \item $\beta(L) \ge -1/2 \Longleftrightarrow p_{-}(L)\ge 1$;
    \item $p_{L}(\beta) \le \phcl{L}$, and the equality holds if and only if $\beta \ge \min\{\beta(L), -1/2\}$.
\end{enumerate}
This implies
\[ \beta>-1 \ \text{and} \ \phcl{L} <1 \Longleftrightarrow \beta>\beta(L). \]
Indeed, if the right-hand side holds, then $\beta>\beta(L)\ge -1$ and $\phcl{L} = p_L(\beta) <p_L(\beta (L))=1$. Conversely, assume the left-hand side. When $p_{-}(L)<1$, $\phcl{L} <1$ implies $\beta>\beta(L)$. When $p_{-}(L)\ge 1$, by construction, $\phcl{L} <1$ implies $\beta>-1/2= \min\{\beta(L),-1/2\}$, so $\phcl{L} = p_L(\beta) < 1$. We thus infer $\beta>\beta(L)$. 
 
It remains to verify that $\gamma>\beta(L)$. Remark that for any $\delta>-1$ and $r>0$, 
\[ r> \frac{np_-(L)}{n+(2\delta+1)p_-(L)} \Longleftrightarrow 2\delta+1 - \frac{n}{r} > -\frac{n}{p_-(L)}.\]
As $\frac{np_-(L)}{n+(2\beta+1)p_-(L)} \le \phcl{L} < p \le 1$, we have $2\gamma+1-{n}= 2\beta+1-\frac{n}{p} > -\frac{n}{p_-(L)}$. By definition of $\beta(L)$, this exactly says $\gamma>\beta(L)$ as desired.
\end{proof}

\begin{proof}[Proof of Theorem \ref{thm:lions_ext} \ref{item:lions_cont}]
    We first show that $u\in C((0,\infty);\scrS')$. Since $u$ lies in $L^2_{\loc}((0,\infty);W^{1,2}_{\loc})$, using the equation (by \ref{item:lions_sol}), we get $\partial_t u \in L^2_{\loc}((0,\infty);W^{-1,2}_{\loc})$. Next, fix $t>0$, $\phi \in \scrS$, and pick $a,b$ so that $0<a<t<b$. As $p>\phcl{L} \ge \frac{n}{n+2\beta+2}$, one can infer from the proof of \eqref{e:heat_uc_nabla_u_N2} in Proposition \ref{prop:rep_heat} that
    \begin{equation}
        \label{e:lions_cont_phi_S_tent-dual}
        \|\I_{(a,b)} \nabla \phi\|_{(T^p_{\beta+1/2})'} \lesssim_{a,b} \cP_N(\phi)
    \end{equation}
    for some sufficiently large $N>0$, where $\cP_N$ is the semi-norm on $\scrS$ given by \eqref{e:cPN_S(Rn)}. Then for $s \in (a,t)$, we apply the equation to get
    \begin{align*}
        |\langle u(t)-u(s),\phi \rangle|
        &\le \left | \int_s^t \langle \partial_\tau u(\tau),\phi \rangle d\tau \right| \\
        &\le \int_s^t \int_{\bR^n} |A\nabla u||\nabla \phi| + \int_s^t \int_{\bR^n} |F| |\nabla \phi| \\
        &\lesssim (\|\I_{(s,t)} \nabla u\|_{T^p_{\beta+1/2}} + \|\I_{(s,t)} F\|_{T^p_{\beta+1/2}}) \cP_N(\phi).
    \end{align*}
    For $p<\infty$, this proves the continuity as desired, since both $\|\I_{(s,t)} \nabla u\|_{T^p_{\beta+1/2}}$ and $\|\I_{(s,t)} F\|_{T^p_{\beta+1/2}}$ tend to 0 as $s \to t$. For $p=\infty$, we need finer estimates. Note that in this case, \eqref{e:lions_cont_phi_S_tent-dual} can be refined as
    \[ \|\I_{(s,t)} \nabla \phi\|_{T^1_{-(\beta+1/2)}} \lesssim \left( \int_s^t \tau^{-(2\beta+1)} d\tau \right)^{1/2} \cP_N(\phi). \]
    Repeating the above computation yields
    \[ |\langle u(t)-u(s),\phi \rangle| \lesssim \left( \int_s^t \tau^{-(2\beta+1)} d\tau \right)^{1/2} (\|\nabla u\|_{T^\infty_{\beta+1/2}} + \|F\|_{T^\infty_{\beta+1/2}}) \cP_N(\phi), \]
    which also converges to 0 as $s \to t$.
    
    Next, we consider the trace. We split the discussion in 4 cases. \\

    \paragraph{Case 1: $\beta>-1$ and $2<p \le \infty$}
    We use the formula in \ref{item:lions_formula}. Write $\tilF:=(A-\bI)\nabla u+F$ so that $u=\Div \cL^{-\Delta}_1(\tilF)$. Since $\tilF \in T^p_{\beta+{1}/{2}}$, we invoke \cite[Theorem 4.2(e)]{Auscher-Hou2023_SIOTent} to get that $\cL_1^{-\Delta}(\tilF)(t)$ tends to 0 in $E^q_\delta$ as $t \to 0$, so $u(t)=\Div \cL_1^{-\Delta}(\tilF)(t)$ tends to 0 in $E^{-1,q}_\delta$ as $t \to 0$. \\
    
    \paragraph{Case 2: $-1<\beta<-1/2$ and $\max\{\phcl{L},1\} < p \le 2$} 
    By interpolation of Hardy--Sobolev spaces, it is enough to show that $u(t)$ tends to 0 as $t \to 0$ in $\DotH^{2\beta+1,p}$ and $\DotH^{-1,p}$.
    
    For $\DotH^{-1,p}$, we also apply the formula in \ref{item:lions_formula}. Again, as $\tilF\in T^p_{\beta+{1}/{2}}$, \cite[Theorem 4.2(e)]{Auscher-Hou2023_SIOTent} says that $\cL_1^{-\Delta}(\tilF)(t)$ tends to 0 in $L^p$ when $t \to 0$, so $u(t)$ tends to 0 in $\DotW^{-1,p}=\DotH^{-1,p}$ as $t \to 0$.
    
    For $\DotH^{2\beta+1,p}$, we first claim that for any $t>0$,
    \begin{equation}
        \label{e:lions_u(t)_beta<-1/2_p>1}
        \|u(t)\|_{\DotH^{2\beta+1,p}} \lesssim \|\I_{(0,t)} \tilF\|_{T^p_{\beta+{1}/{2}}}.
    \end{equation}
    If it holds, then it is clear that $u(t)$ tends to 0 in $\DotH^{2\beta+1,p}$ as $t \to 0$, since $\|\I_{(0,t)} \tilF\|_{ T^p_{\beta+{1}/{2}} }$ tends to 0 as $t \to 0$.

    Let us prove the claim \eqref{e:lions_u(t)_beta<-1/2_p>1} by duality. Fix $\varphi \in \scrS_\infty$. Define 
    \[ g(s,x):= \I_{\{s<t\}}(s)\nabla e^{-(t-s)\Delta} \varphi(x). \]
    Observe that
    \begin{equation}
        \label{e:lions_est_g-by-phi}
        \|g\|_{T^{p'}_{-(\beta+{1}/{2})}} \lesssim \|\varphi\|_{\DotH^{-(2\beta+1),p'}}.
    \end{equation}
    Indeed, for $\phcl{L}<p<2$, we have
    \[ \|g\|_{T^{p'}_{-(\beta+{1}/{2})}} \le \left ( \int_{\bR^n} \left ( \int_0^t \cM (|\nabla e^{(t-s)\Delta} \varphi|^2)(x) s^{2\beta+1} ds \right )^{p'/2} dx \right )^{1/p'}, \]
    where $\cM$ is the Hardy--Littlewood maximal operator on $\bR^n$. By Minkowski's inequality and boundedness of $\cM$ on $L^{p'/2}$, we further have
    \begin{align*}
        \|g\|_{T^{p'}_{-(\beta+{1}/{2})}} 
        &\lesssim \left ( \int_0^t \|\nabla e^{(t-s)\Delta} \varphi\|_{p'}^2 ~ s^{2\beta+1} ds \right )^{1/2} \\
        &\lesssim \left ( \int_0^t  (t-s)^{-2(\beta+1)} \|\varphi\|_{ \DotH^{-(2\beta+1),p'} }^2 s^{2\beta+1} ds \right )^{1/2} \lesssim \|\varphi\|_{\DotH^{-(2\beta+1),p'}},
    \end{align*}
    proving \eqref{e:lions_est_g-by-phi}. Here, as $0<-(2\beta+1)<1$, the second inequality comes by interpolating the boundedness of $(e^{t\Delta})$ on $\DotH^{1,p'}$ (see Proposition \ref{prop:HeatExt_cont_Hsp}) and that of $(t^{1/2} \nabla e^{t\Delta})$ on $L^{p'} \simeq \DotH^{0,p'}$. 
    For $p=2$, the same interpolation argument also yields
    \[ \|g\|_{T^2_{-(\beta+{1}/{2})}} \eqsim \left ( \int_0^t \|\nabla e^{(t-s)\Delta} \varphi\|_{2}^2\ s^{2\beta+1} ds \right )^{1/2} \lesssim \|\varphi\|_{\DotH^{-(2\beta+1),2}}. \]
    This shows \eqref{e:lions_est_g-by-phi}. Then we apply the formula \ref{item:lions_formula}, duality of tent spaces, and the estimate \eqref{e:lions_est_g-by-phi} to get
    \begin{align*}
        |\langle u(t),\varphi \rangle| 
        &= |\langle \Div \cL^{-\Delta}_1(\tilF)(t),\varphi \rangle| \le \int_0^t \int_{\bR^n} |\tilF(s,x)| |\nabla e^{-(t-s) \Delta} \varphi (x)| \, dsdx \\
        &\lesssim \|\I_{(0,t)} \tilF\|_{T^p_{\beta+{1}/{2}}} \|g\|_{T^{p'}_{-(\beta+{1}/{2})}} \lesssim \|\I_{(0,t)} \tilF\|_{T^p_{\beta+{1}/{2}}} \|\varphi\|_{\DotH^{-(2\beta+1),p'}}.
    \end{align*}
    Therefore, the claim \eqref{e:lions_u(t)_beta<-1/2_p>1} follows since for any $t>0$,
    \[ \|u(t)\|_{\DotH^{2\beta+1,p}} = \sup_{\varphi} |\langle u(t),\varphi \rangle| \lesssim \|\I_{(0,t)} \tilF\|_{T^p_{\beta+{1}/{2}}}, \]
    where the supremum is taken among $\varphi \in \scrS_\infty$ with $\|\varphi\|_{\DotH^{-(2\beta+1),p'}}=1$, noting that $\scrS_\infty$ is dense in $\DotH^{-(2\beta+1),p'}$. \\

    \paragraph{Case 3: $-1<\beta<-1/2$ and $\phcl{L}<p \le 1$}
    We use embedding of tent spaces. 
    Pick $-1<\beta_0<\beta<-1/2$ and $\max\{\tilp_L(\beta_0),1\}<p_0<2$ as in Lemma \ref{lemma:unique_embed} so that $T^p_{\beta+1/2}$ embeds into $T^{p_0}_{\beta_0+1/2}$. In particular, as $F \in T^p_{\beta+{1}/{2}}$, we have $F \in T^{p_{0}}_{\beta_0+{1}/{2}}$, for which we get back to Case 2. Hence, $u(t)$ tends to 0 as $t \to 0$ in $\DotH^{2\beta_0+1,p_{0}}$, and therefore in $\DotH^{-1,r}$ by Sobolev embedding and definition of $r$. \\

    \paragraph{Case 4: $\beta \ge -1/2$ and $\phcl{L} < p \le 2$}
    Lemma \ref{lemma:lions_Caccioppoli} yields
    \begin{align*}
        \|u(t)\|_p^p 
        &\le \int_{\bR^n} \left ( \fint_{B(x,\frac{\sqrt{t}}{4})} |u(t,y)|^2\,  dy \right )^{p/2} dx \\ 
        &\lesssim \int_{\bR^n} \left ( \fint_{t/2}^t \fint_{B(x,\frac{\sqrt{t}}{2})} |u|^2 \right )^{p/2} dx + \int_{\bR^n} \left ( \int_{t/2}^t \fint_{B(x,\frac{\sqrt{t}}{2})} |F|^2 \right )^{p/2} dx 
    \end{align*}
    When $\beta=-1/2$, Since $u \in T^p_{1/2}$ and $F \in T^p_{0}$, Lebesgue's dominated convergence theorem yields the integrals on the right-hand side tend to 0 as $t \to 0$. Thus, $u(t)$ tends to 0 in $L^p$ as $t \to 0$. When $\beta>-1/2$, the right-hand side can be controlled by $t^{p(\beta+{1}/{2})} ( \|u\|_{T^p_{\beta+1}} + \|F\|_{T^p_{\beta+{1}/{2}}} )^p$, which also tends to 0 as $t \to 0$ as $\beta>-1/2$.
    
    It remains to prove that for $\beta \ge -1/2$ and $\phcl{L}<p<1$, $u(t)$ tends to 0 in $\scrS'$, because in this case, $L^p$ does not embed into $\scrS'$. To this end, we also use embedding of tent spaces. Pick $\beta_0 \in (-1,\beta]$ and $p_0>1$ as defined in Lemma \ref{lemma:unique_embed}, so $F \in T^{p_0}_{\beta_0+1/2}$. Then we apply the above discussion to get $u(t)$ tends to 0 as $t \to 0$ in $L^{p_0}$ if $\beta_0 \ge -1/2$, or in $\DotH^{-1,r}$ if $-1<\beta_0<-1/2$, see Case 2. But both spaces of course embed into $\scrS'$. We hence obtain the trace in $\scrS'$ as desired. 
    
    This completes the proof.
\end{proof}

As an application, we establish distributional traces for weak solutions to the equation $\partial_t u - \Div(A \nabla u) = 0$, which is useful when studying uniqueness.

\begin{prop}[Trace]
    \label{prop:trace}
    Let $\beta>-1$ and $\frac{n}{n+2\beta+2} < p \le \infty$. Let $u$ be a global weak solution to $\partial_t u - \Div(A\nabla u) = 0$ with $\nabla u \in T^p_{\beta+{1}/{2}}$. Then there exists a unique $u_{0} \in \scrS'$ such that $u(t)$ converges to $u_{0}$ in $\scrS'$ as $t \to 0$, and
    \begin{equation}
        \label{e:trace_rep}
        u = \cE_{-\Delta}(u_0)+\cR^{-\Delta}_{1/2}((A-\bI)\nabla u).
    \end{equation}
    In addition,  
    \begin{enumerate}[label=\normalfont(\roman*)]
        \item 
        \label{item:trace_beta>0}
        If $\beta \ge 0$ and $\frac{n}{n+2\beta} \le p \le \infty$, then $u_{0}$ is a constant.        
        \item 
        \label{item:trace_-1<beta<0}
        If $-1<\beta<0$, then there exist $g \in \DotH^{2\beta+1,p}$ and $c \in \bC$ so that $u_{0}=g+c$ with
        \[ \|g\|_{\DotH^{2\beta+1,p}} \lesssim \|\nabla u\|_{T^p_{\beta+{1}/{2}}}. \]
    \end{enumerate}
    \end{prop}

\begin{proof}
    Note that $\frac{n}{n+2\beta+2}=\phcl{-\Delta}$, so by Theorem \ref{thm:lions_ext} \ref{item:lions_reg}, as $\nabla u \in T^p_{\beta+{1}/{2}}$, we get $\cR^{-\Delta}_{1/2}((A-\bI)\nabla u) \in T^p_{\beta+1}$ and $\nabla \cR^{-\Delta}_{1/2}((A-\bI)\nabla u) \in T^p_{\beta+{1}/{2}}$. Besides, Theorem \ref{thm:lions_ext} \ref{item:lions_cont} says that $\cR^{-\Delta}_{1/2}((A-\bI)\nabla u)(t)$ vanishes in $\scrS'$ as $t \to 0$.
    
    Define $\tilu \in L^2_{\loc}((0,\infty);W^{1,2}_{\loc})$ by
    \[ \tilu:=u-\cR^{-\Delta}_{1/2}((A-\bI)\nabla u). \]
    Observe that $\tilu$ is a global weak (hence distributional) solution to the heat equation with $\nabla \tilu \in T^p_{\beta+{1}/{2}}$. Then we invoke Proposition \ref{prop:rep_heat} to get $u_0 \in \scrS'$ so that $\tilu=\cE_{-\Delta}(u_0)$, which proves \eqref{e:trace_rep} and also implies that $u(t)$ converges to $u_{0}$ in $\scrS'$ as $t \to 0$. Moreover, Theorem \ref{thm:heat_Hsp} \ref{item:heat_rep} says that $u_0$ has the properties asserted in \ref{item:trace_beta>0} and \ref{item:trace_-1<beta<0}. This completes the proof.
\end{proof}
\section{Homogeneous Cauchy problem}
\label{sec:hc}

This section is devoted to investigating existence of global weak solutions to the homogeneous Cauchy problem \eqref{e:hc}, where $v_0$ lies in $\DotH^{2\beta+1,p}$ with $-1<\beta<0$, as asserted in Theorem \ref{thm:wp_hc}.

Before stating our results, let us first briefly revisit the heat equation. For any $v_0 \in \scrS'$, the heat extension of $v_0$, $\cE_{-\Delta}(v_{0}) \in C^\infty(\bR^{1+n}_+)$ is a global weak (hence classical) solution to the heat equation with initial data $v_0$. We established a series of estimates on boundedness of the heat extension on $\DotH^{s,p}$, cf. Proposition \ref{prop:HeatExt_cont_Hsp}, Theorem \ref{thm:RW_sol}, and Corollary \ref{cor:RW_grad}, with the correspondence $s=2\beta+1$. For general $L$, we cannot start from tempered distributions, and our weak solutions will be constructed by extending $\cE_L$ (as defined in \eqref{e:sol_map} on $L^2$) to $\DotH^{2\beta+1,p}$. When $L=-\Delta$, both approaches are consistent.

\subsection{Main properties of the semigroup solution map}
\label{ssec:hc_results}
Let us state the main theorem on extension of $\cE_L$ to $\DotH^{2\beta+1,p}$, of which the existence of the global weak solution asserted in Theorem \ref{thm:wp_hc} is a direct consequence. 

\begin{theorem}[Extension of $\cE_L$ to $\DotH^{2\beta+1,p}$]
    \label{thm:hc_ext_EL}
    Let $-1<\beta<0$ and $\phcl{L} < p \le \infty$. Then $\cE_L$ extends to an operator from $\DotH^{2\beta+1,p}$ to $L^2_{\loc}((0,\infty);W^{1,2}_{\loc})$, also denoted by $\cE_L$, so that the following properties hold for any $v_0 \in \DotH^{2\beta+1,p}$ and $v:=\cE_L(v_0)$.
    \begin{enumerate}[label=\normalfont(\alph*)]
        \item \label{item:hc_EL_reg} {\normalfont (Regularity)}
        $\nabla v$ belongs to $T^p_{\beta+{1}/{2}}$ with equivalence
        \begin{equation}
            \label{e:hc_reg_nabla-v}
            \|\nabla v\|_{T^p_{\beta+{1}/{2}}} \eqsim \|v_0\|_{\DotH^{2\beta+1,p}}.
        \end{equation}
        Moreover, if $-1<\beta<-1/2$, then $v$ lies in $T^p_{\beta+1}$ with
        \begin{equation}
            \label{e:hc_reg_v}
            \|\nabla v\|_{T^p_{\beta+{1}/{2}}} \eqsim  \|v\|_{T^p_{\beta+1}} \eqsim \|v_0\|_{\DotH^{2\beta+1,p}}.
        \end{equation}

        \item \label{item:hc_EL_formula} {\normalfont (Explicit formulae)}
        It holds that
        \begin{align}
            \label{e:hc_EL_formula_RL}
            v 
            &= \cE_{-\Delta}(v_0) - \cR^L_{1/2}((A-\bI) \nabla \cE_{-\Delta}(v_0)) \\
            \label{e:hc_EL_formula_RDelta}
            &= \cE_{-\Delta}(v_0)+\cR^{-\Delta}_{1/2}((A-\bI)\nabla v),
        \end{align}
        where $\cE_{-\Delta}$ is the heat extension given in Section \ref{ssec:RW}, and $\cR^{L}_{1/2}, \cR^{-\Delta}_{1/2}$ are given by Theorem \ref{thm:lions_ext}.

        \item \label{item:hc_EL_cont} {\normalfont (Continuity)}
       	$v$ belongs to $C([0,\infty);\scrS')$ with $v(0)=v_0$. 
        
        \item \label{item:hc_EL_strong_cont} {\normalfont (Strong continuity)}
        When $p_-(2\beta+1,L)<p<p_+(2\beta+1,L)$, $v$ also belongs to $C^\infty((0,\infty);\DotH^{2\beta+1,p}) \cap C_0([0,\infty);\DotH^{2\beta+1,p})$ with
        \[ \sup_{t>0} \|v(t)\|_{\DotH^{2\beta+1,p}} \eqsim \|v_0\|_{\DotH^{2\beta+1,p}}. \]

        \item \label{item:hc_EL_sol}
        $v$ is a global weak solution to \eqref{e:hc} with initial data $v_0$.
    \end{enumerate}
\end{theorem}

The proof is deferred to Section \ref{ssec:hc_pf}. Let us give some remarks.

\begin{remark}
    When $\beta > -1/2$, the equivalence \eqref{e:hc_reg_v} fails:  \cite[Theorem 4.3]{Auscher-Hou2023_SIOTent} shows that any weak solution $v \in T^p_{\beta+1}$ to the equation $\partial_t v - \Div(A\nabla v)=0$ must be zero when $p_L^{\,\flat}(\beta) < p \le \infty$, where $p_L^{\, \flat}(\beta)$ is given by \eqref{e:pL^flat(beta)}. 
	
    For $\beta=-1/2$, it also fails. An alternative of the equivalence \eqref{e:hc_reg_v} can be achieved with the Kenig--Pipher space $X^p$ as
    \[ \begin{cases}
    \|v\|_{X^p} \eqsim \|\nabla v\|_{T^p_0} \eqsim \|v_0\|_{L^p} & \text{ if } p_-(L)=p_-(0,L)<p<\infty \\ 
    \|v\|_{X^\infty} \eqsim \|v_0\|_{L^\infty},  & \text{ if } p=\infty \\
     \|\nabla v\|_{T^\infty_0} \eqsim \|v_0\|_{\bmo} & \text{ if } p=\infty 
    \end{cases}. \]
    See \cite[Theorem 7.6, Corollary 5.5 \& 5.10]{Auscher-Monniaux-Portal2019Lp} and \cite[Theorem 7.6]{Zaton2020wp}.
\end{remark}

\subsection{$L$-adapted Hardy--Sobolev spaces}
\label{ssec:L-HS}
Our main tool is  $L$-adapted Hardy--Sobolev spaces. We give a brief review, and the reader can refer to \cite[Chapter 4]{Amenta-Auscher2016EllipticBVP} and \cite[\S 8.2]{Auscher-Egert2023OpAdp} for details. See also \cite[Chapter 10]{Hytonen-NVW2017BanachSpaces_II} for definitions and basic facts of sectorial operators and $H^\infty$-calculus.

Denote by $S_\mu^+$ the set $\{z \in \bC \setminus \{0\}:|\arg(z)|<\mu\}$ for $0<\mu<\pi$. The class $\Psi^\infty_\infty(S_\mu^+)$ consists of holomorphic functions $\psi:S_\mu^+ \rightarrow \bC$ fulfilling that for any $\sigma \in \bR$, there exists a constant $C>0$ so that
\[ |\psi(z)| \le C|z|^\sigma. \]
The class $H^\infty(S_\mu^+)$ consists of bounded holomorphic functions on $S_\mu^+$.

It is well-known that $L$ is a sectorial operator of angle $\omega$ for some $\omega \in [0,\pi/2)$ and has bounded $H^\infty(S_\mu^+)$-calculus for any $\mu \in (\omega,\pi)$. For any $\psi \in H^\infty(S_\mu^+)$, define the operator $\cQ_{\psi,L}: L^2 \rightarrow L^\infty((0,\infty);L^2)$ by
\[ \cQ_{\psi,L}(f)(t) := \psi(t L) f. \]
Note that for $\psi(z)=e^{-z}$, $\cQ_{\psi,L}$ identifies with $\cE_L$.

\begin{definition}[$L$-adapted Hardy--Sobolev spaces]
    \label{def:Ladp}
    Let $s \in \bR$, $0<p<\infty$, and $\omega < \mu < \pi$. Let $\psi$ be a non-zero function in $H^\infty(S_\mu^+)$. The \textit{$L$-adapted Hardy--Sobolev space} $\bH^{s,p}_{L,\psi}$ consists of $f \in L^2$  so that $\cQ_{\psi,L}(f)$ lies in $T^p_{(s+1)/2} \cap L^2_{1/2}(\bR^{1+n}_+)$, equipped with the (quasi-)norm
    \[ \|f\|_{\bH^{s,p}_{L,\psi}} := \|\cQ_{\psi,L}(f)\|_{T^p_{(s+1)/2}}. \]
\end{definition}

\textit{A priori}, this space depends on the choice of $\psi$. For each $(s,p)$, there is a subclass of bounded holomorphic functions, called \textit{admissible functions for $\bH^{s,p}_{L}$}, for which the spaces $\bH^{s,p}_{L,\psi}$ agree as sets with mutually equivalent (quasi-)norms. We denote that space by $\bH^{s,p}_{L}$.
All non-zero $\psi \in \Psi^\infty_\infty(S_\mu^+)$ are admissible functions for $\bH^{s,p}_{L}$. See \cite[\S 8.2]{Auscher-Egert2023OpAdp} for some explicit classes of admissible functions for $\bH^{s,p}_{L}$. For example, $\psi(z)=e^{-z}$ is admissible for $p\le 2$ and $s<0$.

The spaces $\bH^{s,p}_{L}$ are not complete for the defining (quasi-)norms. Aside from abstract completion, we are interested in realizing the completion as a subspace of tempered distributions. The next result provides a range of $(s,p)$, called the identification range, within which a completion in fact equals to $\DotH^{s,p}$. When $0 \le s \le 1$, sharpness of this range is shown in \cite[\S 19.1]{Auscher-Egert2023OpAdp}. We describe the range for $-1 \le s \le 0$, and also exhibit a possible extra range where a completion exists in $\DotH^{s,p}$ without knowing equality. This fact will be useful when we come to uniqueness.

We call \textit{identification range} of $L$ the set 
\[ \scrI_L :=\{(s,p)\in [-1,1] \times (0,\infty)\, : \,p_-(s,L) < p < p_+(s,L)\}. \]

\begin{prop}[Completion of $\bH^{s,p}_L$]
    \label{prop:Ladp-id}
    Let $-1 \le s \le 1$ and $0<p < \infty$. 
    \begin{enumerate}[label=\normalfont(\roman*)]
        \item \label{item:Ladp-id} {\normalfont (Identification)}
        If $(s,p) \in \scrI_L$, then $\bH^{s,p}_L$ agrees with $\DotH^{s,p} \cap L^2$ with equivalent (quasi-)norms
        \[ \|f\|_{\bH^{s,p}_L} \eqsim \|f\|_{\DotH^{s,p}}. \] 
        In particular, $\DotH^{s,p}$ is a completion of $\bH^{s,p}_L$.
        
        \item \label{item:Ladp-embed}{\normalfont (Extra embedding)} 
        Suppose $p_-(L)<1$. Let $-1<s<0$ and $1<p<2$. Then $\bH^{s,p}_L \subset \DotH^{s,p} $ with
        \[ \|f\|_{\DotH^{s,p}} \lesssim  \|f\|_{\bH^{s,p}_L}.\] 
        If, moreover, 
        \begin{equation}
            \label{e:DotHsp_L_embed_condition}
            s>-n\left( \frac{1}{p_-(L)}-1 \right) - \frac{q_+(L^\ast)}{p'} \left( 1-n\left( \frac{1}{p_-(L)}-1 \right) \right).
        \end{equation}
        then the closure of $\bH^{s,p}_L$ in $\DotH^{s,p}$ with respect to the norm $\|\cdot\|_{\bH^{s,p}_L}$ is a completion of $\bH^{s,p}_L$.  
    \end{enumerate}
\end{prop}

The proof is deferred to Appendix \ref{sec:ap_pf_prop_Ladp-id}. 
Of course, the extra embedding range is only interesting when $(s,p) \notin \scrI_L$. In fact, setting $s=2\beta+1$, when $p_-(L)<1$ and $\beta<\beta(L)$, the condition \eqref{e:DotHsp_L_embed_condition} exactly means $p>\phcl{L}$, so we observe from Figure~\ref{fig:exponents} that such range can be non-empty. 

\begin{lemma}[Semigroup extension of $\cE_L$]
    \label{lemma:EL_ext_semigp}
    Let $s\in \bR$ and $0<p\le \infty$. Then $(e^{-tL})$ is a continuous bounded semigroup on $\bH^{s,p}_L$. In particular, for any $f \in \bH^{s,p}_L$, $\cE_L(f)$ lies in $C^\infty((0,\infty);\bH_{L}^{s,p}) \cap C_0([0,\infty);\bH_{L}^{s,p})$ 
    \footnote{For $p=\infty$, $\bH^{s,\infty}_L$ is contained in the dual of $\bH^{-s,1}_{L^\ast}$ by $L^2(\bR^n)$-inner product. Here, it is equipped with the induced weak*-topology. }
    with
    \[ \sup_{t \ge 0} \|\cE_L(f)(t)\|_{\bH_{L}^{s,p}} = \sup_{t \ge 0} \|e^{-tL}f\|_{\bH_{L}^{s,p}} \eqsim \|f\|_{\bH_{L}^{s,p}}. \]
    
    Restricting to the identification range $(s,p)\in \scrI_L$, $(e^{-tL})$ extends to a continuous bounded semigroup on $\DotH^{s,p}$, and for any $f \in \DotH^{s,p}$,  
    \[ \cE_L(f)\in C^\infty((0,\infty);\DotH^{s,p}) \cap C_0([0,\infty);\DotH^{s,p}) \] 
    with
    \[ \sup_{t>0} \|\cE_L(f)(t)\|_{\DotH^{s,p}} \eqsim \|f\|_{\DotH^{s,p}}. \]
\end{lemma}

\begin{proof} 
    The first point is a general result. See again \cite[\S 8.2]{Auscher-Egert2023OpAdp}  applied to $T=L$   and in particular, \cite[Proposition 8.3]{Auscher-Egert2023OpAdp} replacing $\sqrt{T}$ by $T$. The second point follows from the first by Proposition~\ref{prop:Ladp-id} \ref{item:Ladp-id}. 
\end{proof}

\subsection{Proof of Theorem \ref{thm:hc_ext_EL}}
\label{ssec:hc_pf}
We use the formula of $\cE_L$ shown in Corollary \ref{cor:formula_L2} \ref{item:formula_L2_EL} to construct its extension. Define the operator $\cT_{1/2}$ from $L^2$ to $L^2_{\loc}(\bR^{1+n}_+)$ by
\begin{equation}
    \label{e:cT_1/2_cR^L-heat}
    \cT_{1/2}(f) = - \cR_{1/2}^L ((A-\bI) \nabla \cE_{-\Delta} (f)).
\end{equation}
We first extend $\cT_{1/2}$ to $\DotH^{2\beta+1,p}$.

\begin{lemma}[Extension of $\cT_{1/2}$]
    \label{lemma:cT1/2}
    Let $-1<\beta<0$ and $\phcl{L}<p \le \infty$. Then for any $f \in \DotH^{2\beta+1,p} \cap L^2$, 
    \[ \|\cT_{1/2}(f)\|_{T^p_{\beta+1}}+      \|\nabla \cT_{1/2}(f)\|_{ T^p_{\beta+{1}/{2}} } \lesssim \|f\|_{\DotH^{2\beta+1,p}}. \]
    Hence, $\cT_{1/2}$ extends by density (or weak*-density if $p=\infty$) to a bounded operator from $\DotH^{2\beta+1,p}$ to $T^p_{\beta+1}$ with the above estimate. We use the same notation for the extension.
\end{lemma}

\begin{proof}
    We just need to show the inequality. To this end, fix $f \in \DotH^{2\beta+1,p} \cap L^2$. The formula \eqref{e:cT_1/2_cR^L-heat}, together with estimates from Theorem \ref{thm:lions_ext} \ref{item:lions_reg} applied to $\cR^L_{1/2}$, yields that
    \[ \|\cT_{1/2}(f)\|_{T^p_{\beta+1}} + \|\nabla \cT_{1/2}(f)\|_{T^p_{\beta+{1}/{2}}} \lesssim \|\nabla \cE_{-\Delta}(f)\|_{T^p_{\beta+{1}/{2}}} \lesssim \|f\|_{\DotH^{2\beta+1,p}}. \]
    The last inequality follows from Corollary \ref{cor:RW_grad} \ref{item:rw_grad_nabla-HeatExt(f)<f}. 
\end{proof}

\begin{proof}[Proof of Theorem \ref{thm:hc_ext_EL}]  
    Let $-1<\beta<0$ and $\phcl{L} < p \le \infty$. We define the extension $\cE_L$ by
    \begin{equation}
        \label{e:def_cEL_Hsp}
        \cE_L(v_0) := \cE_{-\Delta}(v_0) + \cT_{1/2}(v_0),  \quad \forall v_0 \in \DotH^{2\beta+1,p},
    \end{equation}
    where $\cT_{1/2}$ is given by Lemma \ref{lemma:cT1/2}. 
    
    Let us verify the properties for $v:=\cE_L(v_0)$. As $\DotH^{2\beta+1,p}$ is contained in $\scrS'$, we get $\cE_{-\Delta}(v_0) \in C^\infty(\bR^{1+n}_+)$. Also, Lemma \ref{lemma:cT1/2} says $\cT_{1/2}(v_0) \in T^p_{\beta+1}$ and $\nabla \cT_{1/2}(v_0) \in T^p_{\beta+{1}/{2}}$. In particular, it implies $\cT_{1/2}(v_0) \in L^2_{\loc}((0,\infty);W^{1,2}_{\loc})$, and so does $v$.
    
    First consider \ref{item:hc_EL_reg}. The inequality ``$\lesssim$" in  \eqref{e:hc_reg_nabla-v} follows directly  from Corollary \ref{cor:RW_grad} \ref{item:rw_grad_nabla-HeatExt(f)<f} and Lemma \ref{lemma:cT1/2} as
    \[ \|\nabla v\|_{T^p_{\beta+{1}/{2}}} \le \|\nabla \cE_{-\Delta}(v_0)\|_{T^p_{\beta+{1}/{2}}} + \|\nabla \cT_{1/2}(v_0)\|_{T^p_{\beta+{1}/{2}}} \lesssim \|v_0\|_{\DotH^{2\beta+1,p}}. \]
    We postpone the proof of the converse inequality ``$\gtrsim$" in \eqref{e:hc_reg_nabla-v} and \eqref{e:hc_reg_v} to the end of the proof.

    Next, the formulae in \ref{item:hc_EL_formula} hold when $v_{0} \in \DotH^{2\beta+1,p} \cap L^2$ by Corollary \ref{cor:formula_L2}, and the above bounds of the terms involved allow us to use density (or weak*-density) to extend them to all $v_0 \in \DotH^{s,p}$, valued in $L^2_{\loc}(\bR^{n+1}_{+})$.

    To prove \ref{item:hc_EL_cont}, we use the formula \eqref{e:hc_EL_formula_RDelta} and prove that each term has the desired regularity. As $\DotH^{2\beta+1,p}$ is contained in $\scrS'$, we have $\cE_{-\Delta}(v_0) \in C([0,\infty);\scrS')$. Moreover, as $\phcl{-\Delta} \le \phcl{L} <p$ and $\nabla v \in T^p_{\beta+{1}/{2}}$, we deduce from Theorem \ref{thm:lions_ext} \ref{item:lions_cont} that $\cR^{-\Delta}_{1/2}((A-\bI) \nabla v) \in C([0,\infty);\scrS')$ as wanted.

    Property \ref{item:hc_EL_strong_cont} is a direct consequence of uniqueness of extensions by density: the extension defined by \eqref{e:def_cEL_Hsp} agrees with the one given by Lemma \ref{lemma:EL_ext_semigp}, which has the desired regularity properties.

    For \ref{item:hc_EL_sol},  note that $\cE_{-\Delta}(v_0)$ is a global weak solution to the heat equation, and Theorem \ref{thm:lions_ext} yields that $w:=\cR^{-\Delta}_{1/2}((A-\bI)\nabla v)$ is a global weak solution to $\partial_t w - \Delta w = \Div((A-\bI)\nabla v)$. Therefore, we infer from formula \eqref{e:hc_EL_formula_RDelta} that $v$ is a global weak solution to $\partial_t v - \Delta v = \Div((A-\bI)\nabla v)$, that is $\partial_t v - \Div(A\nabla v)=0$. Meanwhile, \ref{item:hc_EL_cont} shows that $v(t)$ converges to $v_0$ in $\scrS'$ as $t \to 0$. We hence conclude that $v$ is a global weak solution to \eqref{e:hc} with initial data $v_0$.
    
    Let us now prove the rest of \ref{item:hc_EL_reg}.  We begin with  the inequality ``$\gtrsim$" in \eqref{e:hc_reg_nabla-v}. We just proved in \ref{item:hc_EL_sol} that $v$ is a global weak solution to the equation $\partial_t v -\Div(A\nabla v)=0$ with $\nabla v \in T^p_{\beta+{1}/{2}}$, and that $v(t)$ converges to $v_0$ in $\scrS'$ as $t \to 0$, so $v_0$ agrees with the trace $g$ of $v$ given by Proposition \ref{prop:trace} \ref{item:trace_-1<beta<0} in $\DotH^{2\beta+1,p}$. In particular, we get $\|v_0\|_{\DotH^{2\beta+1,p}} \lesssim \|\nabla v\|_{ T^p_{\beta+{1}/{2}} }$ as desired. This proves \eqref{e:hc_reg_nabla-v}.
    
    Next, to see \eqref{e:hc_reg_v}, we only need to show that 
    $ \|v\|_{T^p_{\beta+1}} \lesssim \|v_0\|_{\DotH^{2\beta+1,p}}$, thanks to \eqref{e:hc_reg_nabla-v} and Corollary \ref{cor:energyinequalitytent space}. This inequality is a direct consequence of \eqref{e:hc_EL_formula_RDelta}, Theorem \ref{thm:RW_sol} \ref{item:rw_HeatExt(f)<f}, and \eqref{e:hc_reg_nabla-v} (for the Laplacian) as
    \begin{align*}
        \|v\|_{T^p_{\beta+1}} 
        &\lesssim \|\cE_{-\Delta}(v_0)\|_{T^p_{\beta+1}} + \|\cR^{-\Delta}_{1/2}((A-\bI)\nabla v)\|_{T^p_{\beta+1} } \\
        &\lesssim \|\cE_{-\Delta}(v_0)\|_{T^p_{\beta+1}} + \|\nabla v\|_{T^p_{\beta+{1}/{2}} } \lesssim \|v_0\|_{\DotH^{2\beta+1,p}}.
    \end{align*}   
    This shows \ref{item:hc_EL_reg} and the proof is complete. 
\end{proof}
\section{Uniqueness and representation}
\label{sec:ur}

In this section, we prove uniqueness of weak solutions to the homogeneous Cauchy problem and Lions' equation as asserted in Theorems \ref{thm:wp_hc} and \ref{thm:wp_lions}. We also provide the proof of Theorem \ref{thm:rep}.

\subsection{Uniqueness}
\label{ssec:unique}
The main theorem on uniqueness is as follows. 

\begin{theorem}[Uniqueness]
    \label{thm:unique}
    Let $\beta>-1$ and $\phcl{L} < p \le \infty$. Let $u$ be a global weak solution to the initial value problem
    \[ \begin{cases}
        \partial_t u-\Div(A\nabla u) = 0, \\ 
        u(0)=0
    \end{cases}, \]
    with $\nabla u \in T^p_{\beta+{1}/{2}}$. Then $u=0$.
\end{theorem}

\begin{remark}
    Let us mention that \cite[Theorem 4.3]{Auscher-Hou2023_SIOTent} proves uniqueness for $\beta>-1/2$ and $p_{L}^{\,\flat}(\beta)<p \le \infty$ (where $p_L^{\,\flat}(\beta)$ is defined in \eqref{e:pL^flat(beta)} and equals  $\phcl{L}$ when $p_{-}(L)\ge 1$), requiring that $u$ lies in $T^p_{\beta+1}$ (but with no initial data needed). However, recall that Corollary \ref{cor:energyinequalitytent space} asserts the inequality $\|\nabla v\|_{T^p_{\beta+{1}/{2}} } \lesssim \|v\|_{T^p_{\beta+1}}$ holds for any global weak solution $v$ to the equation $\partial_t v - \Div(A\nabla v)=0$. Therefore, we improve the class of uniqueness in this range of $\beta$ with more values of $p$. 
\end{remark}

\begin{proof}[Proof of Theorem~\ref{thm:unique}]
Let $u$ be a global weak solution to  $\partial_t u - \Div(A\nabla u) = 0$ with $\nabla u \in T^p_{\beta+{1}/{2}}$ and null initial data. Thanks to Lemma \ref{lemma:unique_embed}, we may assume
$p>\max\{\phcl{L},1\}$ in what follows. Let us begin by some important common facts. 

First, as $p>\phcl{L} \ge \frac{n}{n+2\beta+2}$ and $u(0)=0$, Proposition \ref{prop:trace} yields
\begin{equation}
    \label{e:unique_rep}
    u =  \cR^{-\Delta}_{1/2}((A-\bI)\nabla u).
\end{equation}
Hence, Theorem~\ref{thm:lions_ext} for $\cR^{-\Delta}_{1/2}$ applies:  $u \in T^p_{\beta+{1}}$ and $u(s)$ tends to 0 as $s \to 0$ in $\scrS'$ and in various finer topology depending on $\beta$ and $p$.

Second, Lemma \ref{lemma:size_nabla_u_tent} gives local $L^2$-norm estimates of $u$  allowing us to invoke \cite[Theorem 5.1]{Auscher-Monniaux-Portal2019Lp} to get an interior representation of $u$ as follows: For $0<s<t<\infty$ and $h \in \Cc(\bR^n)$, it holds that
\begin{equation}
    \label{e:homo-id}
    \int_{\bR^n} u(t,x) \ovh(x) dx = \int_{\bR^n} u(s,x) \overline{e^{-(t-s)L^\ast} h}(x) dx.
\end{equation}

Now, fix $t>0$ and $h \in \Cc(\bR^n)$. We wish to let $s\to 0$, yielding $u(t)=0$ in  $\scrD'(\bR^n)$. To take advantage of the different modes of convergence established in Theorem~\ref{thm:lions_ext} \ref{item:lions_cont}, we split the argument in two cases. 

\

\paragraph{Case 1: $\beta>-1/2$}
As $p>\max\{\phcl{L},1\} \ge p_L^{\,\flat}(\beta)$ and $u \in T^p_{\beta+1}$, we can use \cite[Theorem 4.3]{Auscher-Hou2023_SIOTent} to conclude $u=0$ (the proof there uses \eqref{e:homo-id} as well). \\

\paragraph{Case 2: $-1<\beta \le -1/2$} 
We begin with $p\ge2$ and then  $\max\{\phcl{L}, 1\}<p\le 2$, which is further split into two sub-cases.
\\

\subparagraph{Case 2(a): $2 \le p \le \infty$}
Pick $\delta>0$. We use the slice spaces $E^{-1,p}_{\delta}$ and $E^{1,p'}_\delta$ introduced in Section \ref{ssec:tent-slice}. Recall that the $L^2(\bR^n)$-inner product realizes $E^{-1,p}_\delta$ as the dual of $E^{1,p'}_\delta$. The formula \eqref{e:unique_rep} and Theorem \ref{thm:lions_ext} \ref{item:lions_cont} for $\cR^{-\Delta}_{1/2}$ imply
\[ \lim_{s \to 0} u(s) = 0 \quad \text{ in } E^{-1,p}_{\delta}. \]
Moreover, we claim that 
\begin{equation}
    \label{e:unique_p-(L)>1_p>2_e-tL*h}
    \lim_{s \to 0} e^{-(t-s)L^\ast} h = e^{-tL^\ast} h \quad \text{ in } E^{1,p'}_{\delta}.
\end{equation}
Then taking limit of the right-hand integral in \eqref{e:homo-id} yields $u(t)=0$.  

To prove \eqref{e:unique_p-(L)>1_p>2_e-tL*h}, pick $\lambda \in (0,1)$ and $R>\delta>0$ so that $\supp(h) \subset B:=B(0,R)$ and for any $|x|>R$, $D(x):=\dist(B(x,\delta),\supp(h)) \ge \lambda |x|$. If $|x| \le R$, we have
\[ \|\I_{B(x,\delta)} \nabla e^{-tL^\ast} h\|_2 \le \| \nabla e^{-tL^\ast} h\|_2 \lesssim t^{-1/2} \|h\|_2. \]
If $|x|>R$, the exponential $L^2-L^2$ off-diagonal estimates of $(t^{1/2} \nabla e^{-tL^\ast})$ yield there exists $\gamma>0$ so that
\[ \|\I_{B(x,\delta)} \nabla e^{-tL^\ast} h\|_2 \lesssim t^{-1/2} e^{-\frac{\gamma D(x)^2}{t}} \|h\|_2 \lesssim t^{-1/2} e^{-\frac{\gamma\lambda^2 |x|^2}{t}} \|h\|_2. \]
We hence get 
\[ \sup_{0 \le s \le t/2} \|\I_{B(x,\delta)} \nabla e^{-(t-s)L^\ast} h\|_2 \lesssim t^{-\frac{1}{2}} ( \I_{B}(x) + \I_{B^c}(x) e^{-\frac{\gamma\lambda^2 |x|^2}{t}} ) \|h\|_2. \]
Note that the function on the right-hand side lies in $L^{p'}$, so $(\nabla e^{-(t-s)L^\ast} h)_{0 \le s \le t/2}$ form a uniformly bounded set in $E^{p'}_\delta$. Moreover, Lebesgue's dominated convergence theorem and continuity of $(\nabla e^{-tL^\ast})$ on $L^2$ imply  that $\nabla e^{-(t-s)L^\ast} h$ converges to $\nabla e^{-tL^\ast} h$ in $E^{p'}_\delta$ as $s \to 0$. Hence equivalently, $e^{-(t-s)L^\ast} h$ converges to $e^{-tL^\ast} h$ in $E^{1,p'}_\delta$ as $s \to 0$. This proves \eqref{e:unique_p-(L)>1_p>2_e-tL*h}. \\

\subparagraph{Case 2(b): $p_-(2\beta+1,L)<p<2$}
Recall that we also assume $p>1$, and we know that $\phcl{-\Delta} \le p_-(2\beta+1,L)$, so in particular, we have  
\[ \max\{\phcl{-\Delta},1\} \le \max\{p_-(2\beta+1,L), 1\} < p < 2.\]
Using the formula \eqref{e:unique_rep} and Theorem \ref{thm:lions_ext} \ref{item:lions_cont} for $\cR^{-\Delta}_{1/2}$, we get
\[ \lim_{s \to 0} u(s) = 0 \quad \text{ in } \DotH^{2\beta+1,p}. \]
On the other hand, as  $h \in \DotH^{-(2\beta+1),p'} \cap L^2$ and  
\[ 2<p'<(\max\{p_-(2\beta+1,L), 1\})'=p_+(-(2\beta+1),L^\ast), \]
we infer from Theorem \ref{thm:hc_ext_EL} \ref{item:hc_EL_strong_cont} that 
\[ \lim_{s \to 0} e^{-(t-s)L^\ast} h = \lim_{s \to 0} \cE_{L^\ast}(h)(t-s) = e^{-tL^\ast} h \quad \text{ in } \DotH^{-(2\beta+1),p'}. \]
Realizing the right-hand integral of  \eqref{e:homo-id} in the sense of duality for $\DotH^{2\beta+1,p}$ and $\DotH^{-(2\beta+1),p'}$, we obtain that it tends to 0 as $s\to 0$. \\

\subparagraph{Case 2(c): $\phcl{L} < p \le p_-(2\beta+1,L)$}
This case only occurs when $p_-(L)<1$. Theorem \ref{thm:lions_ext} \ref{item:lions_cont} still gives us a limit $u(s)\to 0$ in $\DotH^{2\beta+1,p}$ but now we do not know whether $e^{-(t-s)L^\ast} h$ tends to $e^{-tL^\ast} h$ in the dual space. However, the extra embedding in Proposition \ref{prop:Ladp-id} \ref{item:Ladp-embed} allows us to enhance the convergence of $u(s)$. 

\begin{lemma}[Null limit of $u(s)$ in  $\bH^{2\beta+1,p}_L$] 
    \label{lemma:unique_p-(L)<1_u_HqL} 
    In the range determined by this case, for all $t>0$, $u(t)$ lies in $\bH^{2\beta+1,p}_L$ with a uniform bound. Moreover, 
    \[ \lim_{t \to 0} u(t) = 0 \quad \text{ in } \bH^{2\beta+1,p}_L. \]
\end{lemma}

The proof is presented right after. Admitting this lemma, let us show $u=0$. Since $u(s) \in \bH^{2\beta+1,p}_L$ for all $s>0$, we have $u(s) \in L^2$ and we get from \eqref{e:homo-id} that $u(t)=e^{-(t-s)L}u(s)$ for $t \ge s$. By Lemma~\ref{lemma:EL_ext_semigp}, we get 
\[ \sup_{t\ge s} \|u(t)\|_{\bH^{2\beta+1,p}_L} \lesssim \|u(s)\|_{\bH^{2\beta+1,p}_L}, \]
with a uniform bound in $s$. As the right-hand side tends to $0$ as $s\to 0$, we conclude that $\sup_{t>0} \|u(t)\|_{\bH^{2\beta+1,p}_L}=0$, so $u=0$.
\end{proof}

\begin{proof}[Proof of Lemma \ref{lemma:unique_p-(L)<1_u_HqL}]
The proof is divided into 4 steps.\\

\paragraph{Step 1: Regularity of $u$}
We prove $u \in C^\infty((0,\infty);L^p \cap L^2)$. Since $1<p<2$, as proved in \cite[(4.19)]{Auscher-Hou2023_SIOTent}, for a.e. $s>0$, we have $u(s) \in L^p$ with
\[ \|u(s)\|_p \lesssim s^{\beta+1/2} \|u\|_{T^p_{\beta+1}}.\]
For such an $s$, $L^p$-boundedeness of $(e^{-tL})$ yields $e^{-(t-s)L} u(s) \in L^p$ for all $t \ge s$, and it follows from \eqref{e:homo-id} that it equals to $u(t)$. Thus we obtain $u(t)=e^{-(t-s)L} u(s)$ for any $t \ge s>0$. Applying analyticity of $(e^{-tL})$ on $L^p$, we get $u \in C^\infty((0,\infty);L^p)$. As $e^{-tL}$ also maps $L^p$ to $L^2$ for any $t>0$, we have $u \in C^\infty((0,\infty);L^2)$ as desired. \\

\paragraph{Step 2: A key estimate}
For any $k \ge 1$, we show that
\begin{equation}
    \label{e:unique_p-(L)<1_tk-partial_tu}
    \|t^k (\partial_t^k u)(t)\|_{T^p_{\beta+1}} \lesssim_{k} \|u\|_{T^p_{\beta+1}}.
\end{equation}
We prove it for $k=1$, and iteration concludes the argument. Observe that by Step 1, we have that for any $t>0$ and $\tau>0$,
\begin{equation}
    \label{e:unique_p-(L)<1_dtau(u(t+tau))}
    (\partial_t u)(t+\tau) = \partial_\tau (u(t+\tau)) = \partial_\tau (e^{-\tau L} u(t)) = - Le^{-\tau L} u(t).
\end{equation}
In particular, pick $\tau=t$, we get $(\partial_t u)(2t) = - Le^{-tL} u(t)$, and hence,
\begin{align*}
    \|t \partial_t u \|_{T^p_{\beta+1}} 
    &\eqsim \left( \int_{\bR^n} \left( \int_0^\infty \fint_{B(x,(2s)^{1/2})} |s^{-(\beta+1)} s(\partial_t u)(2s,y)|^2\, dsdy \right)^{p/2} dx \right)^{1/p} \\
    &= \left( \int_{\bR^n} \left( \int_0^\infty \fint_{B(x,(2s)^{1/2})} |s^{-(\beta+1)} ( sLe^{-sL} u(s))(y)|^2 \, dsdy \right)^{p/2} dx \right)^{1/p}. 
\end{align*}
Write $B:=B(x,(2s)^{1/2})$, $C_0:=2B$, and $C_j:=2^{j+1} B \setminus 2^j B$ for $j \ge 1$. The $L^2-L^2$ off-diagonal estimates of $(sLe^{-sL})$ yield that
\[ \|\I_B sLe^{-sL} u(s)\|_2 \le \sum_{j \ge 0} \|\I_B sLe^{-sL} \I_{C_j} u(s)\|_2 \lesssim \sum_{j \ge 0} e^{-c2^{2j}} \|\I_{C_j} u(s)\|_2, \]
where $c>0$ is independent of $s$. Applying this on the above computation, we get 
\begin{align*}
    &\|t \partial_t u \|_{T^p_{\beta+1}} \lesssim \sum_{j \ge 0} e^{-c2^{2j}} \left( \int_{\bR^n} \left( \int_0^\infty \int_{C_j} s^{-\frac{n}{2}} |s^{-(\beta+1)} u(s,y)|^2 \, dsdy \right)^{p/2} dx \right)^{1/p} \\
    &\quad \lesssim \sum_{j \ge 0} 2^{j\frac{n}{2}} e^{-c2^{2j}} \left( \int_{\bR^n} \left( \int_0^\infty \fint_{B(x,2^{j+1} (2s)^{1/2})} |s^{-(\beta+1)} u(s,y)|^2 \, dsdy \right)^{p/2} dx \right)^{1/p} \\
    &\quad \lesssim \sum_{j \ge 0} 2^{j(\frac{n}{2}+\frac{n}{p})} e^{-c2^{2j}} \|u\|_{T^p_{\beta+1}} \lesssim \|u\|_{T^p_{\beta+1}}.
\end{align*}
The factor $2^{j\frac{n}{p}}$ in the third inequality comes from change of angle for tent space norms, see \cite[Theorem 1.1]{Auscher2011_Change_angle}. \\

\paragraph{Step 3: Boundedness of $u(t)$ in $\bH^{2\beta+1,p}_L$.}
Now we prove that for any $t>0$, $u(t)$ belongs to $\bH^{2\beta+1,p}_L$ with
\begin{equation}
    \label{e:unique_p-(L)<1_u(t)bdd_HqL}
    \sup_{t>0} \|u(t)\|_{\bH^{2\beta+1,p}_L} \lesssim  \|u\|_{T^p_{\beta+1}}.
\end{equation}
By Step 1, $u(t)$ lies in $L^2$ for any $t>0$, so we just need to show the norm estimate. To this end, pick an integer $k \ge 1$ with $(k-\beta-1)p-\frac{n}{2}>0$. From \cite[\S 8.2]{Auscher-Egert2023OpAdp}, $\psi(z):=z^k e^{-z}$ is a valid admissible function for $\bH^{2\beta+1,p}_L$. So we use \eqref{e:unique_p-(L)<1_tk-partial_tu}, \eqref{e:unique_p-(L)<1_dtau(u(t+tau))} in Step 2, and the identity $\partial_\tau^k \big( u(t+\tau) \big) = (\partial_t^k u)(t+\tau)$ to deduce that for any $t>0$,  
\begin{align*}
    &\|u(t)\|_{\bH^{2\beta+1,p}_L} \\
    &\quad \eqsim \left( \int_{\bR^n} \left( \int_0^\infty \fint_{B(x,\tau^{1/2})} |\tau^{-(\beta+1)} \left( (\tau L)^k e^{-\tau L} u(t) \right)(y)|^2 \, dyd\tau \right)^{p/2} dx \right)^{1/p} \\
    &\quad \eqsim \left( \int_{\bR^n} \left( \int_0^\infty \fint_{B(x,\tau^{1/2})} |\tau^{-(\beta+1)} \tau^k \partial_\tau^k \big( u(t+\tau,y) \big)|^2 \, dyd\tau \right)^{p/2} dx \right)^{1/p} \\
    &\quad \eqsim \left( \int_{\bR^n} \left( \int_t^\infty \int_{B(x,(\sigma-t)^{1/2})} (\sigma-t)^{(k-\beta-1)p-\frac{n}{2}}  |(\partial_t^k u)(\sigma,y)|^2 \, dyd\sigma \right)^{p/2} dx \right)^{1/p} \\
    &\quad \lesssim \|\sigma^k (\partial_t^k u)(\sigma)\|_{T^p_{\beta+1}} \lesssim \|u\|_{T^p_{\beta+1}}.
\end{align*}
The first inequality with implicit constant  independent of $t$ uses $(k-\beta-1)p-\frac{n}{2}>0$. This proves \eqref{e:unique_p-(L)<1_u(t)bdd_HqL}. \\

\paragraph{Step 4: Limit at $t=0$.}
We finish by showing that $u(t)$ tends to 0 in $\bH^{2\beta+1,p}_L$ as $t \to 0$. Suppose $(k-\beta-1)p-\frac{n}{2}>0$. Remark that examination of the argument in Step 2 allows time truncation. More precisely,  if $0<t\le \delta$, then  one finds 
\[ \| \I_{\{\tau<2\delta\}} \, \partial_\tau^{k}(u(t+\tau))\|_{T^p_{\beta+1-k}} \lesssim \| \I_{\{\tau<\delta\}} \,  u\|_{T^p_{\beta+1}}, \]
where the implicit constant does not depend on $t,\delta$. Using Newton--Leibniz formula for $L^2$-valued functions, we get 
\[ \tau^k\partial_\tau^{k}(u(t+\tau))- \tau^k \partial_\tau^{k}(u(t'+\tau))= \frac{1}{\tau}\int_{t'}^t \tau^{k+1}\partial_\tau^{k+1}(u(h+\tau))\, dh. \]
Also, following the argument of Step 2  with the help of Minkowski's integral inequality, we find for $t,t' \in (0, \delta]$,  
\[ \| \I_{\{\tau<2\delta\}} \, (\partial_\tau^{k}(u(t+\tau))-\partial_\tau^{k}(u(t+\tau)))\|_{T^p_{\beta+1-k}} \lesssim \frac {|t-t'|}\delta \|u\|_{T^p_{\beta+1}}, \]
again with implicit constant independent of $t,t', \delta$. This implies that $((\tau,x) \mapsto \partial_\tau^{k}(u(t+\tau,x)))_{t>0}$ is Cauchy in $T^p_{\beta+1-k}$ when $t \to 0$. We deduce as in Step 3 that $(u(t))_{t>0}$ is Cauchy in $\bH^{2\beta+1,p}_L$ when $t\to 0$. As mentioned in the first paragraph of Case 2(c), $u(t)$ tends to 0 in $\DotH^{2\beta+1,p}$. The conditions on $\beta,p$ allow us to apply Proposition \ref{prop:Ladp-id} \ref{item:Ladp-embed}, so we conclude that the limit of $u(t)$ exists and must be zero in $\bH^{2\beta+1,p}_L$. This completes the proof.
\end{proof}

\subsection{Representation}
\label{ssec:rep}
Let us demonstrate Theorem \ref{thm:rep}.
\begin{proof}[Proof of Theorem \ref{thm:rep}]
   Let $u$ be a global weak solution to  $\partial_t u - \Div(A\nabla u)=0$ with $\nabla u \in T^p_{\beta+{1}/{2}}$. As $p>\phcl{L} \ge \frac{n}{n+2\beta+2}$, Proposition \ref{prop:trace} asserts that there exists $u_0 \in \scrS'$ so that
   \[ u = \cE_{-\Delta}(u_0)+\cR^{-\Delta}_{1/2}((A-\bI)\nabla u). \]
   
   In case \ref{item:rep_beta>0}, as $\beta \ge 0$ and $\frac{n}{n+2\beta} \le p \le \infty$, Proposition \ref{prop:trace} \ref{item:trace_beta>0} says $u_0$ is a constant. Hence, $w:=u-\cE_{-\Delta}(u_0)$ is a global weak solution to the Cauchy problem
   \[ \begin{cases}
        \partial_t w-\Div(A\nabla w) = 0, \\ 
        w(0)=0
    \end{cases} \]
   with $\nabla w \in T^p_{\beta+{1}/{2}}$. As $\phcl{L} <p \le \infty$, we invoke Theorem \ref{thm:unique} to get $w=0$, so $u=\cE_{-\Delta}(u_0)$ is a constant as desired.

   In case \ref{item:rep_-1<beta<0}, Proposition \ref{prop:trace} \ref{item:trace_-1<beta<0} shows that there exist $g \in \DotH^{2\beta+1,p}$ and $c \in \bC$ so that $u_0=g+c$. Then we apply Theorem \ref{thm:hc_ext_EL} to $g$ to get
   \[ u = \cE_{-\Delta}(g)+\cR^{-\Delta}_{1/2}((A-\bI)\nabla u)+c = \cE_L(g)+c. \]

   This completes the proof.
\end{proof}
\section{Results for homogeneous Besov spaces}
\label{sec:besov}

One can also study the case where the initial data are taken in homogeneous Besov spaces $\DotB^s_{p,p}$. The definition of $\DotB^s_{p,p}$ can be analogously adapted from Definition \ref{def:Hsp}, as ``realization" of homogeneous Besov spaces defined on $\scrS'/\scrP$. In fact, one can also take the definition of $\DotB^s_{p,p}$ as in \cite[Definition 2.15]{Bahouri-Chemin-Danchin2011_FA}, which does not contain polynomials for any $s$ and $p$. See \cite[Remark 2.26]{Bahouri-Chemin-Danchin2011_FA} for more detailed discussion. 

The counterparts of tent spaces are $Z$-spaces, introduced by \cite{Barton-Mayboroda2016_Zspaces} (with a different notation). For any $p \in (0,\infty)$ and $\beta \in \bR$, the \textit{(parabolic) $Z$-space} $Z^p_\beta$ consists of measurable functions $F$ on $\bR^{1+n}_+$ for which the (quasi-)norm
\[ \|F\|_{Z^p_\beta} := \left( \int_{\bR^{1+n}_+} \left ( \int_{t/2}^t \fint_{B(x,t^{1/2})} |s^{-\beta} F(s,y)|^2 dsdy \right )^{p/2} \frac{dt}{t} dx \right )^{1/p} < \infty. \]
For $p=\infty$, we set
\[ \|F\|_{Z^\infty_\beta} := \sup_{t>0,x\in \bR^n} \left ( \int_{t/2}^t \fint_{B(x,t^{1/2})} |s^{-\beta} F(s,y)|^2 dsdy \right )^{1/2}. \]

The relation between $\DotB^s_{p,p}$ (resp. $Z^p_\beta$) and $\DotH^{s,p}$ (resp. $T^p_\beta$) is given by real interpolation, see \cite[Theorem 2.30]{Amenta-Auscher2016EllipticBVP} and \cite[\S 2.6]{Auscher-Egert2023OpAdp}. Let $0<p_0,p_1 \le \infty$ and $\theta \in (0,1)$. Let $s_0,s_1,\beta_0,\beta_1 \in \bR$ with $s_0 \ne s_1$ and $\beta_0 \ne \beta_1$. Define $s,\beta \in \bR$ and $p \in (0,\infty]$ so that $s:=(1-\theta)s_0+\theta s_1$, $\beta:=(1-\theta)\beta_0+\theta \beta_1$, and $\frac{1}{p} := \frac{1-\theta}{p_0} + \frac{\theta}{p_1}$. Then we have
\begin{equation}
 \label{eq:TZinterpolation} ( \DotH^{s_0,p_0}, \DotH^{s_1,p_1} )_{\theta,p} = \DotB^s_{p,p}, \quad ( T^{p_0}_{\beta_0}, T^{p_1}_{\beta_1} )_{\theta,p} = Z^p_\beta. 
 \end{equation}
Moreover, it has been shown in \cite[Theorem 2.34]{Amenta-Auscher2016EllipticBVP} that
\begin{equation}
    \label{eq:TZTemlbedding}
    T^{p_0}_{\beta_0} \hookrightarrow Z^{p_1}_{\beta_1} \hookrightarrow T^{p_2}_{\beta_2}, 
\end{equation}
if $0<p_0<p_1<p_2 \le \infty$ and $2\beta_0-\frac{n}{p_0} = 2\beta_1-\frac{n}{p_1} = 2\beta_2-\frac{n}{p_2}$.

Let us first consider a special case for $\beta=0$ (or $s=1$) and $p=\infty$. Recall that we say a distribution $g$ belongs to the homogeneous Sobolev space $\DotW^{1,\infty}$ if $\nabla g$ is bounded, and $\|g\|_{\DotW^{1,\infty}} := \|\nabla g\|_{\infty}$. Rademacher's theorem asserts $\DotW^{1,\infty}$ coincides with the set of Lipschitz continuous functions (up to almost everywhere equality).
\begin{prop}
    \label{prop:heat-besov-1-infty}
    \begin{enumerate}[label=\normalfont(\roman*)]
        \item \label{item:heat_besov_1-infty-est}
        Let $g \in \DotW^{1,\infty}$ be a Lipschitz function. Then the function $(t,x) \mapsto \nabla e^{t\Delta} g(x)$ belongs to $Z^\infty_{1/2}$ with
        \[ \|\nabla e^{t\Delta} g\|_{Z^\infty_{1/2}} \eqsim \|g\|_{\DotW^{1,\infty}}. \]
    
        \item \label{item:heat_besov_1-infty-rep}
        Let $u$ be a distributional solution to the heat equation on $\bR^{1+n}_+$ with $\nabla u \in Z^\infty_{1/2}$. Then there exists $u_0 \in \DotW^{1,\infty}$  such that $u(t)=e^{t\Delta} u_0$ for all $t>0$.
    \end{enumerate}
\end{prop}

\begin{proof}
    Observe that the space $Z^\infty_{1/2}$ coincides with the parabolic version of the Kenig--Pipher space $X^\infty$ introduced by \cite{Kenig-Pipher1993Xp}. Then we invoke \cite[Theorem 5.4]{Auscher-Monniaux-Portal2019Lp} to get \ref{item:heat_besov_1-infty-est} as 
    \[ \|\nabla e^{t\Delta} g\|_{Z^\infty_{1/2}} = \|e^{t\Delta} \nabla g\|_{X^\infty} \eqsim \|\nabla g\|_{L^\infty} = \|g\|_{\DotW^{1,\infty}}. \]
    Moreover, it also asserts that any weak solution $G \in X^\infty$ to the heat equation has a trace $g \in L^\infty$ so that $G(t)=e^{t\Delta} g$ for all $t>0$.
    
    To prove \ref{item:heat_besov_1-infty-rep}, we claim that such $u$ also has a distributional limit $u_0 \in \scrS'$  as $t \to 0$, and $u(t)=e^{t\Delta} u_0$ for any $t>0$. Then applying the above assertion to $\nabla u \in X^\infty$ yields $\nabla u_0=g$, so $u_0$ must be (equal almost everywhere to) a Lipschitz function as desired.
    
    The claim follows by a verbatim adaptation of the proof of Proposition \ref{prop:rep_heat}. We just list the main modifications here. In this case, the estimate in Lemma \ref{lemma:size_nabla_u_tent} that yields the size condition becomes: For $0<a<b<\infty$ and $R>1$,
    \[ \int_a^b \int_{B(0,R)} |u|^2 \lesssim_{a,b} R^{3n+2} \left( \|\nabla u\|_{ Z^\infty_{1/2} }^2 + \|u\|_{L^2((a,b) \times B(0,1))}^2 \right). \]
    For the uniform control condition \eqref{e:heat_verify_UC}, it follows from \eqref{e:heat_uc_nabla_u_N2}, which still holds for $N>n+1$ since
    \[ \int_0^2 \int_{\bR^n} |\nabla u| |\nabla \phi| \lesssim \|\nabla u\|_{ Z^\infty_{1/2} } \|\I_{(0,2)} \phi\|_{Z^1_{-1/2}} \lesssim \|\nabla u\|_{ Z^\infty_{1/2} } \cP_N(\phi). \]
    Therefore, applying \cite[Theorem 1.1]{Auscher-Hou2023_RepHeat} again provides the trace $u_0 \in \scrS'$ as wanted. This completes the proof.
\end{proof}

\begin{theorem}[Well-posedness of Cauchy problems of type \eqref{e:ivp_ic} for Besov spaces]
    Let $\beta>-1$, $s \in \bR$, and $0<p \le \infty$. With modifications for $\beta=0$ (or $s=1$) and $p=\infty$ as in Proposition \ref{prop:heat-besov-1-infty} (\textit{i.e.}, the initial data $u_0$ lies in $\DotW^{1,\infty}$), Theorems \ref{thm:heat_Hsp}, \ref{thm:wp_hc}, \ref{thm:rep}, \ref{thm:wp_lions}, and \ref{thm:ic-hc-lions} are all valid in the same range of exponents, when replacing homogeneous Hardy--Sobolev spaces (resp.~tent spaces) by homogeneous Besov spaces (resp.~$Z$-spaces) with the same exponents.
\end{theorem}

\begin{proof} 
	We just provide some key ingredients of the proof. 
	For convenience, we still use the same label of the theorems for their variants in Besov spaces. Interpolation \eqref{eq:TZinterpolation} allows one to prove Theorem \ref{thm:heat_Hsp} \ref{item:heat_tent_est} and the needed boundedness for existence in Theorems \ref{thm:wp_hc}, \ref{thm:wp_lions}, and \ref{thm:ic-hc-lions}. For $\beta=-1/2$ (or $s=0$) and $p=\infty$, one uses Proposition \ref{prop:heat-besov-1-infty}. To prove uniqueness, for $p \ne \infty$, it follows from embedding into tent spaces in \eqref{eq:TZTemlbedding}. For $p=\infty$, following the proof of Theorem \ref{thm:unique}, we may suppose $u=\cR^{-\Delta}_{1/2}((A-\bI)\nabla u)$ (cf. \eqref{e:unique_rep}), so it suffices to show that for any $\tilF\in Z^\infty_{\beta+1/2}$, $\cR^{-\Delta}_{1/2}(\tilF)(t)$ tends to 0 as $t \to 0$ in $E^{-1,\infty}_\delta$ for any $\delta>0$. Using Theorem \ref{thm:lions_ext} \ref{item:lions_formula} and \cite[Corollary 5.6]{Auscher-Hou2023_SIOTent}, we have
	\[ \|\cR^{-\Delta}_{1/2}(\tilF)(t)\|_{E^{-1,\infty}_\delta} \lesssim \|\cL_1^{-\Delta}(\tilF)(t)\|_{E^\infty_\delta} \lesssim t^{\beta+1} \|\tilF\|_{T^\infty_{\beta+1/2}}. \]
	Then interpolation \eqref{eq:TZinterpolation} yields that $\|\cR^{-\Delta}_{1/2}(\tilF)(t)\|_{E^{-1,\infty}_\delta} \lesssim t^{\beta+1} \|\tilF\|_{Z^\infty_{\beta+1/2}}$ holds for $\beta>-1$, which implies the desired zero limit when $t \to 0$.
	
	Next, we prove Theorem \ref{thm:heat_Hsp} \ref{item:heat_rep}, from which Theorem \ref{thm:rep} also follows, using the same argument in Section \ref{ssec:rep}. Existence of the trace again follows from a verbatim adaptation of the proof of Proposition \ref{prop:rep_heat}, as is shown in Proposition \ref{prop:heat-besov-1-infty} \ref{item:heat_besov_1-infty-rep}. Then we show the properties of the trace. For \ref{item:heat_rep_s>1}, when $s \ge 0$, $\frac{n}{n+s} \le p \le \infty$ but $(s,p) \ne (0,\infty)$, one can imitate the proof of Corollary \ref{cor:RW_grad} \ref{item:rw_f<HeatExt(f)}, using interpolation and direct computation, to get that for $0<a<1$ and $\phi \in \Cc(\bR^n)$,
	\[ \fint_a^{2a} | \langle \nabla e^{t\Delta} u_0,\phi \rangle| \, dt \lesssim_\phi \|\I_{(a,2a)} \nabla \cE_{-\Delta}(u_0)\|_{Z^p_{s/2}}, \]
	which tends to 0 as $a \to 0$, so $\nabla u_0=0$ and $u_0$ is a constant. For $(s,p)=(0,\infty)$, this is Proposition \ref{prop:heat-besov-1-infty} \ref{item:heat_besov_1-infty-rep}. This hence proves \ref{item:rw_s>0_f=0}. 
	
	For \ref{item:heat_rep_s<1}, when $s<1$ and $\frac{n}{n+s+1}<p \le \infty$, we apply interpolation for the map $u \mapsto u_0$ to obtain $u_0 \in \DotB^s_{p,p}$ with $\|u_0\|_{\DotB^s_{p,p}} \lesssim \|u\|_{Z^p_{s/2}}$. This completes the proof.
\end{proof}
\section{An endpoint case \texorpdfstring{$\beta=-1$}{beta=-1}}
\label{sec:beta=-1}

In this section, we show existence of global weak solutions to the homogeneous Cauchy problem for $\beta=-1$.

\begin{prop}
    \label{prop:hc_beta=-1}
    Let $p_-(-1,L)=q_+(L^\ast)' < p \le \infty$. For any $v_0 \in \DotH^{-1,p}$, there is a global weak solution $v \in C([0,\infty);\scrS')$ to \eqref{e:hc} with initial data $v_0$ so that
    \begin{equation}
    	\label{e:hc_beta=-1_est}
    	\|\nabla v\|_{T^p_{-1/2}} \lesssim \|v\|_{T^p_0} \lesssim \|v_0\|_{\DotH^{-1,p}}.
    \end{equation}
    Moreover, if $p_-(-1,L) < p < p_+(-1,L)$, then $v \in C([0,\infty);\DotH^{-1,p})$.
\end{prop}

\begin{remark}
    For $\beta=-1$, neither $\cR^L_{1/2}$ nor $\cR^{-\Delta}_{1/2}$ is defined on $T^p_{-1/2}$, so Lemma \ref{lemma:cT1/2} fails, and we do not have methods to prove uniqueness or representation for general parabolic equations. {However, it is possible for the heat equation; for example,}  a representation result is presented in \cite[Theorem 3.1]{Auscher-Hou2023_RepHeat} for $v \in T^\infty_0$ with trace $v_0 \in \DotH^{-1,\infty} \simeq \bmo^{-1}$. Similarly, we do not know the converse inequality in \eqref{e:hc_beta=-1_est}.
\end{remark}

Our main strategy here is to define the operator $\cE_L$ on $\DotH^{-1,p}$ by an extension of $(e^{-tL} \Div)_{t>0}$ acting on $\DotH^{0,p}$, since any $\DotH^{-1,p}$-distributions can be written as the divergence of $\DotH^{0,p}$-functions. This fact follows from boundedness of Riesz transforms on $L^p$ when $1<p<\infty$ and on $\bmo$ when $p=\infty$. Here and in the sequel, we omit to specify that the operator applies to $\bC^n$-valued functions.

\begin{lemma}
	\label{lemma:beta=-1_ext_e^-tLdiv}
    Let $q_+(L^\ast)' < p \le \infty$ and $t>0$.
    \begin{enumerate}[label=\normalfont(\roman*)]
        \item \label{item:beta=-1_Dt_bdd}
        The operator $e^{-tL} \Div: W^{1,2} \to W^{1,2}$ extends to a bounded operator from $\DotH^{0,p}$ to $W^{1,2}_{\loc}$, denoted by $\cG_t$.

        \item \label{item:beta=-1_Dt_cv}
        For any $f \in \DotH^{0,p}$, $\cG_t(f)$ converges to $\Div f$ in $\scrS'$ as $t \to 0$.
    \end{enumerate}  
\end{lemma}

\begin{proof}
    The extension in \ref{item:beta=-1_Dt_bdd} is constructed by duality. Let $B$ be a ball in $\bR^n$ and $\phi \in L^2(B)$. When $q_+(L^\ast)'<p \le 2$, $L^2-L^{p'}$ estimates of $(t^{1/2} \nabla e^{-tL^\ast})$ yields
	\[ \|\nabla e^{-tL^\ast} \phi\|_{p'} \lesssim t^{-\frac{1}{2}-\frac{n}{2}[p',2]} \|\phi\|_2. \]
    When $p=\infty$, standard molecular estimates combining $L^2-L^2$ off-diagonal estimates of $(t^{1/2} \nabla e^{-tL^\ast})$ and the fact that $\nabla e^{-tL^\ast} \phi$ has mean value 0 imply      
    \[ \|\nabla e^{-tL^\ast} \phi\|_{H^1} \lesssim (t^{-1/2} |B|^{1/2} + t^{\frac{n}{4}-\frac{1}{2}}) \|\phi\|_2. \]
    By interpolation, we have $\|\nabla e^{-tL^\ast} \phi\|_{\DotH^{0,p'}} \lesssim_{t,B} \|\phi\|_2$ for $q_+(L^\ast)'<p \le \infty$. Thus, for any $t>0$, we define $\cG_t: \DotH^{0,p} \to L^2_{\loc}$ by the pairing
	\[ \langle \cG_t(f),\phi \rangle := \langle f,\nabla e^{-tL^\ast} \phi \rangle, \quad \forall \phi \in L^2_\ssfc(\bR^n). \]
    Meanwhile, note that $(t\nabla e^{-tL^\ast} \Div)$ also has $L^2-L^{p'}$ off-diagonal estimates for $q_+(L^\ast)'<p \le 2$, considering the decomposition
	\begin{equation}
		\label{e:decomp_t_grad_e^-tL_div}
        t\nabla e^{-tL^\ast} \Div = (t^{1/2} \nabla e^{-\frac{t}{2} L^\ast}) (t^{1/2} e^{-\frac{t}{2} L^\ast} \Div).
	\end{equation}
	So repeating the above argument for $(t\nabla e^{-tL^\ast} \Div)$ yields $\nabla \cG_t(f)$ also lies in $L^2_{\loc}$, hence $\cG_t(f) \in W^{1,2}_{\loc}$. This proves \ref{item:beta=-1_Dt_bdd}.

    Next, we proceed with \ref{item:beta=-1_Dt_cv}. Pick $\phi \in \scrS$. It suffices to prove $\nabla e^{-tL^\ast} \phi$ tends to $\nabla \phi$ in $\DotH^{0,p'}$ when $t \to 0$. For $q_+(L^\ast)'<p \le 2$, as $2 \le p' < q_+(L^\ast) = p_+(1,L^\ast)$, it follows by continuity of the semigroup $(e^{-tL^\ast})$ on $H^{1,p'}$, see Lemma \ref{lemma:EL_ext_semigp}. For $p=\infty$, we first assert the following
    \begin{lemma}
        \label{lemma:nabla_etL*_psi-H1}
        Let $\psi \in \Cc(\bR^n)$ and $B$ be a ball in $\bR^n$ with $\supp(\psi) \subset B$. Then
        \[ \|\nabla e^{-tL^\ast}\psi -\nabla \psi\|_{H^1} \lesssim |B|^{1/2} \|\nabla e^{-tL^\ast} \psi - \nabla \psi\|_2 + t^{n/4} \|\nabla \psi\|_2. \]
    \end{lemma}
    The proof is provided right below. Admitting it, let us show that for any $\phi \in \scrS$, $\nabla e^{-tL^\ast} \phi$ tends to $\nabla \phi$ in $H^1$ as $t \to 0$. Indeed, write $B:=B(0,1)$, $C_0:=4B$, and $C_j:=2^{j+2}B \setminus 2^{j-2}B$ for $j \ge 1$. Let $(\chi_j)$ be a smooth partition of unity so that $\supp(\chi_j) \subset C_j$ with $\|\chi_j\|_\infty \le 1$ and $\|\nabla \chi_j\|_\infty \lesssim 1$. Define $\phi_j:=\phi \chi_j$. Note that
    \begin{equation}
        \label{e:beta=-1_nabla-phi_j-L2}
        \begin{aligned}
        \|\nabla \phi_j\|_2 
        &\lesssim \left( \int_{C_j} \langle y \rangle^{-2N} |\langle y \rangle^{N} \nabla \phi(y)|^2 dy \right)^{1/2} + \left( \int_{C_j} \langle y \rangle^{-2N} |\langle y \rangle^{N} \phi(y)|^2 dy \right)^{1/2} \\
        &\lesssim 2^{-j(N-\frac{n}{2})} \cP_{N+1}(\phi),
        \end{aligned}
    \end{equation}
    where $\cP_{N+1}$ is the semi-norm on $\scrS$ defined in \eqref{e:cPN_S(Rn)}. Then applying Lemma \ref{lemma:nabla_etL*_psi-H1} on $\psi=\phi_j$, we get
    \begin{equation}
        \label{e:beta=-1_nabla-e-tL*-phi_cv_H1}
        \begin{aligned}
            \|\nabla e^{-tL^\ast} \phi-\nabla \phi\|_{H^1} 
            &\lesssim \sum_{j \ge 0} \|\nabla e^{-tL^\ast} \phi_j-\nabla \phi_j\|_{H^1} \\
            &\lesssim \sum_{j \ge 0} 2^{jn/2} \|\nabla e^{-tL^\ast} \phi_j-\nabla \phi_j\|_2 + t^{n/4} \|\nabla \phi_j\|_2.
        \end{aligned}
    \end{equation}
    Using \eqref{e:beta=-1_nabla-phi_j-L2}, we have 
    \[ \sum_{j \ge 0} 2^{jn/2} \|\nabla e^{-tL^\ast} \phi_j-\nabla \phi_j\|_2 + t^{n/4} \|\nabla \phi_j\|_2 \lesssim \sum_{j \ge 0} 2^{-j(N-n)} \cP_{N+1}(\phi), \]
    which converges when $N>n$. Therefore, Lebesgue's dominated convergence theorem implies the right-hand side in \eqref{e:beta=-1_nabla-e-tL*-phi_cv_H1} tends to 0 as $t \to 0$, noting that the first term tends to 0 by continuity of $(e^{-tL^\ast})$ on $\DotH^{1,2}$. This completes the proof.
\end{proof}

\begin{proof}[Proof of Lemma \ref{lemma:nabla_etL*_psi-H1}]
    The conservation property of the semigroup yields $\nabla e^{-tL^\ast} \psi = \nabla e^{-tL^\ast} (\psi-\fint_B \psi)$, so the desired inequality again follows from applying the molecular decomposition and Poincar\'e's inequality.
\end{proof}

Let us finish by proving Proposition \ref{prop:hc_beta=-1}.

\begin{proof}[Proof of Proposition \ref{prop:hc_beta=-1}]
Let $q_+(L^\ast)' < p \le \infty$. Define the operator $\cG: \DotH^{0,p} \to L^2_{\loc}((0,\infty);W^{1,2}_{\loc})$ by
\begin{equation}
    \label{e:cG}
	\cG(f)(t):=\cG_t(f).
\end{equation}
We first show that $v:=\cG(f)$ satisfies
\begin{equation}
    \label{e:beta=-1_v_tent-f}
    \|v\|_{T^p_0} \lesssim \|f\|_{\DotH^{0,p}}
\end{equation}
When $q_+(L^\ast)'<p \le 2$, we know from \cite[\S 8.2]{Auscher-Egert2023OpAdp} that $\psi(z):=z^{1/2} e^{-z}$ is an admissible function for $s \le 0$ and $p \le 2$. We hence infer from $L^p$-boundedness of $L^{-1/2}\Div$ (see \textit{e.g.}, \cite[Theorem 4.1]{Auscher2007Memoire}) that 
\begin{equation}
    \label{e:beta=-1_cG-tent}
	\|\cG(f)\|_{T^p_0} = \|(tL)^{1/2} e^{-tL} L^{-1/2} \Div f\|_{T^p_{1/2}} \eqsim \|L^{-1/2} \Div f\|_p \lesssim \|f\|_p.
\end{equation}
Next, when $p=\infty$, we use a classical argument on the relation between Carleson measures and $\bmo$-functions. Let $f$ be in $\DotH^{0,\infty} \simeq \bmo$. Recall that
\[ \|\cG(f)\|_{T^\infty_0} = \sup_B \left( \frac{1}{|B|} \int_0^{r(B)^2} \int_B |\cG_t(f)(y)|^2 dtdy \right)^{1/2}. \]
Let $B$ be a ball in $\bR^n$. Write $C_0:=2B$ and $C_j:=2^{j+1} B \setminus 2^j B$ for any $j \ge 1$. Consider the decomposition
\[ f = \langle f \rangle_{C_0} + \sum_{j \ge 0} (f-\langle f \rangle_{C_0}) \I_{C_j} =: \langle f \rangle_{C_0} + \sum_{j \ge 0} f_j \quad \text{ in } \scrD', \]
where $\langle f \rangle_{C_0}:=\fint_{C_0} f$. As $f_j \in L^2$, by definition, one finds that for any $t>0$, $\cG(f)(t) = \sum_{j \ge 0} \cG(f_j)(t) = \sum_{j \ge 0} e^{-tL} \Div f_j$ in $\scrD'$, so we get
\[ \|\cG(f)\|_{L^2((0,r(B)^2) \times B)} \lesssim \sum_{j \ge 0} \|e^{-tL} \Div f_j\|_{L^2((0,r(B)^2) \times B)} =: I_j. \]
For $I_0$, this follows from \eqref{e:beta=-1_cG-tent} as
\[ I_0 \le \|t^{1/2} e^{-tL} \Div f_0\|_{L^2_{1/2}(\bR^{1+n}_+)} \lesssim \|f_0\|_2 = \|f-\langle f \rangle_{C_0}\|_{L^2(C_0)} \lesssim |B|^{1/2} \|f\|_{\bmo}. \]
For $j \ge 1$, we infer from $L^2-L^2$ off-diagonal estimates of $(t^{1/2} e^{-tL} \Div)$ that
\begin{align*}
	I_j
	&\lesssim \left( \int_0^{r(B)^2} e^{-2c\frac{(2^j r(B))^2}{t}} \frac{dt}{t} \int_{C_j} |f-\langle f \rangle_{C_0}|^2 \right)^{1/2} \\
    &\lesssim \left( \int_0^{r(B)^2} |2^{j+1} B| e^{-2c\frac{(2^j r(B))^2}{t}} \frac{dt}{t} \fint_{2^{j+1} B} |f-\langle f \rangle_{2^{j+1} B}|^2 \right)^{1/2} \\
    &\quad + \left( \int_0^{r(B)^2} |2^{j+1} B| e^{-2c\frac{(2^j r(B))^2}{t}} \frac{dt}{t} \right)^{1/2} |\langle f \rangle_{2^{j+1} B}-\langle f \rangle_{2B}| \\
    &\lesssim 2^{2jn} \log(1+j) e^{-c2^{2j}} |B|^{1/2} \|f\|_{\bmo}.
\end{align*}
Gathering the estimates, we obtain $\|\cG(f)\|_{L^2((0,r(B)^2) \times B)} \lesssim |B|^{1/2} \|f\|_{\bmo}$. Note that the controlling constant is independent of the ball $B$, so by taking supremum over all  balls $B$ in $\bR^n$, we have $\|\cG(f)\|_{T^\infty_0} \lesssim \|f\|_{\bmo}$ as desired.  
The rest for $2\le p<\infty$ hence follows from interpolation. This proves \eqref{e:beta=-1_v_tent-f}. 
  
Meanwhile, observe that when $f \in \scrS_{\infty}$, $v$ is clearly a global weak solution to $\partial_{t} v -\Div A \nabla v=0$ on $\bR^{1+n}_{+}$, so by  Corollary~\ref{cor:energyinequalitytent space}, we also have $\|\nabla v \|_{T^p_{-1/2}} \lesssim  \| v \|_{T^p_{0}}$.

Now, for any $v_0 \in \DotH^{-1,p}$, pick $V_0 \in \DotH^{0,p}$ so that $v_0 = \Div V_0$ with $\|v_0\|_{\DotH^{-1,p}} \eqsim \|V_0\|_{\DotH^{0,p}}$. Using the above bounds for $f=V_0$, we obtain existence of the weak solution and the estimates \eqref{e:hc_beta=-1_est} by a standard density argument.
	
To prove continuity, recall that Lemma \ref{lemma:beta=-1_ext_e^-tLdiv} \ref{item:beta=-1_Dt_cv} says $v(t)$ converges to $v_0$ in $\scrS'$ as $t \to 0$. In fact, a similar argument yields $v \in C((0,\infty);\scrS')$. When $p_-(-1,L) < p < p_+(-1,L)$, the desired stronger continuity follows from Lemma \ref{lemma:EL_ext_semigp}, due to uniqueness of extensions by density. This completes the proof.	
\end{proof}
\appendix

\section{Proof of Proposition \ref{prop:Ladp-id}}
\label{sec:ap_pf_prop_Ladp-id}

\subsection{Proof of Proposition \ref{prop:Ladp-id} \ref{item:Ladp-id}}
Let us first establish an appropriate atomic decomposition for distributions in $\DotH^{1,p}$ with $\frac{n}{n+2} <p \le \frac{n}{n+1}$. For $\frac{n}{n+1}<p \le 1$, it was done by \cite[Proposition 8.31]{Auscher-Egert2023OpAdp}.

\begin{definition}[$\DotH^{1,p}$-atom]
    Let $\frac{n}{n+2} <p \le \frac{n}{n+1}$. A function $a\in L^2$ is called an \textit{$\DotH^{1,p}$-atom} if
    \begin{enumerate}[label=(\roman*)]
        \item \label{item:H1p-atom-supp}
        there is a ball $B \subset \bR^n$ so that $\supp(a) \subset B$, which is called \textit{associated} to $a$;
        \item \label{item:H1p-atom-mean}
        $a$ is of mean zero, \textit{i.e.}, $\int_{\bR^n} a = 0$;
        \item \label{item:H1p-atom-L2-norm}
        $\|\nabla a\|_2 \le |B|^{[2,p]}$.
    \end{enumerate}
\end{definition}
Compared with the case for $\frac{n}{n+1}<p \le 1$ (see \cite[Definition 8.30]{Auscher-Egert2023OpAdp}), the only extra condition we impose is \ref{item:H1p-atom-mean}.

\begin{lemma}[Atomic decomposition for $\DotH^{1,p}$]
    Let $\frac{n}{n+2} <p \le \frac{n}{n+1}$. For any $f \in \DotH^{1,p} \cap W^{1,2}$, there exist $(\lambda_i) \in \ell^p$ and $\DotH^{1,p}$-atoms $a_i$ so that
    \[ f= \sum_{i \ge 1} \lambda_i a_i, \]
    where the convergence holds in $\DotW^{1,2}$ (that is, with respect to the semi-norm $\|\nabla \cdot\|_{2}$). Moreover, the estimate holds as
    \[ \|f\|_{\DotH^{1,p}} \lesssim \|(\lambda_i)\|_{\ell^p}. \]
\end{lemma}

\begin{proof}
    The proof is analogous to that of \cite[Proposition 8.31]{Auscher-Egert2023OpAdp}. We only give a sketch here. Let $D$ be the \textit{Dirac operator}, which is defined on $L^2(\bR^n;\bC^{1+n})$ by
    \[ D := \begin{bmatrix}
        0 & \Div\\ 
        -\nabla & 0
    \end{bmatrix}. \] 
    The main strategy of their proof is to use  the atomic decomposition of $D$-adapted Hardy space $\bH^p_D$ $(0<p<\infty)$ to the vector $g:=D[f,0]^T=-[0,\nabla f]^T$. The only difference is that we decompose $g$ with  $(\bH^p_D,1,2)$-molecules instead of $(\bH^p_D,1,1)$-molecules, which is allowed by \cite[Theorem 8.17]{Auscher-Egert2023OpAdp}. The same localization arguments apply to obtain an $L^2$-convergent  decomposition  $\nabla f=\sum_{i \ge 1} \lambda_i \nabla a_i$, with $\lambda_{i},a_{i}$ as required.
\end{proof}

For any $p \in (0,\infty]$, denote by $p_\ast$ the lower Sobolev conjugate of $p$, \textit{i.e.}, $1/p_\ast := 1/p+1/n$. For any $p \in (0,n)$, denote by $p^\ast$ the upper Sobolev conjugate of $p$, \textit{i.e.}, $1/p^\ast := 1/p-1/n$.

\begin{proof}[Proof of Proposition \ref{prop:Ladp-id} \ref{item:Ladp-id}]
    We only need to show $\bH^{s,p}_L$ agrees with $\DotH^{s,p} \cap L^2$ with equivalent (quasi-)norms
    \[ \|f\|_{\bH^{s,p}_L} \eqsim \|f\|_{\DotH^{s,p}}, \]
    when $-1\le s\le 1$ and $p_{-}(s,L)<p<p_{+}(s,L)$. First observe that by interpolation of $\bH^{s,p}_L$ (see \cite[Theorem 4.28]{Amenta-Auscher2016EllipticBVP}), it suffices to consider three cases, $s=0$ and $s=\pm 1$. 
    
    For $s=0$, this is \cite[Theorem 9.7]{Auscher-Egert2023OpAdp}, as $p_{\pm}(0,L)=p_{\pm}(L)$.

    For $s=1$, \cite[Theorem 9.7]{Auscher-Egert2023OpAdp} shows that it holds when $(\max\{p_-(L),1\})_\ast < p < q_+(L) =p_{+}(1,L)$, so it remains to consider the case when $\frac{n}{n+1} \le p_-(L) < 1$ and $p_{-}(1,L)=(p_-(L))_\ast < p \le 1_\ast$. The method of proof is adapted from the argument for \cite[\S 9.2, Part 5]{Auscher-Egert2023OpAdp}, and we only need to show that for any $\psi \in \Psi_\infty^\infty$, there is a constant $C>0$ so that for any $\DotH^{1,p}$-atom $a$,
    \begin{equation}
        \label{e:H1p-id_goal_psi(tL)a_bdd}
        \|\psi(tL) a\|_{T^p_1} \le C.
    \end{equation}
    
    Let us verify it.  Let $\psi$ be in $\Psi_\infty^\infty$, $a$ be an $\DotH^{1,p}$-atom, and $B$ be a ball of $\bR^n$ associated to $a$. Note that
    \[ \|\psi(tL) a\|_{T^p_1} \lesssim \|\cA(t^{-1} \psi(tL) a)\|_{L^p(4B)} + \|\cA(t^{-1} \psi(tL) a)\|_{L^p((4B)^c)} =: I_1+I_2, \]
    where $\cA$ is the conical square function defined by
    \[ \cA(f)(x):= \left( \int_0^\infty \fint_{B(x,t^{1/2})} |f(t,y)|^2 dtdy \right)^{1/2}. \]
    The boundedness of $I_1$ is exactly as for \cite[(9.30)]{Auscher-Egert2023OpAdp}. Let us concentrate on $I_2$. To this end, we fix $(t,x) \in \bR^{1+n}_+$ with $x \in (4B)^c$. As $a$ is of mean zero, Poincar\'e's inequality yields
    \begin{equation}
        \label{e:H1p-id_a_L2}
        \|a\|_2 \lesssim r(B) \|\nabla a\|_2 \lesssim |B|^{[2,p^\ast]}.
    \end{equation}
    This estimate, together with the support condition and mean value 0, shows that $a$ belongs to $H^q$ for $\frac{n}{n+1}<q \le 1$ with the estimate
    \begin{equation}
        \label{e:H1p-id_a_Hq}
        \|a\|_{H^q} \lesssim |B|^{[q,p^\ast]}.
    \end{equation}
    Then we fix $p_-(L) < q < p^\ast$. The $H^q-L^2$ boundedness of $(\psi(tL))_{t>0}$ (see \cite[Lemma 4.4]{Auscher-Egert2023OpAdp}) yields
    \begin{equation}
        \label{e:H1p-id_psiL2_Hq-L2bdd}
        \|\psi(tL) a\|_{2} \lesssim t^{\frac{n}{2}[2,q]} \|a\|_{H^q}.
    \end{equation}
    Meanwhile, using  $L^2-L^2$ off-diagonal estimates,
    \footnote{Here we use polynomial-order decay as in \cite[Lemma 4.16]{Auscher-Egert2023OpAdp}.}
    we get that for any $M>0$,
    \begin{equation}
        \label{e:H1p-id_psiL2_L2-L2ode}
        \|\psi(tL) a\|_{L^2(B(x,t^{1/2}))} \lesssim \left ( 1+\frac{\dist(B(x,t^{1/2}),B)^2}{t} \right )^{-M} \|a\|_2.
    \end{equation}
    As $q<p^\ast<2$, \textit{i.e.}, $n/q+1>n/p>n/2$, we pick $\theta \in (0,1)$ so that
    \begin{equation}
        \label{e:H1p-id_theta}
        \rho:=\frac{n}{q}+1-n\theta [q,2] > \frac{n}{p}.
    \end{equation}
    Note that $\theta$ can be chosen arbitrarily close to 0. We interpolate \eqref{e:H1p-id_psiL2_Hq-L2bdd} and \eqref{e:H1p-id_psiL2_L2-L2ode} with respect to $\theta$ to get
    \begin{align*}
        &\|\psi(tL) a\|_{L^2(B(x,t^{1/2}))} \\
        &\quad \lesssim t^{\frac{n}{2}[2,q](1-\theta)} \left ( 1+\frac{\dist(B(x,t^{1/2}),B)^2}{t} \right )^{-M\theta}  \|a\|_{H^q}^{1-\theta} \|a\|_2^\theta.
    \end{align*}
    Taking $M>\frac{\rho}{2\theta}$, we integrate it over $t$ (with right weight and norm) to get
    \[ \cA(t^{-1} \psi(tL) a)(x) \lesssim \dist(x,B)^{-\rho} \|a\|_{H^q}^{1-\theta} \|a\|_2^\theta. \]
    Gathering \eqref{e:H1p-id_a_L2}, \eqref{e:H1p-id_a_Hq}, and definition of $\rho>n/p$ (cf. \eqref{e:H1p-id_theta}), we obtain
    \begin{align*}
        I_2 
        &= \|\cA(t^{-1} \psi(tL) a)\|_{L^p((4B)^c)} \lesssim r(B)^{\frac{n}{p}-\rho} \|a\|_{H^q}^{1-\theta} \|a\|_2^\theta \\
        &\lesssim |B|^{\frac{1}{p}-\frac{1}{n}-\frac{1}{q}+\theta[q,2]} |B|^{(1-\theta)[q,p^\ast]} |B|^{\theta [2,p^\ast]} \lesssim 1.
    \end{align*}
    This proves \eqref{e:H1p-id_goal_psi(tL)a_bdd} and hence concludes the case $s=1$.
    
    For $s=-1$, we use duality. As $1 \le p_-(-1,L)<p<p_+(-1,L) \le \infty$, we have $\max\{p_-(1,L^\ast),1\} < p' < p_+(1,L^\ast)=q_{+}(L^*)$. Then for any $f \in \bH^{-1,p}_L$, we have
    \[ \|f\|_{\DotH^{-1,p}} = \sup_{g \in \DotH^{1,p'}\cap L^2(\bR^n)} \frac{|\langle f,g \rangle|}{\|g\|_{\DotH^{1,p'}}} \eqsim \sup_{ g \in \bH^{1,p'}_{L^\ast} } \frac{ |\langle f,g \rangle| }{ \|g\|_{ \bH^{1,p'}_{L^\ast} } } \eqsim \|f\|_{\bH^{-1,p}_L}. \]
    The equality follows by density of $\DotH^{1,p'} \cap L^2$ in $\DotH^{1,p'}$ and the first equivalence by $\bH^{1,p'}_{L^\ast}=\DotH^{1,p'} \cap L^2$ with $\|g\|_{\bH^{1,p'}_{L^\ast}} \eqsim \|g\|_{\DotH^{1,p'}}$ (from Case $s=1$ for $L^*$). The last equivalence holds by \cite[Proposition 8.9]{Auscher-Egert2023OpAdp}. This completes the proof.
\end{proof}

\subsection{Proof of Proposition \ref{prop:Ladp-id} \ref{item:Ladp-embed}}
Recall that we assume $p_{-}(L)<1$, which implies $p_{+}(L^*)=\infty$. When $1<p<2$ and $-1 \le s\le 0$, by duality, it is equivalent to prove    $ \DotH^{-s,p'}\cap L^2 \subset \bH^{-s,p'}_{L^*}$ with $\|f\|_{\bH^{-s,p'}_{L^*}} \lesssim \|f\|_{\DotH^{-s,p'}}$. 
For $s=0$ and $s=-1$, this is explicitly proved in Part 2 and Part 9 of the proof of \cite[Theorem 9.7]{Auscher-Egert2023OpAdp}. Interpolation concludes the argument.
    
It remains to show that the extension of the identity map  to the closure of $\bH^{s,p}_L$ (for its norm) is injective  in the restricted range described in the statement. More precisely, let $(f_j)$ be a Cauchy sequence in $\bH^{s,p}_L$ that tends to 0 in $\DotH^{s,p}$. Our goal is to show $\|f_j\|_{\bH^{s,p}_L}$ tends to 0.
    
To this end, we use duality. Pick $\phi \in \Cc(\bR^{1+n}_+)$ and define $\Phi:=\int_0^\infty e^{-tL^\ast} \phi(t) \frac{dt}{t}$. Fubini's theorem ensures 
\[ \int_0^\infty \int_{\bR^n} (e^{-tL} f_j)(y) \overline{\phi}(t,y) dy\frac{dt}{t} = \int_{\bR^n} f_j(y) \overline{\Phi(y)} dy.  \]
We claim $\Phi \in \DotH^{-s,p'}$. If so, then using the fact that $(f_j)$ tends to 0 in $\DotH^{s,p}$, we get $(\cE_L(f_j))$ tends to 0 in $\scrD'(\bR^{1+n}_+)$ by arbitrariness of $\phi$. Moreover, as $(f_j)$ is a Cauchy sequence in $\bH^{s,p}_L$, so is $(\cE_L(f_j))$ in $T^p_{(s+1)/2}$, because the exponential function $e^{-z}$ is an admissible function for $\bH^{s,p}_L$ when $s<0$ and $p\le 2$. Hence, it tends to 0 in $T^p_{(s+1)/2}$, which implies $\lim_{j \to \infty} \|f_j\|_{\bH^{s,p}_L} = 0$.

We finish by verifying the claim. Pick $0<a<b<\infty$ so that $\supp(\phi) \subset (a,b) \times \bR^n$. Note that $\phi \in \Cc(\bR^{1+n}_+)$ implies $\Phi \in L^2$, so we apply Theorem \ref{thm:RW_sol} to $\Delta \Phi$ to get
\begin{equation}
    \label{e:DotHsp_L_embed_Phi_norm}
    \|\Phi\|_{\DotH^{-s,p'}} \eqsim \| \tau \Delta e^{\tau \Delta} \Phi\|_{T^{p'}_{(-s+1)/2}}.
\end{equation}
Pick $0<a<b<\infty$ so that $\supp(\phi) \subset (a,b) \times \bR^n$. Using H\"older's inequality and Minkowski's inequality, we have  (the implicit constants may depend on $a,b$)
\begin{align*}
    &\| \tau \Delta e^{\tau \Delta} \Phi\|_{T^{p'}_{(-s+1)/2}} \\
    &\ \lesssim \left ( \int_{\bR^n} \left ( \int_0^\infty \fint_{B(x,\tau^{1/2})} \int_a^b |\tau^{\frac{s-1}{2}} \tau \Delta e^{\tau \Delta}e^{-tL^\ast} \phi(t)|^2 \, dtdyd\tau \right )^{p'/2} dx \right )^{1/p'} \\
    &\ \lesssim \left ( \int_a^b \int_0^\infty \left (  \int_{\bR^n} \left( \fint_{B(x,\tau^{1/2})} |\tau^{\frac{s-1}{2}} \tau \Delta e^{\tau \Delta} e^{-tL^\ast}\phi(t)|^2 \, dy \right)^{p'/2} dx \right )^{2/p'} \, d\tau dt\right )^{1/2} \\
    &\ \lesssim \left ( \int_a^b \int_0^\infty \|\tau^{\frac{s-1}{2}} \tau \Delta e^{\tau \Delta}e^{-tL^\ast} \phi(t)\|_{p'}^2 \, d\tau dt\right )^{1/2}
\end{align*}
    
When $\tau$ is large, uniform $L^{p'}$-boundedness of $(\tau \Delta e^{\tau \Delta})$ and $(e^{-tL^\ast})$ implies
\begin{equation*}
    \label{e:DotBs_ppL_embed_Phi_kernel_Lp'<1}
    \|\tau \Delta e^{\tau \Delta} e^{-tL^\ast} \phi(t)\|_{p'} \lesssim \|\phi(t)\|_{p'}.
\end{equation*}
As $s<0$, $s-1<-1$, so the convergence of the integral when $\tau\ge 1$ is ensured.
    
When $\tau$ is small, we need different methods to gain a positive power of $\tau$. For $2 \le q < q_+(L^\ast)'$, we use $L^{q'}$-boundedness of $(e^{\tau \Delta} \Div)$ and $(t^{{1}/{2}} \nabla e^{-tL^\ast})$ for the decomposition to get
\begin{equation}
    \label{e:DotHsp_L_embed_Phi_kernel_Lq-est}
    \begin{aligned}
        \|\tau \Delta e^{\tau \Delta} e^{-tL^\ast} \phi(t)\|_q 
        &= \tau^{{1}/{2}} t^{-{1}/{2}} \left\| \left( \tau^{{1}/{2}} e^{\tau \Delta} \Div \right) \left( t^{{1}/{2}} \nabla e^{-tL^\ast} \phi(t) \right) \right\|_q \\
        &\lesssim \tau^{1/2} t^{-1/2} \|\phi(t)\|_q.
    \end{aligned}
\end{equation}
Meanwhile, let $\alpha \in (0,1)$ be a parameter to be determined with the constraint
\[ 0<\alpha<n\left( \frac{1}{p_-(L)}-1 \right). \]
Denote by $\Dot{\Lambda}^\alpha$ the homogeneous $\alpha$-H\"older space. One can also use Gaussian decay of the kernel of $(\tau \Delta e^{\tau \Delta})$ and uniform $\Dot{\Lambda}^\alpha$-boundedness of $(e^{-tL^\ast})$ (by duality from $H^p$-boundedness of $(e^{-tL})$) to get
\begin{equation}
    \label{e:DotHsp_L_embed_Phi_kernel_Linfty-est}
    \|\tau \Delta e^{\tau \Delta} e^{-tL^\ast} \phi(t)\|_\infty \lesssim \tau^{\alpha/2} \|e^{-tL^\ast} \phi(t)\|_{\Dot{\Lambda}^\alpha}\lesssim \tau^{\alpha/2} \|\phi(t)\|_{\Dot{\Lambda}^\alpha}.
\end{equation}
Interpolating \eqref{e:DotHsp_L_embed_Phi_kernel_Lq-est} and \eqref{e:DotHsp_L_embed_Phi_kernel_Linfty-est} yields
\begin{equation}
    \label{e:DotBs_ppL_embed_Phi_kernel_Lp'<tau-t}
    \|\tau \Delta e^{\tau \Delta} e^{-tL^\ast} \phi(t)\|_{p'} \lesssim_\phi \tau^{\frac{\alpha}{2}(1-\frac{q}{p'})+\frac{q}{2p'}}.
\end{equation}

Observe that when $s>-n( \frac{1}{p_-(L)}-1 ) - \frac{q_+(L^\ast)}{p'} ( 1-n( \frac{1}{p_-(L)}-1 ) )$, \textit{i.e.}, \eqref{e:DotHsp_L_embed_condition} is satisfied, there exist $\alpha \in (0,1)$ and $q \in [2,q_+(L^\ast))$ so that 
\[ \begin{cases}
0<\alpha<n(\frac{1}{p_-(L)}-1) \\ 
s-1+\alpha(1-\frac{q}{p'})+\frac{q}{p'}>-1
\end{cases}. \]
So the convergence of the integral for small $\tau$ follows. We thus obtain $\Phi \in \DotH^{-s,p'}$ from \eqref{e:DotHsp_L_embed_Phi_norm}. This completes the proof.

\subsection*{Copyright}
A CC-BY 4.0 \url{https://creativecommons.
org/licenses/by/4.0/} \\public copyright license has been applied by the authors to the present document and will be applied to all subsequent versions up to the Author Accepted Manuscript arising from this submission.

\bibliographystyle{alpha}
\bibliography{sample}

\newcommand{\etalchar}[1]{$^{#1}$}
\begin{thebibliography}{HvNVW17}

\bibitem[AA11]{Auscher-Axelsson2011_MR_L2beta}
P.~Auscher and A.~Axelsson.
\newblock Remarks on maximal regularity.
\newblock In {\em Parabolic problems}, volume~80 of {\em Progr. Nonlinear
  Differential Equations Appl.}, pages 45--55. Birkh\"{a}user/Springer Basel
  AG, Basel, 2011.

\bibitem[AA18]{Amenta-Auscher2016EllipticBVP}
A.~Amenta and P.~Auscher.
\newblock {\em Elliptic boundary value problems with fractional regularity
  data}, volume~37 of {\em CRM Monograph Series}.
\newblock American Mathematical Society, Providence, RI, 2018.
\newblock The first order approach.

\bibitem[AE23]{Auscher-Egert2023OpAdp}
P.~Auscher and M.~Egert.
\newblock {\em Boundary value problems and {H}ardy spaces for elliptic systems
  with block structure}, volume 346 of {\em Progress in Mathematics}.
\newblock Birkh\"{a}user/Springer, Cham, 2023.

\bibitem[AF17]{Auscher-Frey2017_NS-KT}
P.~Auscher and D.~Frey.
\newblock On the well-posedness of parabolic equations of {N}avier-{S}tokes
  type with {$BMO^{-1}$} data.
\newblock {\em J. Inst. Math. Jussieu}, 16(5):947--985, 2017.

\bibitem[AH23a]{Auscher-Hou2023_RepHeat}
P.~Auscher and H.~Hou.
\newblock On representation of solutions to the heat equation, 2023.
\newblock \href{https://doi.org/10.48550/arXiv.2310.19330}{arXiv:2310.19330}.

\bibitem[AH23b]{Auscher-Hou2023_SIOTent}
P.~Auscher and H.~Hou.
\newblock On well-posedness and maximal regularity for parabolic {C}auchy
  problems on weighted tent spaces, 2023.
\newblock \href{https://doi.org/10.48550/arXiv.2311.04844}{arXiv:2311.04844}.

\bibitem[AHL{\etalchar{+}}02]{Auscher-Hofmann-Lacey-McIntosh-Tchamitchian2002_Kato}
P.~Auscher, S.~Hofmann, M.~Lacey, A.~McIntosh, and P.~Tchamitchian.
\newblock The solution of the {K}ato square root problem for second order
  elliptic operators on {${\bR}^n$}.
\newblock {\em Ann. of Math. (2)}, 156(2):633--654, 2002.

\bibitem[AM19]{Auscher-Mourgoglou2019_Ep_delta}
P.~Auscher and M.~Mourgoglou.
\newblock Representation and uniqueness for boundary value elliptic problems
  via first order systems.
\newblock {\em Rev. Mat. Iberoam.}, 35(1):241--315, 2019.

\bibitem[Ame18]{Amenta2018_WeightedTent}
A.~Amenta.
\newblock Interpolation and embeddings of weighted tent spaces.
\newblock {\em J. Fourier Anal. Appl.}, 24(1):108--140, 2018.

\bibitem[AMP19]{Auscher-Monniaux-Portal2019Lp}
P.~Auscher, S.~Monniaux, and P.~Portal.
\newblock On existence and uniqueness for non-autonomous parabolic {C}auchy
  problems with rough coefficients.
\newblock {\em Ann. Sc. Norm. Super. Pisa Cl. Sci. (5)}, 19(2):387--471, 2019.

\bibitem[AP23]{Auscher-Portal2023Lions}
P.~Auscher and P.~Portal.
\newblock Stochastic and deterministic parabolic equations with bounded
  measurable coefficients in space and time: well-posedness and maximal
  regularity, 2023.
\newblock \href{https://doi.org/10.48550/arXiv.2310.06460}{arXiv:2310.06460}.

\bibitem[Aus07]{Auscher2007Memoire}
P.~Auscher.
\newblock On necessary and sufficient conditions for {$L^p$}-estimates of
  {R}iesz transforms associated to elliptic operators on {$\bR^n$} and related
  estimates.
\newblock {\em Mem. Amer. Math. Soc.}, 186(871):xviii+75, 2007.

\bibitem[Aus11]{Auscher2011_Change_angle}
P.~Auscher.
\newblock Change of angle in tent spaces.
\newblock {\em C. R. Math. Acad. Sci. Paris}, 349(5-6):297--301, 2011.

\bibitem[AvNP14]{Auscher-vanNeerven-Portal2014_SPDE-tent}
P.~Auscher, J.~van Neerven, and P.~Portal.
\newblock Conical stochastic maximal {$L^p$}-regularity for {$1\leqslant
  p<\infty$}.
\newblock {\em Math. Ann.}, 359(3-4):863--889, 2014.

\bibitem[BCD11]{Bahouri-Chemin-Danchin2011_FA}
H.~Bahouri, J.-Y. Chemin, and R.~Danchin.
\newblock {\em Fourier analysis and nonlinear partial differential equations},
  volume 343 of {\em Grundlehren der mathematischen Wissenschaften [Fundamental
  Principles of Mathematical Sciences]}.
\newblock Springer, Heidelberg, 2011.

\bibitem[BM16]{Barton-Mayboroda2016_Zspaces}
A.~Barton and S.~Mayboroda.
\newblock Layer potentials and boundary-value problems for second order
  elliptic operators with data in {B}esov spaces.
\newblock {\em Mem. Amer. Math. Soc.}, 243(1149):v+110, 2016.

\bibitem[Bou13]{Bourdaud2013_Realization}
G.~Bourdaud.
\newblock Realizations of homogeneous {B}esov and {L}izorkin-{T}riebel spaces.
\newblock {\em Math. Nachr.}, 286(5-6):476--491, 2013.

\bibitem[CMS85]{Coifman-Meyer-Stein1985_TentSpaces}
R.~R. Coifman, Y.~Meyer, and E.~M. Stein.
\newblock Some new function spaces and their applications to harmonic analysis.
\newblock {\em J. Funct. Anal.}, 62(2):304--335, 1985.

\bibitem[dS64]{deSimon1964_MRL2}
L.~de~Simon.
\newblock Un'applicazione della teoria degli integrali singolari allo studio
  delle equazioni differenziali lineari astratte del primo ordine.
\newblock {\em Rend. Sem. Mat. Univ. Padova}, 34:205--223, 1964.

\bibitem[DV23]{Danchin-Vasilyev2023_InhomoNS-tent}
R.~Danchin and I.~Vasilyev.
\newblock Density-dependent incompressible {N}avier--{S}tokes equations in
  critical tent spaces, 2023.
\newblock {\href{https://doi.org/10.48550/arXiv.2305.09027}{arXiv:2305.09027}}.

\bibitem[FS72]{Fefferman-Stein1972Hp}
C.~Fefferman and E.~M. Stein.
\newblock {$H\sp{p}$} spaces of several variables.
\newblock {\em Acta Math.}, 129(3-4):137--193, 1972.

\bibitem[Gra14]{Grafakos2014_ModernFA}
L.~Grafakos.
\newblock {\em Modern {F}ourier analysis}, volume 250 of {\em Graduate Texts in
  Mathematics}.
\newblock Springer, New York, third edition, 2014.

\bibitem[H{\"o}r03]{Hormander2003PDOI}
L.~H{\"o}rmander.
\newblock {\em The analysis of linear partial differential operators. {I}}.
\newblock Classics in Mathematics. Springer-Verlag, Berlin, 2 edition, 2003.
\newblock Distribution theory and Fourier analysis.

\bibitem[HvNVW17]{Hytonen-NVW2017BanachSpaces_II}
T.~Hyt\"{o}nen, J.~van Neerven, M.~Veraar, and L.~Weis.
\newblock {\em Analysis in {B}anach spaces. {V}ol. {II}}, volume~67 of {\em
  Ergebnisse der Mathematik und ihrer Grenzgebiete. 3. Folge / A Series of
  Modern Surveys in Mathematics}.
\newblock Springer, Cham, 2017.
\newblock Probabilistic methods and operator theory.

\bibitem[KM98]{Kalton-Mitrea1998_ComplexInterpolation}
N.~Kalton and M.~Mitrea.
\newblock Stability results on interpolation scales of quasi-{B}anach spaces
  and applications.
\newblock {\em Trans. Amer. Math. Soc.}, 350(10):3903--3922, 1998.

\bibitem[KP93]{Kenig-Pipher1993Xp}
C.~E. Kenig and J.~Pipher.
\newblock The {N}eumann problem for elliptic equations with nonsmooth
  coefficients.
\newblock {\em Invent. Math.}, 113(3):447--509, 1993.

\bibitem[KT01]{Koch-Tataru2001_NSBMO-1}
H.~Koch and D.~Tataru.
\newblock Well-posedness for the {N}avier-{S}tokes equations.
\newblock {\em Adv. Math.}, 157(1):22--35, 2001.

\bibitem[Lio57]{Lions1957_L2}
J.-L. Lions.
\newblock Sur les probl\`emes mixtes pour certains syst\`emes paraboliques dans
  des ouverts non cylindriques.
\newblock {\em Ann. Inst. Fourier (Grenoble)}, 7:143--182, 1957.

\bibitem[Lun95]{Lunardi1995_semigp}
A.~Lunardi.
\newblock {\em Analytic semigroups and optimal regularity in parabolic
  problems}, volume~16 of {\em Progress in Nonlinear Differential Equations and
  their Applications}.
\newblock Birkh\"{a}user Verlag, Basel, 1995.

\bibitem[Mou15]{Moussai2015_Realization_p<1}
M.~Moussai.
\newblock Realizations of homogeneous {B}esov and {T}riebel-{L}izorkin spaces
  and an application to pointwise multipliers.
\newblock {\em Anal. Appl. (Singap.)}, 13(2):149--183, 2015.

\bibitem[Pee76]{Peetre1976_Besov}
J.~Peetre.
\newblock {\em New thoughts on {B}esov spaces}, volume No. 1 of {\em Duke
  University Mathematics Series}.
\newblock Duke University, Mathematics Department, Durham, NC, 1976.

\bibitem[PV19]{Portal-Veraar2019_SPDE-Lp}
P.~Portal and M.~Veraar.
\newblock Stochastic maximal regularity for rough time-dependent problems.
\newblock {\em Stoch. Partial Differ. Equ. Anal. Comput.}, 7(4):541--597, 2019.

\bibitem[Saw18]{Sawano2018_Besov}
Y.~Sawano.
\newblock {\em Theory of {B}esov spaces}, volume~56 of {\em Developments in
  Mathematics}.
\newblock Springer, Singapore, 2018.

\bibitem[Str80]{Strichartz1980_BMOs}
R.~S. Strichartz.
\newblock Bounded mean oscillation and {S}obolev spaces.
\newblock {\em Indiana Univ. Math. J.}, 29(4):539--558, 1980.

\bibitem[Tri83]{Triebel1983Spaces}
H.~Triebel.
\newblock {\em Theory of function spaces}, volume~78 of {\em Monographs in
  Mathematics}.
\newblock Birkh\"{a}user Verlag, Basel, 1983.

\bibitem[Tri20]{Triebel2020_IV}
H.~Triebel.
\newblock {\em Theory of function spaces {IV}}, volume 107 of {\em Monographs
  in Mathematics}.
\newblock Birkh\"{a}user/Springer, Cham, 2020.

\bibitem[Zat20]{Zaton2020wp}
W.~Zato\'{n}.
\newblock Tent space well-posedness for parabolic {C}auchy problems with rough
  coefficients.
\newblock {\em J. Differential Equations}, 269(12):11086--11164, 2020.

\end{thebibliography}

\end{document}